\documentclass[10pt]{amsart}
\usepackage{amsmath, amsrefs, accents, amsfonts, amscd, amssymb, amsthm,
  verbatim, psfrag, dsfont,enumerate,url, color, pdfsync, hyperref} %,parskip}
\usepackage[normalem]{ulem} % For strike through
\usepackage{marginnote}
\usepackage{bbding}
\newtheorem{conjecture}{Conjecture}

\newtheorem{lemma}{Lemma}[section]
\newtheorem{proposition}[lemma]{Proposition}
\newtheorem{theorem}{Theorem}

\newtheorem{definition}[lemma]{Definition}

\newtheorem{question}[lemma]{Question}

\begin{document}
\title[Almost Existence from the Feral Perspective]{Almost Existence from the Feral Perspective and Some Questions}
\author[J.W.~Fish]{Joel W.~Fish}
\thanks{The first author's research in development of this manuscript was
supported in part by NSF-DMS Research Grant Award
1610452}

\author[H. ~Hofer]{Helmut H. W. Hofer}
\date{\today}
\address{
	Joel W.~Fish\\
	Department of Mathematics\\
        University of Massachusetts Boston
	}

\email{joel.fish@umb.edu}
\address{
	Helmut Hofer\\
        School of Mathematics,
	Institute for Advanced Study
	}

\email{hofer@math.ias.edu}

\maketitle

%\begin{abstract}
%\textcolor{red}{To be written}
%\end{abstract}

\tableofcontents
%\allowdisplaybreaks

\newcounter{CurrentSection}
\newcounter{CounterSectionAdiabatic}
\newcounter{CounterPropositionAdiabatic}
\newcounter{CurrentLemma}
\newcounter{CounterLemmaNAME}
\newcounter{CurrentProposition}

\newcounter{CurrentTheorem}
\newcounter{CounterTheoremMain}
\newcounter{CounterSectionMain}

\newcounter{CounterTheoremIntertwine}
\newcounter{CounterSectionIntertwine}

%%%%%%%%%%%%%%%%%%%%%%%%%%%%%%%%%%%%%%%%%%%%%%%%%%%%%%%%%%%%%%%%%%%%%%%%%%%%%%%%
%%%%%                             SECTION                                  %%%%% %%%%%%%%%%%%%%%%%%%%%%%%%%%%%%%%%%%%%%%%%%%%%%%%%%%%%%%%%%%%%%%%%%%%%%%%%%%%%%%%
\section{Introduction and Historical Background}\label{SEC_intro}
This paper is concerned with the ``\textit{almost existence}" phenomenon
  for periodic orbits of Hamiltonian dynamical systems.  
We shall describe some of the background of this  phenomenon,
  and we relate it to the new feral curve theory, \cite{FH2}, which was
  recently initiated by the authors.
The new approach to the  ``\textit{almost existence}" phenomenon suggests
  a larger context which also features questions around the
  \(C^{\infty}\)-closing lemma in the following sense. 
Consider a compact symplectic manifold \((W,\Omega)\) equipped
  with a smooth Hamiltonian \(H\colon W\rightarrow {\mathbb R}\).
Given a regular energy surface $H^{-1}(E)$, one can ask the question:
  Is it the case that  after a small smooth perturbation of $H$ the new
  $H'^{-1}(E)$ has the property that the set of periodic orbits are dense?
This is a classical question, and the $C^1$-closing lemma shows that the
  above assertion is true for a $C^2$-small perturbation of the Hamiltonian,
  i.e. for a $C^1$-close Hamiltonian vector field.  
Whereas the $C^1$-closing lemma holds in general, it is known that the
  $C^{\infty}$-closing lemma does not; see \cite{Herman:1991}.
We believe that the analysis below contributes to a growing body of
  evidence which suggests that the validity of the Hamiltonian
  \(C^\infty\)-closing lemma is intimately connected to the existence of a
  sufficiently rich Gromov-Witten theory of the ambient space.
We touch on this speculation in Section \ref{SEC_speculation}.

The main goal of this paper is to set up a body of results to utilize the
  Feral Curve theory to study questions around almost existence and the
  closing lemma.
Our results show that given a compact pile of Hamiltonian energy surfaces,
  a sufficient supply of pseudoholomorphic curves associated to this pile
  implies ``almost existence"; see Theorem \ref{THM_existence_and_almost_existence}.
This is always attainable provided the pile of energy surfaces can be
  viewed as lying in a compact symplectic manifold with a sufficient supply
  of pseudoholomorphic curves.
As an exercise the reader might enjoy using our more local results to
  reprove the almost existence result for regular compact energy surfaces in
  ${\mathbb R}^{2n}$ by  using the large supply of pseudoholomorphic curves
  obtained as a deformation of holomorphic curves in ${\mathbb C\mathbb
  P}^n$.
Building on previous work in \cite{FH2}, we also show that under suitable
  topological constraints, the almost existence result can be improved to
  the following: Every energy level contains a non-trivial closed invariant
  subset, and for almost all of these energy levels this set is a periodic
  orbit;  see Theorem \ref{THM_existence_and_almost_existence}.
For the convenience of the reader we recall the necessary background.

%%%%%%%%%%%%%%%%%%%%%%%%%%%%%%%%%%%%%%%%%%%%%%%%%%%%%%%%%%%%%%%%%%%%%%%%%%%%%%%%
%%%%%                           SUB-SECTION                                %%%%%
%%%%%%%%%%%%%%%%%%%%%%%%%%%%%%%%%%%%%%%%%%%%%%%%%%%%%%%%%%%%%%%%%%%%%%%%%%%%%%%%
\subsection{The Weinstein Conjecture}
As a starting point for our discussion, we consider two seminal papers by P.
  Rabinowitz, namely \cite{Rab:78} and \cite{Rab}, which are concerned
  with the existence of a periodic orbit on a given compact regular
  energy surface 
  \begin{align*}                                                          %% EQN
    M=H^{-1}(E) 
    \end{align*}
  for an autonomous Hamiltonian system defined on the standard
  phase space ${\mathbb R}^{2n}$.
Rabinowitz showed  the existence of periodic orbits whenever suitable
  geometric conditions are met.
We refer the reader to the introduction of \cite{Abbas-Hofer} for the
  interesting broader historical perspective of Rabinowitz' results.
 
A. Weinstein analyzed these results, particularly \cite{Rab},  and
  proposed the far-reaching Weinstein Conjecture in \cite{Weinstein:79},
  which we shall describe momentarily.
First though, we provide some definitions.  
Consider an odd-dimensional smooth manifold \(M^{2n+1}\),
  with a one-form \(\lambda\) for which \(\lambda\wedge (d\lambda)^n\) is
  a volume form.
In this case, we call \(\lambda\) a contact form for \(M\), which then
  uniquely determines a vector field \(X\) by the following equations:
  \begin{align*}                                                          %% EQN
    \lambda(X)\equiv 1\qquad\text{and}\qquad i_X d\lambda \equiv 0.
    \end{align*}
In this case, we call \(X\) the \textit{Reeb vector field} associated to the
  contact form \(\lambda\).
In a modified form, the Weinstein Conjecture states the following.\\

\noindent{\textbf{Weinstein Conjecture (1978):}} \\
\noindent \emph{
Let \(M\) be a smooth closed odd-dimensional manifold equipped with a
  contact form and an associated Reeb vector field \(X\).
Then the dynamical system given by
  \begin{align*}                                                          %% EQN
    \dot{x}=X(x),
    \end{align*}
  has a nontrivial periodic orbit.\\
}

The first breakthrough concerning the Weinstein Conjecture was Viterbo's
  celebrated result in \cite{Viterbo:1987}, which showed that a regular
  compact contact-type energy surface of a Hamiltonian system in ${\mathbb
  R}^{2n}$ carries a nontrivial periodic orbit.
In 1993 Hofer, in \cite{Hofer:1993}, showed that for a contact form
  $\lambda$ on a closed three-manifold $M$ the Weinstein conjecture is true
  provided at least one of the following holds.
Either $M=S^3$, or $\pi_2(M)\neq 0$, or $\lambda$ is overtwisted.
In 2007 Taubes proved in \cite{Taubes:2007} that the Weinstein conjecture in
  dimension three is true for all $(M,\lambda)$.
  
A natural question then becomes whether or not a generalization of this
  result holds for more general compact energy surfaces in
  \(\mathbb{R}^{2n}\).
However, if $n\geq 3$ results in \cite{Ginzburg:1995} and \cite{Her0} show
  that periodic orbits might not exist on a given energy surface.
There is also a result for $n=2$, see V. Ginzburg and B. G{\"u}rel,
  \cite{GG}, and thus it becomes interesting to study, in some sense, how
  often and how the generalization fails.
This is the content of Section \ref{SEC_almost_existence}.

%%%%%%%%%%%%%%%%%%%%%%%%%%%%%%%%%%%%%%%%%%%%%%%%%%%%%%%%%%%%%%%%%%%%%%%%%%%%%%%%
%%%%%                           SUB-SECTION                                %%%%%
%%%%%%%%%%%%%%%%%%%%%%%%%%%%%%%%%%%%%%%%%%%%%%%%%%%%%%%%%%%%%%%%%%%%%%%%%%%%%%%%
\subsection{Almost Existence}\label{SEC_almost_existence}
By analyzing Viterbo's paper, the second author and E. Zehnder established
  in \cite{HZ:87} that many compact Hamiltonian energy levels contain
  periodic orbits.
Indeed, after some refinement by Rabinowitz in \cite{Rab:87} and Struwe in
  \cite{Struwe:90}, this phenomenon became referred to as ``\textit{almost
  existence}''.
Working in a context in which the almost existence phenomenon holds, the
  actual existence question
  for periodic orbits of Hamiltonian systems quite often can be phrased as
  ``a priori estimates imply existence", see \cite{BHR}.
Finally the phenomenon was explained in \cite{HZ} in terms of
  differentiability properties of the so-called Hofer-Zehnder capacity, see also
  \cite{HZ:1990}.

For our explicit purpose of connecting this topic to feral curve theory,
  we approach the subject in a particular way.
In a first definition we give an abstraction of a regular, smooth and
  compact Hamiltonian energy surface in a symplectic manifold, which we call
  a ``\textit{framed Hamiltonian manifold}''.
By forgetting some of the information carried by a framed Hamiltonian
  manifold we obtain what is called an ``\textit{odd-symplectic manifold}''.

%%%%%%%%%%%%%%%%%%%%%%%%%%%%%%%%%%%%%%%%%%%%%%%%%%%%%%%%%%%%%%%%%%%%%%%%%%%%%%%%
%%%%%%%%%%                          DEFINITION                         %%%%%%%%%
%%%%%                                                                       %%%%
\begin{definition}[framed Hamiltonian manifold and odd-symplectic]
  \label{DEF_framed}
  \hfill\\
A framed Hamiltonian manifold $(M,\lambda,\omega)$ consists of a
  smooth closed odd-dimen\-sional manifold $M=M^{2n+1}$, a two-form $\omega$
  and a one-form $\lambda$ such that $\lambda\wedge \omega^{n}$ is a volume
  form.
When such a \(\lambda\) exists but we only specify \((M, \omega)\), we
  call the pair an odd-symplectic manifold.
\end{definition}
%%%%%                                                                       %%%%
%%%%%%%%%%                                                             %%%%%%%%%
%%%%%%%%%%%%%%%%%%%%%%%%%%%%%%%%%%%%%%%%%%%%%%%%%%%%%%%%%%%%%%%%%%%%%%%%%%%%%%%%
%
A framed Hamiltonian manifold  $(M,\lambda,\omega)$ defines a dynamical
  system.
Namely there exists a non-singular vector field $X$ on $M$ uniquely
  characterized by the equations
\[
i_X\lambda\equiv 1\ \ \textrm{and}\ \ \ i_X\omega\equiv 0.
\]
As in the more special contact case we shall refer to $X$ as the
  \textit{Reeb vector field} (associated to $(M,\lambda,\omega)$).

Assume that we are given a symplectic manifold $(W,\Omega)$ without
  boundary and  consider a compact, smooth, regular and co-oriented
  hypersurface $M$ in \(W\).
Denoting by $\iota:M\rightarrow W$ the inclusion, we abbreviate $\omega=
  \iota^{\ast}\Omega$ and obtain the odd-symplectic manifold $(M,\omega)$.
If  $H:W\rightarrow {\mathbb R}$ is a smooth Hamiltonian and $H^{-1}(E)=M$
  for some number $E\in {\mathbb R}$ and $dH(m)\neq 0$ for $m\in M$ we can
  take a one-form $\lambda$ on $M$ such that
  $\lambda(X_H(m))=1$ for $m\in M$.  
Then $\lambda\wedge\omega^{n}$ is a volume form on $M$.
Hence we obtain a framed Hamiltonian manifold $(M,\lambda,\omega)$. 
One easily verifies that the Reeb vector field $X$ satisfies 
\[
X_H(m)=X(m)\ \ \textrm{for all}\ m\in M,
\]
  where $X_H$ is the Hamiltonian vector field associated to $H$ and
  defined by $i_{X_H}\Omega=-dH$.

In order to study the almost existence phenomenon we also need to consider
  neighborhoods of a smooth, regular and compact energy surface.
Given a co-orientable, compact, smooth and regular hypersurface $M$
  contained in $W$, where $(W,\widetilde{\Omega})$ is a symplectic
  manifold, we obtain an odd-symplectic manifold $(M,\omega)$ as previously
  described. 
Namely we take the inclusion $i:M\rightarrow W$ and define
  $\omega:=i^{\ast}\widetilde{\Omega}$.

We fix a one-form $\lambda$ such that $\lambda\wedge \omega^n$ is a volume
  form on $M$.
We can define on ${\mathbb R}\times M$ with coordinates $(t,m)$ a 2-form
  $\Omega$ by
\[
\Omega =d(t\lambda)+\omega.
\]
It is a trivial exercise that there exists an open neighborhood $U$ of
  $M\equiv \{0\}\times M$ such that $\Omega|U$ is a symplectic form.
Moreover, if we take $U$ small enough we find an embedding
  $\psi:U\rightarrow W$ onto an open neighborhood of $M\subset W$
  such that $\psi^{\ast}\widetilde{\Omega}= \Omega|U$ and in addition
  $\psi(0,m)=m$ for all $m\in M$.
Note that this also implies that given $\lambda_1$ and $\lambda_2$ so that
  $\lambda_i\wedge \omega^n$ are volume forms we find for the corresponding
  $\Omega_i$ defined by 
\[
\Omega_i =d(t\lambda_i)+\omega,
\]
  open neighborhoods $U_1$ and $U_2$ of $M\subset {\mathbb R}\times M$
  such that $\Omega_i|U_i$ are symplectic and there exists a symplectic
  diffeomorphism $\psi:U_1\rightarrow U_2$, which is the identity on
  $M=\{0\}\times M$.
Depending on the case, whether or not $\lambda_1\wedge \omega^n$ or
  $\lambda_2\wedge \omega^n$ define the same orientation on $M$, we must
  have that $\psi$ maps $(s,m)$ for some $s>0$ to some $(s',m')$ with $\pm
  s'>0$.

From the previous discussion it follows that instead of working with a
  suitable open neighborhood $U$ of $M$ we may assume that $U={\mathbb
  R}\times M$ equipped with a symplectic form $\Omega$ such that
  $\omega=\iota^{\ast}\Omega$, where $\iota(m)=(0,m)$ for $m\in M$.
%%%%%%%%%%%%%%%%%%%%%%%%%%%%%%%%%%%%%%%%%%%%%%%%%%%%%%%%%%%%%%%%%%%%%%%%%%%%%%%%
%%%%%%%%%%                          DEFINITION                         %%%%%%%%%
%%%%%                                                                       %%%%
\begin{definition}[compatible $\Omega$]
  \label{DEF_compatible_Omega}
  \hfill\\
Given an odd-symplectic $(M,\omega)$ we call a symplectic form $\Omega$ on
  ${\mathbb R}\times M$ \textit{compatible} provided
  $\iota^{\ast}\Omega=\omega$, where $\iota:M\rightarrow {\mathbb R}\times
  M$ is defined by $\iota(m)=(0,m)$ for $m\in M$.
\end{definition}
%%%%%                                                                       %%%%
%%%%%%%%%%                                                             %%%%%%%%%
%%%%%%%%%%%%%%%%%%%%%%%%%%%%%%%%%%%%%%%%%%%%%%%%%%%%%%%%%%%%%%%%%%%%%%%%%%%%%%%%
%

The notion of \textit{``almost existence''} will be associated to the
  behavior of small neighborhoods of $M=\{0\}\times M$ in ${\mathbb R}\times
  M$ for a symplectic form which restricts to $\omega$.

%%%%%%%%%%%%%%%%%%%%%%%%%%%%%%%%%%%%%%%%%%%%%%%%%%%%%%%%%%%%%%%%%%%%%%%%%%%%%%%%
%%%%%%%%%%                          DEFINITION                         %%%%%%%%%
%%%%%                                                                       %%%%
\begin{definition}[almost existence property - odd symplectic]
  \label{DEF_almost_existence_odd_symplectic}
  \hfill\\
Consider a smooth, compact odd-symplectic manifold $(M,\omega)$, denote by
  $\Omega$ a compatible symplectic form  on ${\mathbb R}\times M$ and
  identify $M\equiv \{0\}\times M$.
We say that $(M,\omega)$ has the \textit{almost existence property}
  provided there exists an open neighborhood $U$ of $M$ with the following
  property.
Given any  proper, smooth and surjective Hamiltonian $H:V\rightarrow
  (-\varepsilon,\varepsilon)$ for some $\varepsilon>0$,  where $V$ is an
  open neighborhood of $M$ contained in $U$  such that $H^{-1}(0)=\iota(M)$
  and $dH(s,m)\neq 0$ for $m\in M$ and $s\in (-\varepsilon,\varepsilon)$
  define the  set $\Sigma_H$ by
\[
\Sigma_H=\{E\in (-\varepsilon,\varepsilon)\ \vert\ \dot{x}=X_H(x)\ 
\textrm{has a periodic orbit with}\ H(x)=E\}.
\]
We say that $(M,\omega)$ has the \textit{almost existence property}
  provided for suitable $U$ it holds for all $H$ as described above that
  $\textrm{measure}(\Sigma_H)=2\cdot\varepsilon$.
\end{definition}
%%%%%                                                                       %%%%
%%%%%%%%%%                                                             %%%%%%%%%
%%%%%%%%%%%%%%%%%%%%%%%%%%%%%%%%%%%%%%%%%%%%%%%%%%%%%%%%%%%%%%%%%%%%%%%%%%%%%%%%
%

Finally we can introduce a special class of symplectic manifolds.

%%%%%%%%%%%%%%%%%%%%%%%%%%%%%%%%%%%%%%%%%%%%%%%%%%%%%%%%%%%%%%%%%%%%%%%%%%%%%%%%
%%%%%%%%%%                          DEFINITION                         %%%%%%%%%
%%%%%                                                                       %%%%
\begin{definition}[almost existence property -- symplectic manifold]
  \label{DEF_almost_existence_property_symplectic}
  \hfill\\
A symplectic manifold $(W,\Omega)$ without boundary has the almost
  existence property provided for every regular, compact, and co-oriented
  hypersurface $M$ the pair $(M,\iota^{\ast}\Omega)$ has the almost existence
  property, where  $\iota:M\rightarrow W$ is the inclusion.
\end{definition}
%%%%%                                                                       %%%%
%%%%%%%%%%                                                             %%%%%%%%%
%%%%%%%%%%%%%%%%%%%%%%%%%%%%%%%%%%%%%%%%%%%%%%%%%%%%%%%%%%%%%%%%%%%%%%%%%%%%%%%%
%

We know that the standard symplectic vector space ${\mathbb R}^{2n}$ has
  the almost existence property, see \cite{HZ:1990}.
We also know that $T^4$ can be be equipped with a symplectic form so that
  $(T^4,\Omega)$ does not have the almost existence property, see also
  \cite{HZ:1990}.
In other words, the almost existence property is nontrivial.
We state the following  theorem for the convenience of the reader. 
It is based on some known facts which we identify as a local property.

%%%%%%%%%%%%%%%%%%%%%%%%%%%%%%%%%%%%%%%%%%%%%%%%%%%%%%%%%%%%%%%%%%%%%%%%%%%%%%%%
%%%%%%%%%%                           THEOREM                           %%%%%%%%%
%%%%%                                                                       %%%%
\begin{theorem}[local almost existence property]
  \label{THM_SOMETHING}\label{th1}
  \hfill\\
Every symplectic manifold $(W,\Omega)$ without boundary has the following
  property.
Given a point $w\in W$ there exists an open neighborhood $U(w)$ so that
  for every closed regular hypersurface  $M\subset U$ the pair
  $(M,i^{\ast}\Omega)$ has the almost existence property.
\end{theorem}
%%%%%                                                                       %%%%
%%%%%%%%%%                                                             %%%%%%%%%
%%%%%%%%%%%%%%%%%%%%%%%%%%%%%%%%%%%%%%%%%%%%%%%%%%%%%%%%%%%%%%%%%%%%%%%%%%%%%%%%
%
In other words, every symplectic manifold without boundary has the "local
  almost existence property". 
It then becomes an interesting question what kind of more global compact,
  regular hypersurfaces in $(W,\Omega)$ have the almost existence property.
We formulate a more precise question next.

%%%%%%%%%%%%%%%%%%%%%%%%%%%%%%%%%%%%%%%%%%%%%%%%%%%%%%%%%%%%%%%%%%%%%%%%%%%%%%%%
%%%%%%%%%%                            LEMMA                            %%%%%%%%%
%%%%%                                                                       %%%%
\begin{question}
  \label{Q_SOMETHING}
Assume that $(W,\Omega)$ is a symplectic manifold without boundary. 
Suppose that $M\subset W$ is a smooth, compact, regular  hypersurface
  without boundary so that the inclusion $\iota: M\rightarrow W$ is isotopic
  to a small hypersurface, i.e. contained in some $U(w)$, see Theorem
  \ref{th1}.
Does then $M$ have the almost existence property? 
If this is not always true, then for which class of symplectic manifolds
  (other than \(\mathbb{R}^{2n}\)) is it true?
\end{question}
%%%%%                                                                       %%%%
%%%%%%%%%%                                                             %%%%%%%%%
%%%%%%%%%%%%%%%%%%%%%%%%%%%%%%%%%%%%%%%%%%%%%%%%%%%%%%%%%%%%%%%%%%%%%%%%%%%%%%%%
%

We note that the literature suggests that compact symplectic manifolds
  with a sufficiently rich Gromov-Witten theory have the almost existence
  property.
The first paper,  predating Gromov-Witten theory, where such an idea  is
  used is \cite{HOFV}. 
It shows that having a suitable moduli space of pseudoholomorphic curves
  associated to a symplectic manifold implies that the Weinstein
  conjecture holds for suitable energy surfaces. 
The method in \cite{HOFV} was then used  
  in \cite{LiuT} and combined with Gromov-Witten
  theory for a more convenient packaging of properties of moduli spaces.
This Gromov-Witten style approach has also been used to prove almost
  existence results in certain cases; see for example \cite{Lu}, Theorem 1.10.

\subsection{Statement of the Main Result}
Here we state the main result of the article.  
The terms used below are standard, however they are provided explicitly in
  Section \ref{SEC_background} below.
For example, the notion of an \(\Omega\)-tame almost complex structure is
  provided in Definition \ref{DEF_almost_complex_structures}; the notion of
  a pseudoholomorphic map is provided in Definition
  \ref{DEF_pseudoholomorphic_map}; and the notion of the genus of such a map
  is provided in Definition \ref{DEF_genus}.

\setcounter{CounterSectionMain}{\value{section}}
\setcounter{CounterTheoremMain}{\value{theorem}}
%%%%%%%%%%%%%%%%%%%%%%%%%%%%%%%%%%%%%%%%%%%%%%%%%%%%%%%%%%%%%%%%%%%%%%%%%%%%%%%%
%%%%%%%%%%                           THEOREM                           %%%%%%%%%
%%%%%                                                                       %%%%
\begin{theorem}[Main Result]
  \label{THM_main_maybe}
  \hfill\\
Let \((W, \Omega)\) be a symplectic manifold without boundary, and let
  \(H\colon W\to \mathbb{R}\) be a smooth proper\footnote{By proper here, we
  mean that for each compact set \(\mathcal{K}\subset \mathbb{R}\), the set
  \(H^{-1}(\mathcal{K})\) is compact.} Hamiltonian.
Fix \(E_-, E_+ \in H(W) \subset \mathbb{R}\) with \(E_- < E_+\), as well
  as positive constants, \(C_g>0\), and \(C_\Omega>0\).
Suppose that for each \(\Omega\)-tame almost complex structure \(J\) on
  \(W\) there exists a proper pseudoholomorphic map
  \begin{align*}                                                          %% EQN
    u:(S,j)\to \{p\in W : E_- < H(p) < E_+ \} 
    \end{align*}
  without boundary, which also satisfies the following conditions:
  \vspace{5pt}
\begin{enumerate}                                                         %% NUM
  \item \emph{({genus and area bounds})}
  The following inequalities hold:
  \begin{align*}                                                          %% EQN
    {\rm Genus}(S) \leq C_g\qquad\text{and}\qquad \int_{S} u^*\Omega
    \leq C_\Omega.
    \end{align*}
  \item \emph{({energy surjectivity})}
  The map \(H\circ u : S\to (E_- , E_+)\) is
  surjective.\vspace{5pt}
  \end{enumerate}
Then there is a periodic Hamiltonian orbit on almost every energy level in
  range \((E_-, E_+)\).
That is, if we let \(\mathcal{I}\subset (E_-, E_+)\) denote the energy
  levels of \(H\) which contain a Hamiltonian periodic orbit, then
  \(\mathcal{I}\) has full measure:
  \begin{align*}                                                          %% EQN
    \mu(\mathcal{I}) = \mu\big((E_-, E_+)\big) = E_+ - E_-.
    \end{align*}
\end{theorem}
%%%%%                                                                       %%%%
%%%%%%%%%%                                                             %%%%%%%%%
%%%%%%%%%%%%%%%%%%%%%%%%%%%%%%%%%%%%%%%%%%%%%%%%%%%%%%%%%%%%%%%%%%%%%%%%%%%%%%%%
%
At this point there are a few points worth making.
The first is that the pseudoholomorphic maps in question here need not be
  compact -- in fact, a careful inspection of the requirements reveals
  that they cannot be compact.

A second point is that the pseudoholomorphic maps allowed by the above
  hypotheses may have domains \((S, j)\) which are diffeomorphic to an open   annulus, but sometimes the domains will be much worse.
For example, our assumptions allow for the possibility that a domain of a
  pseudoholomorphic map may be an infinitely punctured open disk with
  infinitely many closed disks removed.
Or worse: the open disk with the Cantor set removed.
Or any closed Riemann surface with any closed set removed.
To be clear, such assumptions are highly unusual in the standard theory of
  pseudoholomorphic curves, however they are standard in feral curve theory.
This is because feral curves in general are much wilder.
We illustrate this with a plausible example.

Suppose \((W,\Omega)\) is a closed symplectic manifold, with an
  \(\Omega\)-compatible almost complex structure \(J\), a smooth Hamiltonian
  \(H\colon W\to \mathbb{R}\) for which \(0\) is a regular value, and a
  pseudoholomorphic map \(\bar{u}\colon \overline{S}\to W \) where
  \(\overline{S}\) is a closed Riemann surface such as a sphere or torus,
  etc.
For \(\epsilon>0\), what structure does the following set have?
  \begin{align}\label{EQ_example_S}                                       %% EQN
    S := \big\{z\in \overline{S} : -\epsilon < H\big(\bar{u}(z)\big) <
    \epsilon\big\}
    \end{align}
As it need not be the case that \(\pm \epsilon\) are regular values of
  \(H\circ u\), there is no reason that \(S\) should necessarily admit a
  smooth compactification to a compact Riemann surface with smooth
  boundary.
Indeed, all one can really say is that it has the structure of a closed
  Riemann surface with some closed set removed, which is exactly the sort of
  domain surface that Theorem \ref{THM_main_maybe} allows.
Conceptually, it may be easier to think of the maps we allow as arising
  from restrictions like \(u:=\bar{u}\big|_{S}\), with \(S\) as in equation
  (\ref{EQ_example_S}), however it is worth noting that the existence of
  such an extension (or lack thereof) plays no role in our proof.

With such unusual freedom allowed for the domains of our pseudoholomorphic
  maps, the attentive reader may be concerned about the precise notion of
  genus.
This is made rigorous in Definition \ref{DEF_genus} below, although
  it amounts to exhausting \(S\) by compact two-dimensional manifolds with
  boundary and taking the limit of associated the genera.

A third point is that the hypothesis in Theorem \ref{THM_main_maybe}
  regarding the existence of a suitable pseudoholomorphic map for
  \emph{each} \(\Omega\)-tame almost complex structure seems very
  restrictive, however in practice this is not the case.
For example, Gromov-Witten invariants are invariants of a closed
  symplectic manifold obtained by algebraically counting pseudoholomorphic
  maps of specified genus, homology class, and incidence conditions.
In particular, Gromov-Witten invariants are independent of the choice of
  almost complex structure (provided that it is \(\Omega\)-tame), and thus
  a sufficiently rich Gromov-Witten theory for a closed symplectic
  manifold \((W, \Omega)\) is sufficient to guarantee the hypotheses of
  Theorem \ref{THM_main_maybe} are satisfied in many cases.

We will contrast the feral curve techniques used below with the methods
  used in \cite{LiuT} and \cite{Lu}  momentarily, however at present we
  point out that the the proof of Theorem \ref{THM_main_maybe} will show
  that its conclusion holds under weaker assumptions regarding the
  pseudoholomorphic curves.
Specifically, one only needs the existence of pseudoholomorphic curves for
  a particular sequence of adiabatically degenerating almost complex
  structures.
The notion of this adiabatic degeneration is rather technical and is made
  precise in Definition
  \ref{DEF_adiabatically_degenerating_almost_complex_structures}, however
  the idea is to degenerate the almost complex structure so as to
  geometrically ``stretch the neck'' along a continuum of energy levels.
With such an concept internalized, we direct the reader's attention to
  Theorem \ref{THM_localized_almost_existence} in Section
  \ref{SEC_main_proof} below, which is a localized version of Theorem
  \ref{THM_main_maybe} with the assumptions stripped to
  the absolute essentials.

With the exception of Struwe's results in \cite{Struwe}, which predate
  modern symplectic methods, all proofs of almost existence follow a similar
  pattern: prove the Hofer-Zehnder capacity of a domain containing the
  energy surface is finite, and the desired result follows from \cite{HZ}.
For example, in \cite{Lu} Lu takes cues from \cite{LiuT} to use the
  Gromov-Witten invariants to define a \emph{pseudo-}capacity which is
  finite and bounds the Hofer-Zehnder capacity; the almost existence result
  is then immediate.
In contrast, the proof of Theorem \ref{THM_main_maybe} makes no use of
  capacities at all, and only makes use of pseudhoholomorphic curves --
  specifically feral curve theory.
The idea is to stretch the neck along a continuum of energy levels, and
  analyze some basic properties in the limit.
For those familiar with methods from Symplectic Field Theory, it is worth
  pointing out that the key obstacle to overcome is that there are no Hofer
  energy bounds in this case, and hence infinite energy pseudoholomorphic
  curves (i.e feral curves) can appear in the limit ``building.''
The picture that emerges from this analysis is rather interesting.
It seems that as one takes this adiabatic degeneration, curves which cross
  the region of degeneration are inexorably drawn to collapse onto
  families of periodic orbits.
Or more precisely, such a collapse to families of orbits occurs
  \emph{almost} everywhere, and on the complementary measure zero set the
  curves are allowed to jump between such families, or even jump to a more
  general closed invariant subset.
Further analysis of such limiting curves seems well poised to illuminate
  additional dynamical features.

We close this introductory section with an application which appears to be
  inaccessible to methods relying on the finiteness of the Hofer-Zehnder
  capacity.
We state this as Theorem \ref{THM_existence_and_almost_existence} below,
  but first provide a definition.

%%%%%%%%%%%%%%%%%%%%%%%%%%%%%%%%%%%%%%%%%%%%%%%%%%%%%%%%%%%%%%%%%%%%%%%%%%%%%%%%
%%%%%%%%%%                          DEFINITION                         %%%%%%%%%
%%%%%                                                                       %%%%
\begin{definition}[positive contact type]
  \label{DEF_positive_contact_type_boundary}
  \hfill\\
Let \((W, \Omega)\) be a compact symplectic manifold with boundary.  
Let \(M^+\) be a union of connected components of \(\partial W\).
We say \(M^+ \subset \partial W\) is  of \emph{positive contact type} provided
  there exists an outward pointing nowhere vanishing vector field \(Y\)
  defined in a neighborhood of \(M^+\) in \(W\) for which \(\mathcal{L}_Y
  \Omega = \Omega\); here \(\mathcal{L}\) denotes the Lie derivative.
In this case \(\lambda := i_Y \Omega \big|_{M^+}\) is a contact form on
  \(M^+\).
\end{definition}
%%%%%                                                                       %%%%
%%%%%%%%%%                                                             %%%%%%%%%
%%%%%%%%%%%%%%%%%%%%%%%%%%%%%%%%%%%%%%%%%%%%%%%%%%%%%%%%%%%%%%%%%%%%%%%%%%%%%%%%
%

\setcounter{CounterSectionIntertwine}{\value{section}}
\setcounter{CounterTheoremIntertwine}{\value{theorem}}
%%%%%%%%%%%%%%%%%%%%%%%%%%%%%%%%%%%%%%%%%%%%%%%%%%%%%%%%%%%%%%%%%%%%%%%%%%%%%%%%
%%%%%%%%%%                           THEOREM                           %%%%%%%%%
%%%%%                                                                       %%%%
\begin{theorem}[intertwining existence and almost existence]
  \label{THM_existence_and_almost_existence}
  \hfill\\
Let \((W, \Omega)\) be a four-dimensional compact connected exact symplectic
  manifold with boundary \(\partial W = M^+ \cup M^-\).
Suppose \(M^+\) is positive contact type in the sense of Definition
  \ref{DEF_positive_contact_type_boundary}, and suppose that one of the
  following three conditions holds:
\begin{enumerate}                                                         %% NUM
  \item 
  \(M^+\) has a connected component diffeomorphic to \(S^3\),
  \item 
  there exists an embedded \(S^2\subset M^+\) which is homotopically
  nontrivial in \(W\),
  \item 
  \((M^+, \lambda)\) has a connected component which is overtwisted.
  \end{enumerate}
Then for each Hamiltonian \(H\in C^\infty(W,{\mathbb R})\) for which 
  \(H^{-1}(\pm 1) = M^\pm\), the following is true.
For each \(s\in [-1,1]\) the energy level \(H^{-1}(s)\) contains a closed
  non-empty set other than the energy level \(H^{-1}(s)\) itself which is
  invariant under the Hamiltonian flow of \(X_H\); moreover for almost every
  \(s\in [-1, 1]\) this closed invariant subset is a periodic orbit.
\end{theorem}
%%%%%                                                                       %%%%
%%%%%%%%%%                                                             %%%%%%%%%
%%%%%%%%%%%%%%%%%%%%%%%%%%%%%%%%%%%%%%%%%%%%%%%%%%%%%%%%%%%%%%%%%%%%%%%%%%%%%%%%
%
Note that the image of $H$ is not required to lie in $[-1,1]$.

The proof of Theorem \ref{THM_existence_and_almost_existence} is
  provided in Section \ref{SEC_main_proof} below.

%%%%%%%%%%%%%%%%%%%%%%%%%%%%%%%%%%%%%%%%%%%%%%%%%%%%%%%%%%%%%%%%%%%%%%%%%%%%%%%%
%%%%%                           SUB-SECTION                                %%%%%
%%%%%%%%%%%%%%%%%%%%%%%%%%%%%%%%%%%%%%%%%%%%%%%%%%%%%%%%%%%%%%%%%%%%%%%%%%%%%%%%
\subsection{Speculation on the
  \texorpdfstring{$C^\infty$}{C-infinity}-Closing Property} 
  \label{SEC_speculation}

Before moving on to the proofs of Theorem \ref{THM_main_maybe} and Theorem
  \ref{THM_existence_and_almost_existence}, we wish to draw some
  connections to the \(C^\infty\)-closing lemma.
We also aim to pose some speculative questions which we believe the
  feral curve techniques below seem well poised to eventually answer.

We begin with an odd-symplectic manifold $(M,\omega)$, and take a compatible
  symplectic form $\Omega$ on ${\mathbb R}\times M$.
We consider the Fr\'echet space $C^{\infty}(M,{\mathbb R})$ and observe
  that every element $f$ defines a hypersurface in ${\mathbb R}\times M$ by
  setting
  \[
  M_f=\{(f(m),m)\in \mathbb{R}\times M :\ m\in M\}.
  \] 
Considering the hypersurface $M_f\subset ({\mathbb R}\times M,\Omega)$ we
  obtain the distinguished line bundle ${\mathcal L}_{f}\rightarrow M_f$,
  given by
  \[
   {\mathcal L}_f:= \ker(\omega_f)\subset TM_f,
   \]
  where $\omega_f$ is the pull-back of $\Omega$ by the inclusion
  $M_f\rightarrow {\mathbb R}\times M$.
Since ${\mathcal L}_f\subset TM_f$ is a dimension one sub-bundle, it is an
  integrable distribution and we are interested in the closed leaves.
We denote by ${\mathcal C}_f$ the union of all points in $M_f$ which lie
  on a closed leaf.
We say that the periodic orbits are dense on $M_f$ provided
  $\textrm{cl}({\mathcal C}_f)=M_f$.
  
 %%%%%%%%%%%%%%%%%%%%%%%%%%%%%%%%%%%%%%%%%%%%%%%%%%%%%%%%%%%%%%%%%%%%%%%%%%%%%%%%
%%%%%%%%%%                          DEFINITION                         %%%%%%%%%
%%%%%                                                                       %%%%
\begin{definition}[$C^\infty$-closing property]
  \label{DEF_closing_property_framed_hamiltonian}
  \hfill\\
We say that the odd-symplectic manifold $(M,\omega)$ has the
  $C^{\infty}$-\textit{closing property} provided that there exists a
  compatible $\Omega$ on ${\mathbb R}\times M$ so that for a Baire subset
  $\Sigma$ of $C^{\infty}(M,{\mathbb R})$  the following holds:
  \[
  \textrm{\emph{cl}}({\mathcal C}_f) = M_f\ \ \textrm{for all}\ f\in\Sigma.
  \]
  \end{definition}
%%%%%                                                                       %%%%
%%%%%%%%%%                                                             %%%%%%%%%
%%%%%%%%%%%%%%%%%%%%%%%%%%%%%%%%%%%%%%%%%%%%%%%%%%%%%%%%%%%%%%%%%%%%%%%%%%%%%%%%
%
Again one can use the closing property to define a particular class of
  symplectic manifolds.

%%%%%%%%%%%%%%%%%%%%%%%%%%%%%%%%%%%%%%%%%%%%%%%%%%%%%%%%%%%%%%%%%%%%%%%%%%%%%%%%
%%%%%%%%%%                          DEFINITION                         %%%%%%%%%
%%%%%                                                                       %%%%
\begin{definition}[$C^\infty$-closing property -- symplectic
  manifolds]
  \label{DEF_closing_property_symplectic}
  \hfill\\
We say a symplectic manifold $(W,\Omega)$ has the
  $C^{\infty}$-\textit{closing property} provided for every regular compact
  co-oriented hypersurface $M$ in $W$ the induced $(M,\omega)$ has the
  $C^{\infty}$-closing property.
\end{definition}
%%%%%                                                                       %%%%
%%%%%%%%%%                                                             %%%%%%%%%
%%%%%%%%%%%%%%%%%%%%%%%%%%%%%%%%%%%%%%%%%%%%%%%%%%%%%%%%%%%%%%%%%%%%%%%%%%%%%%%%
%

One can play around with the above definition by allowing only
  hypersurfaces isotopic to small ones or those carrying a suitable
  topology.
Alternatively, one might choose to only allow contact-type hypersurfaces. 
We leave it to the reader to explore these ideas and we only mention the
  following conjecture.

%%%%%%%%%%%%%%%%%%%%%%%%%%%%%%%%%%%%%%%%%%%%%%%%%%%%%%%%%%%%%%%%%%%%%%%%%%%%%%%%
%%%%%%%%%%                          CONJECTURE                         %%%%%%%%%
%%%%%                                                                       %%%%
\begin{conjecture}[local $C^{\infty}$-closing property]
  \label{CON_closing}
  \hfill\\
The standard symplectic vector space $({\mathbb
  R}^{2n},\Omega_{\textrm{standard}})$, $n\geq 2$,  has the
  $C^{\infty}$-closing property.
In particular all symplectic manifolds without boundary have the local
  $C^{\infty}$-closing property.
\end{conjecture}
%%%%%                                                                       %%%%
%%%%%%%%%%                                                             %%%%%%%%%
%%%%%%%%%%%%%%%%%%%%%%%%%%%%%%%%%%%%%%%%%%%%%%%%%%%%%%%%%%%%%%%%%%%%%%%%%%%%%%%%
%
The conjecture is open for all $n\geq 2$ and basically nothing is known for $n\geq 3$.
In the case of $n=2$ one knows a partial result, namely that compact,  regular hypersurfaces 
in ${\mathbb R}^4$ of contact type have the $C^{\infty}$-closing property. However, nothing is known 
for general compact regular hypersurfaces in ${\mathbb R}^4$.
Indeed, by a result of Irie, \cite{Irie:2015}, every $(M,\lambda,d\lambda)$, where
  $M$ is a closed three-manifold equipped with a contact form $\lambda$, has
  the $C^{\infty}$-closing property.
In particular every compact regular hypersurface of contact type in
  ${\mathbb R}^4$ has the $C^{\infty}$-closing property.
In the background of Irie's result and a follow-up result, \cite{ AI}, is
  the important volume formula by  Cristofaro-Gardiner,  Hutchings,  and
  Ramos, \cite{CHR}.
At present the proof of \cite{Irie:2015} based on \cite{CHR}  only works in the
  three-dimensional case.
Due to the use of Seiberg-Witten-Floer Theory it will need some new ideas
  to attack the higher-dimensional cases -- perhaps feral curves.

The germ of an idea goes as follows.
The key upshot of feral curves, is that one can stretch the neck along any
  hypersurface.
The downside is that one may find limit sets which are much more
  complicated than a finite set of periodic orbits.
However, Theorem \ref{THM_main_maybe} strongly suggests that generically
  (in the right sense) one can stretch the neck while following a
  \emph{single} curve for each almost complex structure, and pass to the
  limit to find a periodic orbit.
That is, at least one periodic orbit is (generically) found by tracking
  just one curve.
What if there are many curves to track?
High dimensional families of curves, for example.
Here the proposed richness of the Gromov-Witten invariants comes into
  play.
For example, consider \(\mathbb{R}^{2n} \simeq \mathbb{C}^n =
  \mathbb{C}P^n\setminus \mathbb{C}P^{n-1}\), and suppose we consider
  stretching the neck along some generic hypersurface \(M\subset
  \mathbb{R}^{2n}\).
By considering curves of high degree, one obtains high dimensional
  families of curves which stretch and break along periodic orbits.
This raises a question: Which orbits can be found by stretching the neck
  and tracking curves of any fixed (but arbitrarily large) degree?
All orbits, or just a subset?
A dense subset?
If it turns out that neck stretching can find (nearly) every orbit 
  then feral curves seem well poised to recover the
  \(C^\infty\)-closing lemma for arbitrary regular, compact hypersurfaces in \(\mathbb{R}^{2n}\).  
 Currently the known results about the contact-type case in ${\mathbb R}^4$ depend on Seiberg-Witten theory.
  The feral curve theory should be important in removing the contact-type hypothesis. In order to prove the results in higher dimensions
 one way to succeed seems  to be the development of suitable techniques to use higher-dimensional moduli spaces.

\section{Background and Geometric Framework}\label{SEC_background}   
Here we will recall some standard background material, and then we
  provide a geometric framework very specific to the problem of study.
We begin with the notions of an almost complex manifold and
  pseudoholomorphic maps.

%%%%%%%%%%%%%%%%%%%%%%%%%%%%%%%%%%%%%%%%%%%%%%%%%%%%%%%%%%%%%%%%%%%%%%%%%%%%%%%%
%%%%%%%%%%                          DEFINITION                         %%%%%%%%%
%%%%%                                                                       %%%%
\begin{definition}[almost complex structures; compatible and tame]
  \label{DEF_almost_complex_structures}
  \hfill\\
Let \(W\) be a smooth manifold, not necessarily closed, possibly with
  boundary, and let \(J\in \Gamma(\rm{End}(TW))\) be a smooth
  section for which \(J\circ J=-\mathds{1}\).
We call \(J\) an almost complex structure for \(W\), and the pair
  \((W, J)\) an almost complex manifold.
In the case that \(\Omega\) is a symplectic form on \(W\), we say \(J\) is
  \(\Omega\)-compatible provided that \(g(v,w):=\Omega(v, Jw)\) is a
  Riemannian metric.
Under the weaker assumption that \(\Omega(v, Jv) \geq 0\) for all \(v\in
  TW\), and with equality if and only if \(v=0\), we say instead that \(J\)
  is only \(\Omega\)-tame.

\end{definition}
%%%%%                                                                       %%%%
%%%%%%%%%%                                                             %%%%%%%%%
%%%%%%%%%%%%%%%%%%%%%%%%%%%%%%%%%%%%%%%%%%%%%%%%%%%%%%%%%%%%%%%%%%%%%%%%%%%%%%%%
%

%%%%%%%%%%%%%%%%%%%%%%%%%%%%%%%%%%%%%%%%%%%%%%%%%%%%%%%%%%%%%%%%%%%%%%%%%%%%%%%%
%%%%%%%%%%                          DEFINITION                         %%%%%%%%%
%%%%%                                                                       %%%%
\begin{definition}[pseudoholomorphic map]
  \label{DEF_pseudoholomorphic_map}
  \hfill\\
Let \((S, j)\) and \((W, J)\) be smooth almost complex manifolds with
  \({\rm dim}(S)=2\), each possibly with boundary.
A \(C^\infty\)-smooth map \(u:S\to W\) is said to be
  \emph{pseudoholomorphic} provided \(J\cdot Tu = Tu\cdot j\). 
That is, the tangent map of \(u\) intertwines the almost complex
  structures on domain and target.
Unless otherwise specified, we allow \(S\) to be disconnected.  
We say such a map is \emph{proper} provided the pre-image of any compact
  set is compact.
\end{definition}
%%%%%                                                                       %%%%
%%%%%%%%%%                                                             %%%%%%%%%
%%%%%%%%%%%%%%%%%%%%%%%%%%%%%%%%%%%%%%%%%%%%%%%%%%%%%%%%%%%%%%%%%%%%%%%%%%%%%%%%
%

It is worth noting that constant maps are always pseudoholomorphic.
We also note that we will allow the domains of our pseudoholomorphic maps
  to be disconnected, which in conjunction with the fact that constant
  maps are always pseudoholomorphic allows for the possibility that any
  given map may have many constant components -- perhaps infinitely many.
At present we allow this, while noting that second countability of the
  domain Riemann surfaces forces any given map to have at most countably
  many constant components.

Because the domains of our pseudoholomorphic maps are Riemann surfaces
  which need not be compact, we make the notion of \emph{genus} precise with
  the following definition.

\begin{definition}[Genus]
  \label{DEF_genus}
Let \(S\) be a two-dimensional oriented manifold, possibly with boundary,
  with at most countably many connected components, and with the property
  that each connected component of \(\partial S\) is compact.
Then
\begin{enumerate}                                                         %% NUM
    \item If \(S\) is closed and connected, then define \({\rm
      Genus}(S):=g\) where \(\chi(S)=2-2g\) is the Euler characteristic
      of \(S\).
    \item If \(S\) is compact and connected with \(n\) boundary
      components, define \(\widetilde{S}=\big(S \sqcup (\sqcup_{k=1}^n
      D^2)\big)/\sim\) to be the closed surface capped off by \(n\) disks,
      and define \({\rm Genus}(S):={\rm Genus}(\widetilde{S})\).
    \item If \(S\) is compact (possibly with boundary), then \({\rm
      Genus}(S)\) is defined to be the sum of the genera of each connected
      component.
    \item If \(S\) is not compact, then \({\rm Genus}(S)\)
      is defined by taking any nested sequence \(S_1\subset S_2\subset
      S_3\subset\cdots\) of compact surfaces (possibly with boundary) such
      that \(S_k\subset S\) for all \(k\in \mathbb{N}\) and such that \(S
      =\cup_k S_k\); then we define \({\rm Genus}(S):=\lim_{k\to \infty}
      {\rm Genus}(S_k)\).
  \end{enumerate}
\end{definition}
%%%%%                                                                       %%%%
%%%%%%%%%%                                                             %%%%%%%%%
%%%%%%%%%%%%%%%%%%%%%%%%%%%%%%%%%%%%%%%%%%%%%%%%%%%%%%%%%%%%%%%%%%%%%%%%%%%%%%%%

We now turn our attention to geometric structures more specific to the
  proof of Theorem \ref{THM_main_maybe}.
The first key idea is what we call a \emph{framed Hamiltonian energy
  pile} (see Definition \ref{DEF_framed_hamiltonian_energy_pile}) which is
  essentially the neighborhood of a compact energy level with enough
  structure to regard it as something like a family of framed Hamiltonian
  manifolds.
This relationship is made precise with Lemma
  \ref{LEM_energy_levels_are_framed_Hamiltonian_manifolds}.
Also important is Lemma
  \ref{LEM_localization_to_framed_Hamiltonian_energy_pile}, which is a means
  to identify the neighborhood of a compact energy level in a general
  symplectic manifold, with the structure of a framed Hamiltonian energy
  pile, thereby localizing the almost existence problem.

%%%%%%%%%%%%%%%%%%%%%%%%%%%%%%%%%%%%%%%%%%%%%%%%%%%%%%%%%%%%%%%%%%%%%%%%%%%%%%%%
%%%%%%%%%%                          DEFINITION                         %%%%%%%%%
%%%%%                                                                       %%%%
\begin{definition}[framed Hamiltonian energy pile]
  \label{DEF_framed_hamiltonian_energy_pile}
  \hfill\\
Let \(\mathcal{I}_\epsilon\) be the interval \((-\epsilon, \epsilon)\)  
  equipped with the coordinate \(s\), and let \(M\) be a closed
  odd-dimensional manifold.
Assume \(\mathcal{I}_\epsilon \times M \) is equipped with the symplectic
  form \(\Omega\).
Let \(H\colon \mathcal{I}_\epsilon \times M \to \mathbb{R}\) be the smooth
  Hamiltonian \(H(s,p)=s\), and let \(X_H\) be the associated Hamiltonian
  vector field determined by \(i_{X_H}\Omega = - d H\)
and let  \(\hat{\lambda}\) be a one-form on \(\mathcal{I}_\epsilon \times M\)
  which satisfies the following three conditions:
\begin{enumerate}[(H1)]                                                   %% NUM
  \item 
  \(\hat{\lambda}(\partial_s) = 0\), 
  \item 
  \(\mathcal{L}_{\partial_s} \hat{\lambda}=0\) where \(\mathcal{L}\) is
  the Lie derivative,
  \item \(\hat{\lambda}(X_H) >0\).
  \end{enumerate}
We call the triple \((\mathcal{I}_\epsilon \times M,\Omega, \hat{\lambda})
  \) a framed Hamiltonian energy pile.
\end{definition}
If we consider local coordinates $x_1,...,x_N$ on $M$, $N=2n+1$, and the
  coordinate $s$ on $\mathcal{I}_{\varepsilon}$ then $\hat{\lambda}$ can be
  written at the point $(s,x)$ as
\begin{eqnarray}\label{Hope}
\hat{\lambda}(s,x)= \sum_{i=1}^N a_i(x)dx_i.
\end{eqnarray}
due to the imposed conditions (H1) and (H2).
%%%%%                                                                       %%%%
%%%%%%%%%%                                                             %%%%%%%%%
%%%%%%%%%%%%%%%%%%%%%%%%%%%%%%%%%%%%%%%%%%%%%%%%%%%%%%%%%%%%%%%%%%%%%%%%%%%%%%%%
%
We shall show in Lemma
  \ref{LEM_localization_to_framed_Hamiltonian_energy_pile} how framed
  Hamiltonian energy piles arise near a compact regular energy surface. 
At first, however, we begin by deriving a few geometric structures which
  arise as an immediate consequence of having a framed Hamiltonian energy
  pile.
We illuminate them at present.
Define the two-plane distribution \(\hat{\rho}\) on
  $\mathcal{I}_{\varepsilon}\times M$ by
  \begin{align} \label{EQQQ_xi}                                           %% EQN
    \hat{\rho} = {\rm Span}(\partial_s , X_H),
    \end{align}
  and define the codimension-two plane distribution \(\xi\) on
  $\mathcal{I}_{\varepsilon}\times M$ by
  \begin{align}\label{EQ_xi}                                              %% EQN
    \xi = {\rm ker}\; (ds \wedge \hat{\lambda})= ({\rm ker}\; ds) \cap
    ({\rm ker}\; \hat{\lambda}).
    \end{align}
We observe that the vector bundle \(\xi\rightarrow
  \mathcal{I}_{\varepsilon}\times M\) is ${\mathbb R}$-invariant in the
  following sense.
Given $s_0,s_1\in\mathcal{I}_{\varepsilon}$,  the map
  $(h+s_0,m)\rightarrow (h+s_1,m)$, which is defined for small \(\vert
  h\vert \), pushes forward (the obvious restrictions of)  $\xi$ to $\xi$.
Moreover, we have  the splitting 
\[T(\mathcal{I}_\epsilon \times M) =
  \hat{\rho}\oplus \xi,
  \]
   and the associated projections:
  \begin{align*}                                                          %% EQN
    \pi^{\hat{\rho}}\colon \hat{\rho}\oplus \xi \to \hat{\rho} 
    \qquad\text{and}\qquad
    \pi^{\xi}\colon \hat{\rho}\oplus \xi \to \xi.
    \end{align*}
Define the two-form \(\hat{\omega}\) by 
\begin{align}   \label{XX_decomp}                                         %% EQN
  \hat{\omega} = \Omega\circ (\pi^\xi \times \pi^\xi).
  \end{align}
Here and throughout, we will also employ the following notation,
  \begin{align*}                                                          %% EQN
    \hat{X}:= \frac{X_H}{\hat{\lambda}(X_H)}.
    \end{align*}

%%%%%%%%%%%%%%%%%%%%%%%%%%%%%%%%%%%%%%%%%%%%%%%%%%%%%%%%%%%%%%%%%%%%%%%%%%%%%%%%
%%%%%%%%%%                            LEMMA                            %%%%%%%%%
%%%%%                                                                       %%%%
\begin{lemma}[localization to a framed Hamiltonian energy pile]
  \label{LEM_localization_to_framed_Hamiltonian_energy_pile}
  \hfill\\
Let \((\widetilde{W},\widetilde{\Omega})\) be a symplectic manifold
  without boundary, and let \(\widetilde{H}\colon \widetilde{W} \to
  \mathbb{R}\) be a \(C^\infty\)-smooth proper Hamiltonian.
Suppose further that zero is a regular value of
  \(\widetilde{H}\).
Then there exists an \(\epsilon>0\), a framed Hamiltonian energy pile
  \((\mathcal{I}_\epsilon\times M, \Omega, \hat{\lambda})\), and a
  \(C^\infty\)-smooth diffeomorphism
  \begin{align*}                                                          %% EQN
    \Phi\colon \mathcal{I}_\epsilon \times M \to \{|\widetilde{H}| <
    \epsilon\} \subset \widetilde{W}
    \end{align*}
  for which 
  \begin{align*}                                                          %% EQN
    (\widetilde{H}\circ \Phi)(s,p) = s\qquad\text{and}\qquad
    \Phi^*\widetilde{\Omega} = \Omega.
    \end{align*}
Additionally, the framed Hamiltonian energy pile can be found such that
  along the energy level \(\{0\}\times M\), the following are true
\begin{enumerate}                                                         %% NUM
  \item 
  \(\hat{\lambda}(X_H)=1.\)
  \item 
  The vector bundles \(\hat{\rho}\) and \(\xi\) are symplectic complements
    over $\{0\}\times M$.
  That is, for each \(q\in \{0\}\times M\), each \(v_q \in \hat{\rho}_q\),
    and each \(w_q\in \xi_q\), we have \(\Omega(v_q, w_q)=0\).
  \end{enumerate}
\end{lemma}
%%%%%                                                                       %%%%
%%%%%%%%%%                                                             %%%%%%%%%
%%%%%%%%%%%%%%%%%%%%%%%%%%%%%%%%%%%%%%%%%%%%%%%%%%%%%%%%%%%%%%%%%%%%%%%%%%%%%%%%
%
\begin{proof}
We begin by fixing an auxiliary \(\widetilde{\Omega}\)-compatible almost
  complex structure on \(\widetilde{W}\); denote it \(\widetilde{J}\).
This gives rise to the Riemannian metric \(g_{\widetilde{J}} =
  \widetilde{\Omega}\circ ({\rm Id}\times \widetilde{J})\).
Using this metric to compute the gradient of \(\widetilde{H}\), we let
  \(\varphi^t\) be the time \(t\) flow associated to the vector field
  \(\|\nabla \widetilde{H}\|_{g_{\widetilde{J}}}^{-2}\nabla
  \widetilde{H}\) in a neighborhood \(\{|\widetilde{H}|<\epsilon\}\) for
  some small \(\epsilon>0\) yet to be determined.
Defininig \(M:=\widetilde{H}^{-1}(0)\), it follows that for all
  sufficiently small \(\epsilon>0\) 
  \begin{align*}                                                          %% EQN
    &\Phi\colon \mathcal{I}_\epsilon\times M \to \{|\widetilde{H}|
    <\epsilon\} \subset \widetilde{W}\\
    &\Phi(s,p) = \varphi^s(p)
    \end{align*}
  is a diffeomorphism, and \(H:= \widetilde{H}\circ \Phi\)
  satisfies \(H(s,p) = s\).

By construction, the vector field \(X_{\widetilde{H}}\) is tangent to level
  sets of \(\widetilde{H}\), and thus along \(\{\widetilde{H}=0\}\) we can
  define \(\tilde{\lambda}\) to be the one form uniquely determined by the
  conditions
  \begin{align*}                                                          %% EQN
  \tilde{\lambda}(X_{\widetilde{H}})=1 \qquad\text{and}\qquad
    {\rm ker}\; \tilde{\lambda} = \mathbb{R} \nabla \widetilde{H} \, \oplus
    \big({\rm Span}(X_{\widetilde{H}}, \nabla\widetilde{H})\big)^\bot
    \end{align*}
  where \(\bot\) denotes the \(\widetilde{\Omega}\)-symplectic complement.
We then define
  \begin{align*}                                                          %% EQN
    \hat{\lambda} = {\rm pr}_2^* \tilde{\lambda}
    \end{align*}
  where \({\rm pr}_2\colon \mathcal{I}_\epsilon \times M\to M\) is the
  canonical projection.
It is straightforward to verify that \(\hat{\lambda}(\partial_s)=0\), and
  that \(\mathcal{L}_{\partial_s} \hat{\lambda}=0\).
Because \(\hat{\lambda}(X_H)\big|_{\{0\}\times M} = 1\), it then follows
  that \(\hat{\lambda}(X_H)>0\) for all sufficiently small \(\epsilon>0\).
By construction then, \((W, \Omega, \hat{\lambda})\) is a framed
  Hamiltonian energy pile provided that \(\epsilon>0\) is sufficiently
  small.
Moreover by construction, along \(\{0\}\times M\) it is both the case that
  \(\hat{\lambda}(X_H)=1\) and the case \(\hat{\rho}\) and \(\xi\) are
  \(\Omega\)-symplectic complements.
This completes the proof of Lemma
  \ref{LEM_localization_to_framed_Hamiltonian_energy_pile}.
\end{proof}%%%%%%%%%%%%%%%%%%%%%%%%%%%%%%%%%%%%%%%%%%%%%%%%            END PROOF
We make the following important observation.

%%%%%%%%%%%%%%%%%%%%%%%%%%%%%%%%%%%%%%%%%%%%%%%%%%%%%%%%%%%%%%%%%%%%%%%%%%%%%%%%
%%%%%%%%%%                            LEMMA                            %%%%%%%%%
%%%%%                                                                       %%%%
\begin{lemma}[energy levels are framed Hamiltonian manifolds]
  \label{LEM_energy_levels_are_framed_Hamiltonian_manifolds}
  \hfill\\
Let \(\big(\mathcal{I}_\epsilon\times M, \Omega, \hat{\lambda}\big)\) be a
  framed Hamiltonian energy pile, and let \(\hat{\rho}\), \(\xi\), and
  \(\hat{\omega}\) be the associated structures defined above, see
  (\ref{EQQQ_xi}), (\ref{EQ_xi}), and (\ref{XX_decomp}).
Then for each \(s_0\in \mathcal{I}_\epsilon\), the restriction of
  \(\hat{\lambda}\) and \(\hat{\omega}\) to the energy level \(\{s_0\}\times
  M\) is a framed Hamiltonian structure for this energy level.
\end{lemma}
%%%%%                                                                       %%%%
%%%%%%%%%%                                                             %%%%%%%%%
%%%%%%%%%%%%%%%%%%%%%%%%%%%%%%%%%%%%%%%%%%%%%%%%%%%%%%%%%%%%%%%%%%%%%%%%%%%%%%%%
%
\begin{proof}
We begin by recalling that a framed Hamiltonian structure \(\eta=(\lambda,
  \omega)\) for an odd-dimensional manifold \(M\) is a one-form
  \(\lambda\) and a closed two-form \(\omega\) for which  \(\lambda
  \wedge \omega \wedge \cdots \wedge \omega =: {\rm vol}_M\) is a volume form.
To show this latter non-degeneracy condition is satisfied on our energy
  levels, first note that on \(\mathcal{I}_\epsilon \times M\) we have that
  \(T(\mathcal{I}_\epsilon\times M) = \hat{\rho}\oplus \xi\)  is a
  splitting and \(\Omega\) is non-degenerate on each of \(\hat{\rho}\) and
  \(\xi\).
Moreover, by construction 
  \begin{align*}                                                          %% EQN
    \xi = {\rm ker}\; (ds\wedge \hat{\lambda})\qquad\text{and}\qquad
    \hat{\rho} = {\rm ker}\; (\hat{\omega}) = { \rm Span}(\partial_s,
    X_H),
    \end{align*}
  and \(\hat{\omega}=\Omega\circ(\pi^\xi \times \pi^\xi)\) so that
  \begin{align*}                                                          %% EQN
    ds\wedge \hat{\lambda} \wedge \hat{\omega}^n >0 \qquad\text{on
    }\mathcal{I}_\epsilon\times M;
    \end{align*}
    here \({\rm dim} (M) = 2n +1\).
It immediately follows that 
  \begin{align*}                                                          %% EQN
    \hat{\lambda} \wedge \hat{\omega}^n >0 \qquad\text{on } \{s_0\}\times
    M \text{ for each }s_0\in \mathcal{I}_\epsilon.
    \end{align*}
This establishes the non-degeneracy condition.

Next, we establish that the restriction of \(\hat{\omega}\) to each energy
  level \(\{s_0\}\times M\) is closed.
To that end, let \(v,w\in T(\{s_0\}\times M)\).
Then there exist \(a,b\in \mathbb{R}\) and \(v^\xi, w^\xi \in \xi\) such
  that
  \begin{align*}                                                          %% EQN
    v = a X_H + v^\xi\qquad\text{and}\qquad w = b X_H + v^\xi, 
    \end{align*}
  and then
  \begin{align*}                                                          %% EQN
    \Omega(v, w) &= \Omega\big(a X_H + v^\xi,b X_H + w^\xi\big)
    \\
    &= 
    \Omega(v^\xi,w^\xi)
    \\
    &= 
    \Omega\big(\pi^\xi(v),\pi^\xi(w)\big)
    \\
    &= 
    \hat{\omega}(v,w)
    \end{align*}
  which shows that for the inclusion \(\iota\colon \{s_0\}\times M
  \hookrightarrow \mathcal{I}_\epsilon \times M \) we have
  \(\hat{\omega}\big|_{\{s_0\}\times M}
  = \iota^*\Omega\), and thus the restriction of \(\hat{\omega}\) to
  \(\{s_0\}\times M\) is closed.
This completes the proof of Lemma
  \ref{LEM_energy_levels_are_framed_Hamiltonian_manifolds}.
\end{proof}%%%%%%%%%%%%%%%%%%%%%%%%%%%%%%%%%%%%%%%%%%%%%%%%            END PROOF
For the following considerations of adapted almost complex structures we
  summarize the salient points of the previous discussion.
By definition a framed Hamiltonian manifold \((\check{M}, \check{\lambda},
  \check{\omega})\) has an associated Reeb vector field
  \(\check{X}\) uniquely determined by the equations
  \begin{align*}                                                          %% EQN
    \check{\lambda}(\check{X}) = 1\qquad\text{and}\qquad
    i_{\check{X}}\check{\omega} =0.
    \end{align*}
Given a framed Hamiltonian energy pile $(\mathcal{I}_{\varepsilon}\times
  M,\Omega,\hat{\lambda})$ we can consider the Hamiltonian $H$ given by
  $H(s,m)=s$, which has Hamiltonian vector field $X_H$ is defined by
  $i_X\Omega=-dH$.
We normalize it by setting
  \[
  \widehat{X}=X_H/\hat{\lambda}(X_H).
  \]
Then we observe that $\widehat{X}$ satisfies
  \begin{align*}                                                          %% EQN
    \hat{\lambda}(\widehat{X}) = 1, \qquad ds(\widehat{X})=0,
    \qquad\text{and}\qquad i_{\widehat{X}}\hat{\omega}=0.
    \end{align*}
Consequently, on each energy level \(\{s_0\}\times M\) the Reeb 
  vector field associated to the framed Hamiltonian structure
  \((\hat{\lambda}, \hat{\omega})\big|_{\{s_0\}\times M}\) is the
  restriction of \(\widehat{X}\).

We now turn to equipping framed
  Hamiltonian energy piles with a certain class of almost complex
  structures, which we now define.

%%%%%%%%%%%%%%%%%%%%%%%%%%%%%%%%%%%%%%%%%%%%%%%%%%%%%%%%%%%%%%%%%%%%%%%%%%%%%%%%
%%%%%%%%%%                          DEFINITION                         %%%%%%%%%
%%%%%                                                                       %%%%
\begin{definition}[weakly adapted almost complex structure]
  \label{DEF_weakly_adapted_J}
  \hfill\\
Let \(\big(\mathcal{I}_\epsilon\times M, \Omega, \hat{\lambda}\big)\) be a
  framed Hamiltonian energy pile, and let \(\hat{\rho}\), \(\xi\), and
  \(\hat{\omega}\) be the associated structures defined above.
Let \(J\) be an almost complex structure on \(\mathcal{I}_\epsilon\times
  M\) which satisfies the following conditions.
\begin{enumerate}[(J1)]                                                   %% NUM
  \item  
  It preserves the splitting \(\hat{\rho}\oplus \xi\); that is,
  \begin{align*}                                                          %% EQN
    J: \hat{\rho}\to \hat{\rho} ,
    \qquad\text{and}\qquad
    J\colon \xi \to \xi. 
    \end{align*}
  \item There exists a smooth function of the form
  \begin{align*}                                                          %% EQN
    &\phi\colon \mathcal{I}_\epsilon \times M\to (0, 1]
    \\
    &\phi(s,p) = \phi(s)
    \end{align*}
  with the property that
  \begin{align*}                                                          %% EQN
    \phi \cdot J \partial_s =  \widehat{X}\qquad\text{where}\qquad
    \widehat{X}=\frac{X_H}{\hat{\lambda}(X_H)}.
    \end{align*}
  \item 
  The following is a Riemannian metric 
  \begin{align}\label{EQ_Riemannian_metric}                               %% EQN
    g_J(v,w) =(ds \wedge \hat{\lambda} + \hat{\omega})(v, J w).
    \end{align}
  \end{enumerate}
In this case we say \(J\) is an almost complex structure \textit{weakly adapted} to
  the framed Hamiltonian energy pile \((\mathcal{I}_\epsilon \times M,
  \Omega, \hat{\lambda})\).
\end{definition}
%%%%%                                                                       %%%%
%%%%%%%%%%                                                             %%%%%%%%%
%%%%%%%%%%%%%%%%%%%%%%%%%%%%%%%%%%%%%%%%%%%%%%%%%%%%%%%%%%%%%%%%%%%%%%%%%%%%%%%%
%

We pause for a moment to discuss the manner in which these almost complex
  structures are weakly adapted.
Consider a fixed framed Hamiltonian energy pile
  $(\mathcal{I}_{\varepsilon}\times M,\Omega,\hat{\lambda})$. 
Associated to this data we have the plane field bundles
\begin{align*}                                                          %% EQN
  \xi = {\rm ker}\; (ds\wedge \hat{\lambda})\qquad\text{and}\qquad
  \hat{\rho} = {\rm ker}\; (\hat{\omega}) = { \rm Span}(\partial_s,
  X_H)={ \rm Span}(\partial_s, \widehat{X}),
  \end{align*}
   with the two-form \(\hat{\omega}=\Omega\circ(\pi^\xi \times \pi^\xi)\).
 Here we recall that
  \[
  T(\mathcal{I}_{\varepsilon}\times M)= \hat{\rho}\oplus \xi.
  \]
Following motivation from Symplectic Field Theory (SFT) as proposed
  in \cite{EGH}, there is a somewhat natural choice of class of almost
  complex structures here, namely those of the form
  \begin{align*}                                                          %% EQN
    J\partial_s = \widehat{X} \qquad\text{and}\qquad J\colon \xi \to \xi
    \end{align*}
  where \(\hat{\omega}\circ ({\rm Id}\times J)\big|_\xi\) is symmetric and
  positive definite.
There are two issues of note, which make such a choice somewhat
  different from the weakly adapted almost complex structures defined above.
The first such difference is that in general \(\hat{\omega}\) is not
  translation invariant, 
Put another way, this means that \(\widehat{X}\) has \(s\) dependency, or
  rather that \(\widehat{X}(s,p) \neq \widehat{X}(s',p)\) in general for
  \(s\neq s'\).
Because \(\widehat{X}\) is not translation invariant, and in fact in
  general the line bundle \(\mathbb{R} \widehat{X}\subset \hat{\rho}\)
  will fail to be translation invariant, it follows that any almost complex
  structure which preserves \(\hat{\rho}\) will also fail to be translation
  invariant.

It is worth recalling that the framework of SFT
  typically requires a translation invariant almost complex structure (in
  cylindrical homogeneous regions of symplectizations) or else is only
  required to be symplectically tame (in the inhomogenous or cobordant
  regions).
In contrast, Definition \ref{DEF_weakly_adapted_J}  requires the almost
  complex structure to be carefully ``adapted'' in an inhomogeneous region.
Perhaps a more important feature of Definition \ref{DEF_weakly_adapted_J} is the
  ``weakness'' condition, which allows the almost complex structure to
  have the form \(\phi\cdot J\partial_s = \widehat{X}\).
This is best understood as a partial degeneration of the almost complex
  structure which can be undone by a finite amount of
  ``stretching the neck'' along the first factor of
  \(\mathcal{I}_\epsilon\times M\).
More precisely, we define the embedding.   
  \begin{align}                                                           %% EQN
    &\Psi \colon\mathcal{I}_\epsilon \times M  \to \mathbb{R}\times M
    \label{EQ_Psi_new}
    \\
    &\Psi(s, p) = \big(\psi(s), p\big)\notag
    \end{align}
  where
  \begin{align}\label{EQ_psi_new}                                         %% EQN
    \psi(s) = \int_0^s \frac{1}{\phi(t)} dt.
    \end{align}
To see the utility of the map \(\Psi\), we consider an example, which
  already explains its key features with respect to the other relevant data.
Specifically, we consider the case that \(\phi(s) = \delta\) where \(\delta\)
  is some very small positive number, and we assume that \(J_\xi :=
  J\big|_\xi \colon \xi \to \xi\) is translation invariant in the \(s\)
  direction; that is, \(\mathcal{L}_{\partial_s} J_\xi = 0\).
For \(v^\xi\in \xi\) we then have
  \begin{align*}                                                          %% EQN
    J\big(a \partial_s + b \widehat{X} + v^\xi \big)
    =
    -\delta b \partial_s + \frac{a}{\delta} \widehat{X} +
    J_\xi v^\xi.
    \end{align*}
We also see that in this case
  \begin{align*}                                                          %% EQN
    &\Psi\colon (-\epsilon, \epsilon) \times M \to
    \Big(-\frac{\epsilon}{\delta}, \frac{\epsilon}{\delta}\Big)\times M\ \
    (\textrm{diffeomorphism})
    \\ 
    &\Psi(s, p)=(s\delta^{-1} , p)  
    \end{align*}
  and thus
  \begin{align*}                                                          %% EQN
    (\Psi_*J)\big(a \partial_{\check{s}} + b (\Psi_*\widehat{X}) + v^\xi
    \big)
    =
    - b \partial_{\check{s}} + a (\Psi_*\widehat{X}) +
    J_\xi v^\xi,
    \end{align*}
  where \(\check{s}\) is the coordinate on \((-\epsilon \delta^{-1},
  \epsilon \delta^{-1})\).
Note that if we abuse notation by identifying \(\{s_0\}\times M\) with
  \(M\), then
  \begin{align*}                                                          %% EQN
    (\Psi_* \widehat{X})(\check{s},p) = \widehat{X}(\delta \check{s}, p).
    \end{align*}
Similarly, if \(J_\xi = J\big|_\xi\) has \(s\) dependence, then by abusing
  notation again, we have
  \begin{align*}                                                          %% EQN
    (\Psi_* J_\xi)(\check{s} , p) = J_\xi(\delta \check{s}, p).
    \end{align*}
Put another way, because the tangent bundle splits as \(\hat{\rho}\oplus
  \xi\), and because the almost complex structures preserve this
  splitting, and the stretching direction is contained in \(\hat{\rho}\),
  it follows that this ``neck stretching'' does not degenerate the
  restriction \(J_\xi\), and it is chosen to stretch the compressed \(J\)
  into something more well behaved.
For example, \((\Psi_*J)\partial_{\check{s}}= \Psi_*\widehat{X}\).
It is left to the reader to observe that this qualitative behavior is
  preserved when one changes from the example case of \(\phi(s)=\delta\)
  to the more general case that \(\phi\colon \mathcal{I}_\epsilon\times
  M\to (0, 1]\) with \(\phi(s,p)=\phi(s)\).

We will elaborate further on this stretching construction below, however
  it will be helpful to put it in the context of ``adiabatic degeneration'',
  which amounts to fully neck-stretching along a continuum of energy
  levels simultaneously.
We make this precise with the following.

%%%%%%%%%%%%%%%%%%%%%%%%%%%%%%%%%%%%%%%%%%%%%%%%%%%%%%%%%%%%%%%%%%%%%%%%%%%%%%%%
%%%%%%%%%%                          DEFINITION                         %%%%%%%%%
%%%%%                                                                       %%%%
\begin{definition}[adiabatically degenerating almost complex structures]
  \label{DEF_adiabatically_degenerating_almost_complex_structures}
  \hfill\\
Let \(\big(\mathcal{I}_\epsilon\times M, \Omega, \hat{\lambda}\big)\) be a
  framed Hamiltonian energy pile, and let \(\hat{\rho}\), \(\xi\), and
  \(\hat{\omega}\) be the associated structures defined above, and let 
  \(\widehat{X}=X_H/\hat{\lambda}(X_H)\) as above.
Let \(\mathcal{J}=\{J_k\}_{k\in \mathbb{N}}\) be a sequence of weakly
  adapted almost complex structures in the sense of Definition
  \ref{DEF_weakly_adapted_J}, with \(\phi_k \cdot J_k \partial_s =
  \widehat{X}\), which also satisfies the following conditions.
\begin{enumerate}[($\mathcal{J}$1)]                                       %% NUM
  \item 
  \emph{({Adiabatic Degeneration})}
  We require that \(\phi_k\to 0\) in \(C^\infty\) as \(k\to
  \infty\).
  \item 
  \emph{({Symplectic Area Controls Metric Area})}
  For each \(k\in \{1, 2, 3, \ldots\}\), and each  \(v\in
  T(\mathcal{I}_\epsilon\times M)\), we require that
  \begin{align*}                                                          %% EQN
    (ds \wedge \hat{\lambda} + \hat{\omega})(v, J_kv )\leq 2 \Omega(v ,
    J_k v)
    \end{align*}
  \item 
  \emph{({Geometrically Bounded})}
  There exists a sequence \(\{C_n\}_{n=0}^\infty\) of positive
  constants and an auxiliary translation invariant metric \(\tilde{g}\) on
  \(\mathbb{R}\times M\) for which
  \begin{align*}                                                          %% EQN
    \sup_{k\in \mathbb{N}} \| (\Psi_k)_* J_k \|_{C^n} \leq C_n
    \end{align*}
  for each \(n\in \mathbb{N}\); here \(\|\cdot \|_{C^n}\) is the
  \(C^n\)-norm on \(\mathbb{R}\times M\) with respect to the
  auxiliary metric \(\tilde{g}\), and the \(\Psi_k\) are the embeddings as
  in equation (\ref{EQ_Psi_new}).
  \end{enumerate}
We then say that \(\mathcal{J}=\{J_1, J_2, \ldots\}\) is a sequence of
  adiabatically degenerating almost complex structures adapted to the
  framed Hamiltonian energy pile given by \((\mathcal{I}_\epsilon\times M,
  \Omega, \hat{\lambda})\).
\end{definition}

We note that given a framed Hamiltonian energy pile, it is easy to
  construct a candidate sequence of adiabatically degenerating almost
  complex structures.
For example, one might fix a translation invariant \(J_\xi = J\big|_\xi\),
  and then define
  \begin{align*}                                                          %% EQN
    J_k \partial_s = k \widehat{X}\qquad\text{and}\qquad J_k\big|_\xi  =
    J_\xi.
    \end{align*}
In such a candidate, one can easily see that \(\phi_k = k^{-1}\to 0\) as
  \(k\to \infty\), and in fact \((\Psi_k)_* J_k\) is independent of \(k\).
Consequently (\(\mathcal{J}1\)) and (\(\mathcal{J}3\)) are immediately
  satisfied.
However, (\(\mathcal{J}2\)) is less obvious.
That is, it is unclear if the following estimate holds for all \(J_k\):
  \begin{align*}                                                          %% EQN
    (ds \wedge \hat{\lambda} + \hat{\omega})(v, J_kv )\leq 2 \Omega(v ,
    J_k v)
    \end{align*}
Geometrically, the concern is that while the almost complex structures
  \(J_k\) are all \(\Omega\)-tame, they are not \(\Omega\)-compatible.
Put another way, although the \(J_k\) preserve the splitting
  \(T(\mathcal{I}_\epsilon\times M) = \hat{\rho}\oplus \xi\), and
  \(\hat{\rho}\) and \(\xi\) are each symplectic subspaces, these
  subspaces  are not symplectic complements; that is there exist \(v\in
  \hat{\rho}\)
  and \(w\in \xi\) such that \(\Omega(v,w)\neq 0\).
That the cross terms might cause an issue is then compounded by
  the fact that we are degenerating adiabatically so that \(\phi_k\to 0\)
  as \(k\to \infty\).

We resolve this issue by making use of the fact that Lemma
  \ref{LEM_localization_to_framed_Hamiltonian_energy_pile} allows us to
  guarantee that along \(\{0\}\times M\) the sub-bundles \(\hat{\rho}\) and
  \(\xi\) \emph{are} \(\Omega\)-symplectic complements.
As it will turn out, condition (\(\mathcal{J}\)2) can then be achieved by
  first fixing \(\epsilon>0\) sufficiently small.
We will couple this with a convenient means of obtaining a sequence of
  adiabatically degenerating almost complex structures from some easily
  verified bounds and some \(C^\infty\) converging functions
  \(\phi_k\).
This is achieved via Proposition
  \ref{PROP_adiabatically_degenerating_constructions} below.

\setcounter{CounterSectionAdiabatic}{\value{section}}
\setcounter{CounterPropositionAdiabatic}{\value{lemma}}
%%%%%%%%%%%%%%%%%%%%%%%%%%%%%%%%%%%%%%%%%%%%%%%%%%%%%%%%%%%%%%%%%%%%%%%%%%%%%%%%
%%%%%%%%%%                         PROPOSITION                         %%%%%%%%%
%%%%%                                                                       %%%%
\begin{proposition}[adiabatically degenerating constructions]
  \label{PROP_adiabatically_degenerating_constructions}
  \hfill\\
Let \(\big(\mathcal{I}_{\ell}\times M, \Omega, \hat{\lambda}\big)\) be a
  framed Hamiltonian energy pile, and let \(\hat{\rho}\), \(\xi\), and
  \(\hat{\omega}\) be the associated structures defined above, and assume,
  as in the conclusions of Lemma
  \ref{LEM_localization_to_framed_Hamiltonian_energy_pile}, that along
  \(\{0\}\times M\) the sub-bundles \(\hat{\rho}\) and \(\xi\) are
  symplectic complements and \(\hat{\lambda}(X_H)=1\).
Let \(\widehat{X}=X_H/\hat{\lambda}(X_H)\) as above.
Let \(\{\check{J}_k\}_{k\in \mathbb{N}}\) be a sequence of almost
  complex structures on \(\mathcal{I}_\ell \times M\) which satisfy the
  following conditions.
\begin{enumerate}[(D1)]                                                   %% NUM
  \item 
  \(\check{J}_k\colon \hat{\rho} \to \hat{\rho}\; \) and  \(\;
  \check{J}_k\colon \xi \to \xi\) for each \(k\in \mathbb{N}\),
  \item 
  \((ds\wedge \hat{\lambda} +\hat{\omega})\circ ({\rm Id}\times
  \check{J}_k)\) is a Riemannian metric for each \(k\in \mathbb{N}\)
  \item 
  there exist constants  \(\{C_n'\}_{n\in \mathbb{N}}\) such that
  \begin{align*}                                                          %% EQN
    \sup_{k\in \mathbb{N}} \|\check{J}_k\|_{C^n} \leq C_n'
    \end{align*}
  \end{enumerate}
Then there exist an \(\epsilon>0\) with the following significance.
For any sequence of functions \(\phi_k : (-\epsilon, \epsilon)\to (0, 1]\)
  which converge in \(C^\infty\) (though not necessarily to
  zero), the weakly adapted almost complex structures defined by
  \begin{align}\label{EQ_def_Jk_in_body}                                  %% EQN
    \phi_k \cdot J_k \partial_s = \widehat{X}\qquad\text{and}\qquad
    J_k\big|_\xi := \check{J}_k\big|_\xi
    \end{align}
  satisfy the following properties
\begin{enumerate}[(E1)]                                                   %% NUM
  \item For each \(v\in T(\mathcal{I}_{\epsilon}\times M)\) we have
  \begin{align*}                                                          %% EQN
    (ds \wedge \hat{\lambda} + \hat{\omega})(v, J_kv )\leq 2 \Omega(v ,
    J_k v).
    \end{align*}
  \item 
  There exists a sequence \(\{K_n\}_{n=0}^\infty\) of positive
  constants and an auxiliary translation invariant metric \(\tilde{g}\) on
  \(\mathbb{R}\times M\) for which
  \begin{align*}                                                          %% EQN
    \sup_{k\in \mathbb{N}} \| (\Psi_k)_* J_k \|_{C^n} \leq K_n
    \end{align*}
  for each \(n\in \mathbb{N}\); here \(\|\cdot \|_{C^n}\) is the
  \(C^n\)-norm on \(\mathbb{R}\times M\) with respect to the
  auxiliary translation invariant metric \(\tilde{g}\), and the \(\Psi_k\)
  are the embeddings as in equation (\ref{EQ_Psi_new}).
  \end{enumerate}
In particular, if \(\epsilon' \in (0, \epsilon)\) and
  \(\phi_k\big|_{(-\epsilon', \epsilon')}\to 0\), then the almost complex
  structures \(J_k\) are adiabatically degenerating on
  \(\mathcal{I}_{\epsilon'}\times M\) in the sense of Definition
  \ref{DEF_adiabatically_degenerating_almost_complex_structures}, and
  on \(\mathcal{I}_{\epsilon}\times M\) these almost complex structures are
  all tame.
\end{proposition}
%%%%%                                                                       %%%%
%%%%%%%%%%                                                             %%%%%%%%%
%%%%%%%%%%%%%%%%%%%%%%%%%%%%%%%%%%%%%%%%%%%%%%%%%%%%%%%%%%%%%%%%%%%%%%%%%%%%%%%%
%

The proof of Proposition
  \ref{PROP_adiabatically_degenerating_constructions} is elementary but
  somewhat lengthy, so we relegate it to Appendix
  \ref{SEC_supporting_proofs}, however at present we indicate its utility.
We start by noting that a natural way to perform an adiabatic degeneration
  is to start with some fixed almost complex structure \(J\) which
  preserves the splitting \(\hat{\rho}\oplus \xi\) and for which \((ds
  \wedge \hat{\lambda} + \hat{\omega})\circ ({\rm Id} \times J)\) is a
  Riemannian metric.
This almost complex structure will then be \(\Omega\)-compatible along
  \(\{0\}\times M\), and hence \(\Omega\)-tame in some neighborhood of
  \(\{0\}\times M\).
Consequently, \(J\) can be adjusted away from \(\{0\}\times M\) so that it
  is \(\Omega\)-tame everywhere, and still satisfies our compatibility
  conditions in a neighborhood of \(\{0\}\times M\).
Proposition \ref{PROP_adiabatically_degenerating_constructions} then
  guarantees that we can find \(\phi_k\) which tend to zero in a
  neighborhood of zero so that the sequence of almost complex structures
  given by
  \begin{align*}                                                          %% EQN
    \phi_k \cdot J_k\partial_s = \widehat{X}\qquad\text{and}\qquad
    J_k\big|_\xi = J\big|_\xi
    \end{align*}
  are everywhere tame, while also adiabatically degenerating a
  neighborhood  of \(\{0\}\times M\).
Among other things, this means
  \begin{align*}                                                          %% EQN
    (ds \wedge \hat{\lambda} + \hat{\omega})(v, J_kv )\leq 2 \Omega(v ,
    J_k v)
    \end{align*}
  for all \(k\in \mathbb{N}\).
This latter condition is important, since by assumption the
  pseudoholomorphic curves we will study will have uniform bounds on
  \(\Omega\)-energy, which (as a consequence of the above inequality) will
  give us bounds on the area associated to the metric \((ds\wedge
  \hat{\lambda} +\hat{\omega})\circ ({\rm Id}\times J)\).

With Proposition \ref{PROP_adiabatically_degenerating_constructions}
  stated and its use outlined, we now turn our attention to proving our
  main result.
This is the topic of the next section.

\section{Proof of the Main Result}\label{SEC_main_proof}

In this section we prove Theorem \ref{THM_main_maybe}, which will be
  accomplished as follows.
In light of Lemma
  \ref{LEM_localization_to_framed_Hamiltonian_energy_pile} and Proposition
  \ref{PROP_adiabatically_degenerating_constructions}, we see that the
  general problem can essentially be reduced to a localized problem
  involving framed Hamiltonian energy piles and adiabatically degenerating
  almost complex structures.
As such, our first task is to prove Theorem
  \ref{THM_localized_almost_existence} below, which is essentially a
  localized version of Theorem \ref{THM_main_maybe}.
The proof of Theorem \ref{THM_localized_almost_existence} is somewhat
  technical and will take the bulk of this section.
After this is established, Theorem \ref{THM_main_maybe} will follow rather
  quickly. 
We now proceed with our first main technical argument.

%%%%%%%%%%%%%%%%%%%%%%%%%%%%%%%%%%%%%%%%%%%%%%%%%%%%%%%%%%%%%%%%%%%%%%%%%%%%%%%%
%%%%%%%%%%                           THEOREM                           %%%%%%%%%
%%%%%                                                                       %%%%
\begin{theorem}[localized almost existence]
  \label{THM_localized_almost_existence}
  \hfill\\
Let \(\big(\mathcal{I}_\epsilon\times M, \Omega, \hat{\lambda}\big)\) be a
  framed Hamiltonian energy pile, and let \(\hat{\rho}\), \(\xi\), and
  \(\hat{\omega}\) be the associated structures.
Let \(\{J_k\}_{k\in \mathbb{N}}\) be a sequence of adiabatically
  degenerating almost complex structures in the sense of Definition
  \ref{DEF_adiabatically_degenerating_almost_complex_structures}.
Suppose that for each \(k\in \mathbb{N}\) there exists a proper
  pseudoholomorphic map \(u_k:(S_k, j_k) \to (\mathcal{I}_\epsilon\times
  M, J_k)\) without boundary such  that \(s\circ u_k(S_k)
  = \mathcal{I}_\epsilon\),  with the property that there exist no
  connected components of \(S_k\) on which \(u_k\) is the constant map.
Suppose further that there exist positive constants \(C_g\) and
  \(C_\Omega\) for which
  \begin{align*}                                                          %% EQN
    {\rm Genus}(S_k) \leq C_g\qquad\text{and}\qquad \int_{S_k} u_k^*\Omega
    \leq C_\Omega.
    \end{align*}
\emph{Then} for almost every point \(s\in \mathcal{I}_\epsilon\), there
  exists a periodic orbit of the Hamiltonian vector field \(X_H\) on  the
  energy level \(\{s\}\times M\).
That is, the set \(\mathcal{I}_\epsilon'\subset \mathcal{I}_\epsilon\) of
  energy levels of \(H(s,p)=s\) which contain a Hamiltonian periodic orbit
  has full measure:
  \begin{align*}                                                          %% EQN
    \mu(\mathcal{I}_\epsilon') = \mu(\mathcal{I}_\epsilon) = 2\epsilon.
    \end{align*}
\end{theorem}
%%%%%                                                                       %%%%
%%%%%%%%%%                                                             %%%%%%%%%
%%%%%%%%%%%%%%%%%%%%%%%%%%%%%%%%%%%%%%%%%%%%%%%%%%%%%%%%%%%%%%%%%%%%%%%%%%%%%%%%
%

In order to begin, we will first need to recall a version of the co-area
  formula as follows.

%%%%%%%%%%%%%%%%%%%%%%%%%%%%%%%%%%%%%%%%%%%%%%%%%%%%%%%%%%%%%%%%%%%%%%%%%%%%%%%%
%%%%%%%%%%                          PROPOSITION                        %%%%%%%%%
%%%%%                                                                       %%%%
\begin{proposition}[The co-area formula]\label{PROP_co-area}\hfill\\
Let \((S, g)\) be an oriented  \(C^1\)-Riemannian manifold of
  dimension two; we allow that \(S\) need not be complete\footnote{That is,
  there may exist Cauchy sequences, with respect to \(g\), which do not
  converge in \(S\).}.
Suppose that \(\beta:S\to [a, b]\subset\mathbb{R}\) is a \(C^1\)
  function without critical points.
Let \(f:S\to [0, \infty)\) be a measurable function  with respect to
  \(d\mu_g^2\).  
Then
  \begin{equation}\label{EQ_coarea}
    \int_S f\|\nabla \beta\|_g \, d\mu_g^2 = \int_a^b \Big(
    \int_{\beta^{-1}(t)} f \, d\mu_g^1\Big) dt
    \end{equation}
  where \(\nabla\beta\) is the gradient of \(\beta\) computed with respect
  to the metric \(g\).
\end{proposition}
%%%%%                                                                       %%%%
%%%%%%%%%%                                                             %%%%%%%%%
%%%%%%%%%%%%%%%%%%%%%%%%%%%%%%%%%%%%%%%%%%%%%%%%%%%%%%%%%%%%%%%%%%%%%%%%%%%%%%%%
%
\begin{proof}
This is a well known result, however the details of this specific version
  are provided in Appendix A.2 of \cite{FH2}.
\end{proof}%%%%%%%%%%%%%%%%%%%%%%%%%%%%%%%%%%%%%%%%%%%%%%%%            END PROOF

The co-area formula above will be used in a rather particular way, namely
  as a means of expressing \(\int_S u^*(ds\wedge \hat{\lambda})\) as a double
  integral.
More precisely, we have the following.

%%%%%%%%%%%%%%%%%%%%%%%%%%%%%%%%%%%%%%%%%%%%%%%%%%%%%%%%%%%%%%%%%%%%%%%%%%%%%%%%
%%%%%%%%%%                            LEMMA                            %%%%%%%%%
%%%%%                                                                       %%%%
\begin{lemma}[co-area application]
  \label{LEM_coarea_application}
  \hfill\\
Let \(\big(\mathcal{I}_\epsilon\times M, \Omega, \hat{\lambda}\big)\) be a
  framed Hamiltonian energy pile, and let \(J\) be a weakly
  adapted almost complex structure in the sense of Definition
  \ref{DEF_weakly_adapted_J}.
Let \((u, S, j)\) be a \(J\)-holomorphic curve in \(\mathcal{I}_\epsilon
  \times M\) for which \(\partial S = \emptyset\).
\emph{Then}
  \begin{equation}\label{EQ_coarea_lambda}
    \int_{S}u^*(ds\wedge \hat{\lambda}) =\int_{\mathcal{I}_\epsilon}
    \Big( \int_{(s\circ u)^{-1}(t)\setminus \mathcal{X}}
    u^*\hat{\lambda}\Big)\, dt,
    \end{equation}
  where  \(\mathcal{X}:=\{\zeta\in S: d(s\circ u)_\zeta = 0\}\); that is,
  \(\mathcal{X}\) is the set of critical points of the function \(s\circ
  u: S \to \mathcal{I}_\epsilon \subset \mathbb{R}\).
\end{lemma}
%%%%%                                                                       %%%%
%%%%%%%%%%                                                             %%%%%%%%%
%%%%%%%%%%%%%%%%%%%%%%%%%%%%%%%%%%%%%%%%%%%%%%%%%%%%%%%%%%%%%%%%%%%%%%%%%%%%%%%%
%
\begin{proof}
Define \(\widetilde{S}:=S\setminus\mathcal{X}\), which is a manifold
  since \(\widetilde{S}\subset S\) is open.
Observe that by definition of \(\mathcal{X}\) it follows that \(u^*(ds
  \wedge \hat{\lambda}) \big|_{\mathcal{X}}\equiv 0\), so
    \begin{align*}                                                        %% EQN
    \int_{S}u^*(ds \wedge \hat{\lambda})= \int_{\widetilde{S}}
	u^*(ds \wedge \hat{\lambda}).
    \end{align*}
Since \(u:\widetilde{S}\to \mathcal{I}_\epsilon\times M\) is an immersion,
  we may define the metric \(\gamma=u^*g_J\) where \(g_J\) is the
  Riemannian metric as in equation (\ref{EQ_Riemannian_metric}) in
  Definition \ref{DEF_weakly_adapted_J}; note that \(J\) is a
  \(g_J\)-isometry.
The almost complex structure \(j\) on \(S\) induces an orientation on
  \(\widetilde{S}\), and hence we have
  \begin{equation}\label{EQ_int_a_lambda}
      \int_{\widetilde{S}} u^*(ds \wedge \hat{\lambda})=
      \int_{\widetilde{S}} u^*(ds\wedge \hat{\lambda}\big)(\nu, \tau)
      d\mu_{\gamma}^2,
    \end{equation}
  where \((\nu, \tau)\) is a positively oriented \(\gamma\)-orthonormal
  frame field, and \(d\mu_{\gamma}^2\) is the volume form on
  \(\widetilde{S}\) associated to the metric \(\gamma\).  
This observation is elementary, however details are provided in Appendix
  A.2 of \cite{FH2}.
Note that equation (\ref{EQ_int_a_lambda}) holds for arbitrary orthonormal
  frame field \((\nu, \tau)\), however we shall  henceforth make use of the
  following particular frame.
\begin{equation}\label{EQ_nu_tau}
    \nu:= \frac{\nabla (s\circ u)}{\|\nabla (s\circ
    u)\|_{\gamma}}\qquad\text{and}\qquad \tau:= j \nu.
  \end{equation}
Because \(u:S\to \mathcal{I}_\epsilon\) is a \(J\)-holomorphic map and
  \(J\) is a \(g_J\)-isometry, it follows that \(j\) is a
  \(\gamma\)-isometry.
Also note that for each \(v^\xi\in \xi\) and \(a,b\in \mathbb{R}\), we have
  \begin{align*}                                                          %% EQN
    J\big(a \partial_s  + b\widehat{X} +v^\xi\big) 
    =   
    \frac{a}{\phi} \widehat{X}   - \phi b\partial_s+J v^\xi
    \end{align*}
  with \(Jv^\xi \in \xi\), and hence
  \begin{align*}                                                          %% EQN
    ds\circ J = -\phi \hat{\lambda} 
    \qquad\text{and}\qquad
    \hat{\lambda}\circ J =  \frac{1}{\phi} ds.
    \end{align*}
With \(\nu\) and \(\tau\) as in equation (\ref{EQ_nu_tau}), it is then
  straightforward to verify the following.
\begin{align*}
    0 &= (u^*\hat{\lambda})(\nu) = u^*ds(\tau) \\
    0 &< u^*\big({\textstyle \frac{1}{\phi}}ds\big) (\nu) =
    (u^*\hat{\lambda})(\tau)\\
    1 &= \|\tau\|_{\gamma}^2 =\|\nu\|_{\gamma}^2 
  \end{align*}
Also,
  \begin{equation}\label{EQ_gradient_relationship}
      \|\nabla (s\circ u)\|_{\gamma} = \sup_{\substack{ x\in T_\zeta S\\
      \|x\|_{\gamma}=1}} d(s\circ u)(x) =\sup_{\substack{ x\in T_\zeta
      S\\ \|x\|_{\gamma}=1}} ds(Tu\cdot x) =u^*ds(\nu),
    \end{equation}
  and
  \begin{equation}\label{EQ_tau_nu_2}
      (\phi\circ u)\, \|u^*\hat{\lambda} \|_{\gamma} =u^*(\phi
      \,\hat{\lambda})(\tau)=-u^*ds (jj\nu)=\|\nabla(s\circ u)\|_{\gamma}.
    \end{equation} 

With \((\nu, \tau)\) defined as such, we have the following.  
\begin{align*}
    \int_{\widetilde{S}} u^*(ds\wedge \hat{\lambda})(\nu, \tau)
      d\mu_{\gamma}^2
    &=\int_{\widetilde{S}} ds(Tu\cdot\nu )  \hat{\lambda}(Tu \cdot \tau)
      d\mu_{\gamma}^2\\
    &=\int_{\widetilde{S}}{\textstyle\frac{1}{\phi\circ u}} \|\nabla
      (s\circ u)\|_{\gamma}^2 d\mu_{\gamma}^2\\
  \end{align*}

We employ Lemma \ref{LEM_coarea_application} on \(\widetilde{S}\) with
  \(\beta=s\circ u\), and \(f=\frac{1}{\phi\circ u}\|\nabla
  (s\circ u)\|_{\gamma}\) to obtain
  \begin{align*}
      \int_{\widetilde{S}}{\textstyle \frac{1}{\phi\circ u}}\|\nabla
        (s\circ u)\|_{\gamma}^2 d\mu_{\gamma}^2
      &=\int_{\mathcal{I}_\epsilon}\Big( \int_{(s\circ u)^{-1}(t)\setminus
	\mathcal{X}} {\textstyle \frac{1}{\phi \circ u}} \|\nabla (s\circ
	u)\|_{\gamma}\, d\mu_{\gamma}^1\Big) dt\\
      &= \int_{\mathcal{I}_\epsilon} \Big( \int_{(s\circ u)^{-1}(t)\setminus
	\mathcal{X} }  (u^*\hat{\lambda})(\tau)\, d\mu_{\gamma}^1\Big) dt\\
      &=\int_{\mathcal{I}_\epsilon} \Big( \int_{(s\circ u)^{-1}(t)\setminus
	\mathcal{X}} u^*\hat{\lambda} \Big) dt,
    \end{align*}
  and hence by combining equalities we have
  \begin{equation*}
      \int_{S}u^*(ds \wedge \hat{\lambda})=\int_{\mathcal{I}_\epsilon} \Big(
      \int_{(s\circ u)^{-1}(t)\setminus \mathcal{X}}
      u^*\hat{\lambda}\Big) dt,
    \end{equation*}
  which establishes equation (\ref{EQ_coarea_lambda}).
\end{proof}%%%%%%%%%%%%%%%%%%%%%%%%%%%%%%%%%%%%%%%%%%%%%%%%            END PROOF

With these preliminaries established, our next main task is to carefully
  pass to a certain subsequence of our almost complex structures and
  pseudoholomorphic curves.
To that end, we first recall that \(\phi_k\cdot J_k \partial_s =
  \widehat{X}=X_H/\hat{\lambda}(X_H)\)
  where \(\{\phi_k\}_{k\in \mathbb{N}}\) is a sequence of positive functions
  converging to zero in \(C^\infty\).
Next we define the following intervals.
For each \(t\in \mathcal{I}_\epsilon\), and for each sufficiently large
  \(k\in \mathbb{N}\) define the open interval
  \begin{align*}                                                          %% EQN
    \mathcal{I}(t,k) = (t_0, t_1) = \{a \in \mathbb{R} : t_0 < a < t_1\}
    \end{align*}
  where
  \begin{align*}                                                          %% EQN
    \int_t^{t_1} \frac{1}{\phi_k(s)} ds = 1 = \int_{t_0}^{t}
    \frac{1}{\phi_k(s)} ds.
    \end{align*}
Introduce the map
  \begin{align*}                                                          %% EQN
    &{\rm Sh}_c \colon \mathbb{R}\times M \to \mathbb{R}\times M
    \\
    &{\rm Sh}_c (s, p) = (s-c, p).
    \end{align*}

%%%%%%%%%%%%%%%%%%%%%%%%%%%%%%%%%%%%%%%%%%%%%%%%%%%%%%%%%%%%%%%%%%%%%%%%%%%%%%%%
%%%%%%%%%%                            LEMMA                            %%%%%%%%%
%%%%%                                                                       %%%%
\begin{lemma}[the $\Psi_k^t$ are diffeomorphisms]
  \label{LEM_Psi_are_diffeos}
  \hfill\\
We introduce the maps $\Psi^t_k$ defined by 
  \begin{align}\label{EQ_Psi_k_t}                                         %% EQN
    \Psi_k^t = {\rm Sh}_{\psi_k(t)}\circ \Psi_k
    \end{align}
  where  \(\Psi_k\) and \(\psi_k\) are respectively the maps given in
  equations \emph{(\ref{EQ_Psi_new})} and \emph{(\ref{EQ_psi_new})}.
Then $\Psi^t_k$ induces a diffeomorphism
  \begin{align*}                                                          %% EQN
    \Psi_k^t\colon  \mathcal{I}(t,k)\times M \rightarrow (-1,1)\times M.
    \end{align*}
  of the form $(s,p)\rightarrow (g_k(s),p)$ with \(\Psi^t_k(t,p)=(0,p)\).
\end{lemma}
%%%%%                                                                       %%%%
%%%%%%%%%%                                                             %%%%%%%%%
%%%%%%%%%%%%%%%%%%%%%%%%%%%%%%%%%%%%%%%%%%%%%%%%%%%%%%%%%%%%%%%%%%%%%%%%%%%%%%%%
%
\begin{proof}
By definition 
  \[
  \Psi_k(s,p)=\left(\int_0^s\frac{1}{\phi_k(\tau)}d\tau,p\right) =
  \left(\int_0^t\frac{1}{\phi_k(\tau)}d\tau
  +\int_t^s\frac{1}{\phi_k(\tau)}d\tau,p\right).
  \]
  Hence 
  \[
  \Psi_k(s,p)=\left(\psi_k(t) +\int_t^s\frac{1}{\phi_k(\tau)}d\tau,p\right)
  \]
  which implies
  \[
 \Psi^t_k(s,p)= {\rm
 Sh}_{\psi_k(t)}\circ\Psi_k(s,p)=
 \left(\int_t^s\frac{1}{\phi_k(\tau)}d\tau,p\right).
  \]
It follows immediately that $\Psi^t_k$ has image $(-1,1)\times M$ and is
  a diffeomorphism.
\end{proof}%%%%%%%%%%%%%%%%%%%%%%%%%%%%%%%%%%%%%%%%%%%%%%%%            END PROOF

As we have seen,  when \(k\) is large, \(\phi_k\) is close to zero, and    
  \(\Psi_k\) then ``stretches the neck'' to undo the partial degeneration
  done by \(\phi_k\) (via the condition that \(\phi_k \cdot J_k \partial_s
  = \widehat{X}\)); the shift function, \({\rm Sh}\), then
  \(\mathbb{R}\)-shifts in the target to put the image of \(t\in
  \mathcal{I}(t,k)\) at \(0\in (-1, 1)\).
More concisely then, \(\mathcal{I}(t,k)\) is an open neighborhood of
  \(t\in \mathcal{I}_\epsilon\), which when neck-stretched and shifted 
  becomes the standard interval \((-1,1)\).

  Under the hypotheses of Theorem \ref{THM_localized_almost_existence}  %3
we are given a sequence of pseudoholomorphic curves 
\[
u_k\colon (S_k,j_k)\rightarrow (\mathcal{I}_{\varepsilon}\times M,J_k)
\]
where we have a genus bound $C_g$ and a symplectic bound:
\begin{eqnarray}\label{EQN13}
g(S_k)\leq C_g\ \ \textrm{and}\ \  \int_{S_k} u^{\ast}_k\Omega\leq C_{\Omega}.
\end{eqnarray}
Note that we obtain from this, because each of the two-forms \(ds\wedge
  \hat{\lambda}\) and \(\hat{\omega}\) evaluate non-negatively on
  \(J_k\)-complex lines, the following estimate:
  \begin{eqnarray}  \label{EQN133}                                                   %% EQN
 & \int_{S_k} u^*_k(ds\wedge\hat{\lambda})+  \int_{S_k} u_k^*
 \hat{\omega}=\int_{S_k} u_k^*(ds \wedge \hat{\lambda}+ \hat{\omega})
    \leq 2\int_{S_k} u_k^*\Omega   \leq 2C_\Omega.&
    \end{eqnarray}
Note that the first inequality follows from the assumption that the
  \(J_k\) are adiabatically degenerating; (see Definition
  \ref{DEF_adiabatically_degenerating_almost_complex_structures} and more
  specifically condition \(\mathcal{J}\)2).
Define $\bar{C}_{\Omega}:=2\cdot C_{\Omega}$ so that 
\begin{eqnarray}\label{EQN136}
  \int_{S_k} u_k^* \hat{\omega}\leq\bar{C}_{\Omega}\ \ \text{and}\ \
  \int_{S_k} u^*_k(ds\wedge\hat{\lambda})\leq \bar{C}_{\Omega}.
  \end{eqnarray} 
Our goal is the rather careful construction of  subsequences 
  with a certain number of good properties. 
To that end, it will be helpful to recall that if \(\{x_k\}_{k\in
  \mathbb{N}}\) is a sequence of points, then a subsequence can be
  specified using a strictly increasing function \(f\colon\mathbb{N}\to
  \mathbb{N}\) by writing \(\{x_{f(k)}\}_{k\in \mathbb{N}}\).
A further subsequence can be defined using a strictly increasing function
  \(h\colon \mathbb{N}\to f(\mathbb{N})\) giving \(\{x_{h (k)}\}_{k\in
  \mathbb{N}}\).
The following  result is at the heart of our further constructions.
We shall write ${\mathbb N}_0$ for the union of ${\mathbb N}$ and $\{0\}$.

%%%%%%%%%%%%%%%%%%%%%%%%%%%%%%%%%%%%%%%%%%%%%%%%%%%%%%%%%%%%%%%%%%%%%%%%%%%%%%%%
%%%%%%%%%%                         PROPOSITION                         %%%%%%%%%
%%%%%                                                                       %%%%
\begin{proposition}[the key inductive construction]
  \label{PROP_inductive_construction}
  \hfill\\
Given a sequence of pseudoholomorphic curves 
\[
u_k\colon (S_k,j_k)\rightarrow (\mathcal{I}_{\varepsilon}\times M,J_k)
\]
satisfying the bound  {\emph{(\ref{EQN13})}} there exists a  sequence
${(f_m)}_{m\in {\mathbb N}_0}$ of strictly increasing maps $f_m:{\mathbb
N}\rightarrow {\mathbb N}$ and a sequence ${(L_m)}_{m\in {\mathbb N}_0}$
of finite sets $L_m\subset \emph{\textrm{cl}}(\mathcal{I}_{\varepsilon})$
with the following properties.
\begin{itemize}
  \item[(1)] $L_0=\emptyset$ and $f_0(k)=k$ for all $k\in {\mathbb N}$.
  \item[(2)] $L_{m-1}\subset L_m$ for $m\in {\mathbb N}$
  \item[(3)] $f_m:{\mathbb N}\rightarrow f_{m-1}({\mathbb N})$ and
  $L_{m-1}\subset L_m$ for $m\in {\mathbb N}$.
  \item[(4)] For $m\in {\mathbb N}_0$ the following holds. Given $t\in
  L_m\cap \mathcal{I}_{\varepsilon}$ there exists a sequence
  $(\tau_k)\subset \mathcal{I}_{\varepsilon}$ converging to $t$ such that
  the following limit exists and satisfies the inequality
  \[
  \lim_{k\rightarrow\infty} \int_{{(s\circ
  u_{f_m(k)})}^{-1}(\mathcal{I}(\tau_k,f_m(k)))}
  u^{\ast}_{f_m(k)}\hat{\omega}> \frac{\bar{C}_{\Omega}}{2^m}.
  \]
  \item[(5)] For $m\in {\mathbb N}_0$ and given $t\in
  \mathcal{I}_{\varepsilon}\setminus L_{m}$ there exists no sequence
  $\tau_k\rightarrow t$ such that
  \[
  \limsup_{k\rightarrow \infty}  \int_{{(s\circ
  u_{f_m(k)})}^{-1}(\mathcal{I}(\tau_k,f_m(k)))}
  u^{\ast}_{f_m(k)}\hat{\omega}>  \frac{\bar{C}_{\Omega}}{2^m}
  \]
  \end{itemize}
\end{proposition}
%%%%%                                                                       %%%%
%%%%%%%%%%                                                             %%%%%%%%%
%%%%%%%%%%%%%%%%%%%%%%%%%%%%%%%%%%%%%%%%%%%%%%%%%%%%%%%%%%%%%%%%%%%%%%%%%%%%%%%%
%
\begin{proof}
We now begin an inductive  process of constructing a sequence of nested
  subsequences.
The start of the inductive construction is obvious. We define
  $L_0:=\emptyset$ and $f_0:{\mathbb N}\rightarrow {\mathbb N}$ by
  $f_0(k)=k$ for $k\in {\mathbb N}$. Then Item (1) holds and for $m=0$
  also Item (4) is trivially satisfied since $L_0\cap
  \mathcal{I}_{\varepsilon}=\emptyset$.
Again for $m=0$ we see that $\mathcal{I}_{\varepsilon}\setminus
  L_0=\mathcal{I}_{\varepsilon}$.
If take an element $t\in \mathcal{I}_{\varepsilon}$ we see that the
  symplectic bound (\ref{EQN136}) implies that there is no sequence
  $\tau_k\rightarrow t$ with the property (5).
The statements of Items (2) and (3) are only relevant for \(m\geq 1\).
Hence, with the choices we have made all relevant statements hold for
  \(m=0\).

Let us assume that for some $m\in {\mathbb N}$ we have carried out  the
  constructions of \(L_i\) and \(f_i\) for \(i= 0,...,m-1\) so that our
  (relevant) statements in Items (1)--(5) hold for \(i=1,...,m-1\).
We shall now construct $f_m$ and $L_m$.
Specifically we note that the construction will be made so that
  \begin{align*}                                                          %% EQN
    f_m\colon\mathbb{N} \to f_{m-1}(\mathbb{N})\qquad\text{and}\qquad
    L_{m-1} \subset L_m.
    \end{align*}
Since $\bar{C}_{\Omega}/2^m <\bar{C}_{\Omega}/2^{m-1}$ it follows from
  Item (4) (for the case $m-1$) that for every $t\in
  L_{m-1}\cap\mathcal{I}_{\varepsilon}$ there exists a sequence $
  \tau_k\rightarrow t$ such that the following limit exists and satisfies
  the given inequality
  \[
  \lim_{k\rightarrow\infty} \int_{{(s\circ
  u_{f_{m-1}(k)})}^{-1}(\mathcal{I}(\tau_k,f_{m-1}(k)))}
  u^{\ast}_{f_{m-1}(k)}\hat{\omega}> \frac{\bar{C}_{\Omega}}{2^{m}}.
  \]
Of course, this remains true if we pass to a subsequence.    
  By the inductive construction we know that for a given $t\in
  \mathcal{I}_{\varepsilon}\setminus L_{m-1}$ there does not exist
  a sequence $\tau_k\rightarrow t$ with the property 
\begin{eqnarray}\label{EQN14}
  \limsup_{k\rightarrow \infty}  \int_{{(s\circ
  u_{f_{m-1}(k)})}^{-1}(\mathcal{I}(\tau_k,f_{m-1}(k)))}
  u^{\ast}_{f_{m-1}(k)}\hat{\omega}> \frac{\bar{C}_{\Omega}}{2^{m-1}}.
  \end{eqnarray}
However, there might be for such a $t$  a sequence if we replace the
  right-hand side by the smaller number  $\frac{\bar{C}_{\Omega}}{2^{m}}$.
Thus we consider the possible cases.
Assume first that for every $t\in \mathcal{I}_{\varepsilon}\setminus
  L_{m-1}$ there is no sequence for which
\begin{eqnarray}\label{EQN141}
  \limsup_{k\rightarrow \infty}  \int_{{(s\circ
  u_{f_{m-1}(k)})}^{-1}(\mathcal{I}(\tau_k,f_{m-1}(k)))}
  u^{\ast}_{f_{m-1}(k)}\hat{\omega}> \frac{\bar{C}_{\Omega}}{2^{m}}
  \end{eqnarray}
  holds.
In this case we define $L_m:=L_{m-1}$ and $f_{m}(k)=f_{m-1}(k)$. 
Then $f_m:{\mathbb N}\rightarrow f_{m-1}({\mathbb N})$ and $L_{m-1}\subset
  L_m$.
Moreover Items (2)-(5) hold for $m$ and the construction for the $m$-case
  is complete.
    
Assume next we find a $t_{m,1}\in \mathcal{I}_{\varepsilon}\setminus
  L_{m-1}$ for which we have a sequence $\tau_k\rightarrow t_{m,1}$
  for which (\ref{EQN141}) holds. 
Then we take a subsequence $h_{m,1}:{\mathbb N}\rightarrow
  f_{m-1}({\mathbb N})$ for which
\begin{eqnarray}\label{EQN15}
\lim_{k\rightarrow \infty}  \int_{{(s\circ
u_{h_{m,1}(k)})}^{-1}(\mathcal{I}(\tau_{h_{m,1}(k)},h_{m,1}(k)))}
u^{\ast}_{h_{m,1}(k)}\hat{\omega}> \frac{\bar{C}_{\Omega}}{2^{m}}.
\end{eqnarray}
Next we can ask if we find a $t_{m,2}\in
  \mathcal{I}_{\varepsilon}\setminus (L_{m-1}\cup \{t_{m,1}\})$ for which we
  have a sequence $\tau_k\rightarrow t_{m,2}$ such that
\[
\limsup_{k\rightarrow \infty}  \int_{{(s\circ
u_{h_{m,1}(k)})}^{-1}(\mathcal{I}(\tau_{k},h_{m,1}(k)))}
u^{\ast}_{h_{m,1}(k)}\hat{\omega}> \frac{\bar{C}_{\Omega}}{2^{m}}
\]
If that is not the case we define $f_m(k)=h_{m,1}(k)$ and
  $L_m=L_{m-1}\cup\{t_{m,1}\}$. 
One easily verifies (2)-(5) and the construction for the $m$-case is complete.

Otherwise we find a $t_{m,2}\in  \mathcal{I}_{\varepsilon}\setminus
  (L_{m-1}\cup \{t_{m,1}\})$ with the above mentioned property and we can
  take $h_{m,2}:{\mathbb N}\rightarrow h_{m,1}({\mathbb N})$ so that the
  left-hand limit exists and is greater than $\bar{C}_{\Omega}/2^{m}$.
Again we ask if we find a third point $t_{m,3}\in
  L_{m-1}\cup\{t_{m,1},t_{m,2}\}$ with the same property. 
If that is not the case we define $L_m=L_{m-1}\cup\{t_{m,1},t_{m,2}\}$ and
  $f_m(k)=h_{m,2}(k)$. 
Again one verifies Items (2)-(5).
Otherwise we obtain a sequence $\tau_k\rightarrow t_{m,3}$ and
  $h_{m,3}:{\mathbb N}\rightarrow h_{m,2}({\mathbb N})$ such that 
  \[
  \limsup_{k\rightarrow \infty}  \int_{{(s\circ
  u_{h_{m,3}(k)})}^{-1}(\mathcal{I}(\tau_{k},h_{m,3}(k)))}
  u^{\ast}_{h_{m,3}(k)}\hat{\omega}> \frac{\bar{C}_{\Omega}}{2^{m}}.
  \]
In the outline above we consider several cases and in some of them the
  procedure terminates after a finite number of steps.
The only possible way of the procedure not to stop is that we find more
  and more points $t_{m,n}$, $n\in {\mathbb N}$, with the previously
  described properties.
However, we claim that this procedure terminates after a finite number of
  steps $n_m$ so that we can define
  $L_m=L_{m-1}\cup\{t_{m,1},..,t_{m,n_m}\}$ and $f_m(k)=h_{m,n_m}(k)$
  which satisfies by construction Items (2)-(5).
To see that the procedure terminates, assume otherwise.
Pick a positive integer $N$ such that 
  \[
  N\cdot 2^{-m} >1
  \]
  and consider the  $N$ different points $t_{m,1},..,t_{m,N} \in
  \mathcal{I}_{\varepsilon}\setminus L_{m-1}$. 
By construction we have the
  maps
\begin{align*}                                                            %% EQN
  h_{m,1}&:{\mathbb N}\rightarrow f_{m-1}({\mathbb N}), \\
  h_{m,2}&:{\mathbb N}\rightarrow h_{m,1}({\mathbb N}),\\
  &\vdots \\
  h_{m,N}&:{\mathbb N}\rightarrow h_{m,N-1}({\mathbb N}),
  \end{align*}
  with the properties that for $t_{m,n}$, $n\in\{1,...,N\}$, there exists
  a sequence $\tau^n_k\rightarrow t_{m,n}$  such that
\begin{eqnarray}\label{EQN20}
  \lim_{k\rightarrow\infty} \int_{{(s\circ
  u_{h_{m,j}(k)})}^{-1}(\mathcal{I}(\tau_k,h_{m,j}(k)))}
  u^{\ast}_{h_{m,j}(k)}\hat{\omega}> \frac{\bar{C}_{\Omega}}{2^{m}}.
  \end{eqnarray}
Using (\ref{EQN20}) we obtain with $H(k):=h_{m,N}(k)$ and the sequences
  $\tau_{k}^n\rightarrow t_{n}$ for $n=1,...,N$
\begin{eqnarray*}
  \bar{C}_{\Omega}&\geq &\limsup_{k\rightarrow\infty} \int_{S_k}
  u_{H(k)}^{\ast}\hat{\omega}\\
  &\geq & \sum_{n=1}^{N} \lim_{k\rightarrow\infty} \int_{(s\circ
  u_{H(k)})^{-1}(\mathcal{I}(\tau_k^n,H(k)))}
  u^{\ast}_{H(k)}\hat{\omega}\\
  &\geq & N\cdot 2^{-m}\cdot \bar{C}_{\Omega}\\
  &>& \bar{C}_{\Omega}.
  \end{eqnarray*}
Observe that we use that the sets $\mathcal{I}(\tau_k^n,H(k))$ are
  mutually disjoint for large $k$.
  Indeed,  since $\tau^n_k\rightarrow t_{m,n}$ as $k\rightarrow\infty$ and
  the points $t_{m,n}$ are mutually disjoint, this follows from the fact
  that the diameter of the intervals $\mathcal{I}(\tau^n_k,H(k))$ is
  shrinking to $0$.
Thus we have shown that for each fixed \(m\), the procedure which
  generates the set \(\{t_{m,1}, t_{m,2}, \ldots\}\) terminates after a
  finite number of iterations, and hence this set is finite.
This completes the proof of Proposition \ref{PROP_inductive_construction}.
\end{proof}

With these subsequences established, we now pass to a diagonal
  subsequence by defining
  \begin{align}\label{EQ_subsequence}                                     %% EQN
    k_m := f_m(m) \quad\text{for }m\in \mathbb{N}.
    \end{align}
We also define the countable subset
  \begin{align}\label{EQ_L_definition}                                    %% EQN
    L:= \bigcup_{m\in \mathbb{N}} L_m.
    \end{align}
  The following result is then an immediate consequence of the above
  construction.

%%%%%%%%%%%%%%%%%%%%%%%%%%%%%%%%%%%%%%%%%%%%%%%%%%%%%%%%%%%%%%%%%%%%%%%%%%%%%%%%
%%%%%%%%%%                            LEMMA                            %%%%%%%%%
%%%%%                                                                       %%%%
\begin{lemma}[vanishing horizontal area]
  \label{LEM_vanishing_horizontal_area}
  \hfill\\
Let the \((\mathcal{I}_\epsilon \times M, \Omega, \hat{\lambda})\),
  \(\{J_k\}_{k\in \mathbb{N}}\), and \(u_k:(S_k, j_k) \to
  (\mathcal{I}_\epsilon \times M, J_k)\) be as above, and let
  \(\{k_m\}_{m\in\mathbb{N}}\) be the subsequence given in equation
  (\ref{EQ_subsequence}).
Then for each \(t_0 \in \mathcal{I}_\epsilon \setminus L\), we have
  \begin{align*}                                                        %% EQN
    \lim_{m\to \infty} \int_{(s\circ
    u_{k_m})^{-1}\big(\mathcal{I}(t_0, k_m)\big)}
    u^*_{k_m}{\hat{\omega}} =0.
    \end{align*}
\end{lemma}
\begin{proof}
Since $t_0\in\mathcal{I}_{\varepsilon}\setminus L$ it holds that
  $t_0\not\in L_m$ for every $m\in {\mathbb N}_0$.
Hence there does not exist a sequence $\tau_k\rightarrow t_0$ with
  \[
  \limsup_{k\rightarrow \infty}  \int_{{(s\circ
  u_{f_m(k)})}^{-1}(\mathcal{I}(\tau_k,f_m(k)))}
  u^{\ast}_{f_m(k)}\hat{\omega}>  \frac{\bar{C}_{\Omega}}{2^m}
  \]
  and specifically we must have 
\begin{eqnarray}\label{EQN30}
  \limsup_{k\rightarrow \infty}  \int_{{(s\circ
  u_{f_m(k)})}^{-1}(\mathcal{I}(t_0,f_m(k)))}
  u^{\ast}_{f_m(k)}\hat{\omega}\leq  \frac{\bar{C}_{\Omega}}{2^m}.
  \end{eqnarray}
Since $k_m:=f_m(m)$ is the diagonal sequence we deduce from (\ref{EQN30})
  that for every $m\in {\mathbb N}_0$
\begin{eqnarray}\label{EQN40}
  \limsup_{i\rightarrow \infty}  \int_{{(s\circ
  u_{k_i})}^{-1}(\mathcal{I}(t_0,k_i))} u^{\ast}_{k_i}\hat{\omega}\leq
  \frac{\bar{C}_{\Omega}}{2^m}.
  \end{eqnarray}
This implies the assertion
\[
  \lim_{m\to \infty} \int_{(s\circ
    u_{k_m})^{-1}\big(\mathcal{I}(t_0, k_m)\big)}
    u^*_{k_m}{\hat{\omega}} =0.
    \]
\end{proof}

%%%%%                                                                       %%%%
%%%%%%%%%%                                                             %%%%%%%%%
%%%%%%%%%%%%%%%%%%%%%%%%%%%%%%%%%%%%%%%%%%%%%%%%%%%%%%%%%%%%%%%%%%%%%%%%%%%%%%%%
%
We will next make use of this subsequence, and with it we will be
  interested in the \(\hat{\lambda}\) integrals of our curves along various
  levels \(\{t\}\times M\).
This is made precise in equation (\ref{EQ_F_definition}) below, however
  for the moment we note that we will be interested in various properties of
  these functions, (e.g. that they are measurable and integrable).
Indeed, studying properties and limits of such functions will ultimately
  yield the desired result regarding existence of periodic orbits on almost
  every energy level.
We now proceed with this argument.

For each \(m\in \mathbb{N}\) we define \(\mathcal{Y}_m\) to be the
  set of critical values of the functions \(s\circ u_{k_m}\colon S_{k_m}
  \to \mathcal{I}_\epsilon \subset \mathbb{R}\).
By Sard's theorem, each of these sets has measure zero; that is,
  \(\mu(\mathcal{Y}_m)=0\) for each \(m\in \mathbb{N}\).
By countable sub-additivity, we then also have
  \begin{align}\label{EQ_Y_definition}                                    %% EQN
    \mu(\mathcal{Y}) = 0\qquad\text{where} \qquad \mathcal{Y}:=
    \bigcup_{m\in \mathbb{N}} \mathcal{Y}_m.
    \end{align}
For each \(m\in \mathbb{N}\) we then define the functions
  \begin{align}\label{EQ_F_definition}                                    %% EQN
    &F_m\colon \mathcal{I}_\epsilon \to [0, \infty) \subset \mathbb{R}
    \\
    &F_m(t)=
    \begin{cases}
    \int_{(s\circ u_{k_m})^{-1}(t)} u_{k_m}^*\hat{\lambda} &\text{if }t\in
    \mathcal{I}_\epsilon\setminus \mathcal{Y}\notag
    \\
    0 &\text{otherwise}.
    \end{cases}
    \end{align}
We now claim the following.

%%%%%%%%%%%%%%%%%%%%%%%%%%%%%%%%%%%%%%%%%%%%%%%%%%%%%%%%%%%%%%%%%%%%%%%%%%%%%%%%
%%%%%%%%%%                            LEMMA                            %%%%%%%%%
%%%%%                                                                       %%%%
\begin{lemma}[$F_m$ is measurable]
  \label{LEM_Fm_are_measurable}
  \hfill\\
For each \(m\in \mathbb{N}\), the function \(F_m\) defined above is a
  measurable function.
Moreover, for each \(m\in \mathbb{N}\), the function \(F_m\) agrees with
  the function
  \begin{align}\label{EQ_auxiliary_F}                                     %% EQN
    t\mapsto \int_{(s\circ u_{k_m})^{-1}(t)\setminus \mathcal{X}_m}
    u_{k_m}^*\hat{\lambda}
    \end{align}
  almost everywhere; here \(\mathcal{X}_m= \{\zeta\in S: d(s\circ
  u_{k_m})_\zeta =0\}\) as in Lemma \ref{LEM_coarea_application}.
\end{lemma}
%%%%%                                                                       %%%%
%%%%%%%%%%                                                             %%%%%%%%%
%%%%%%%%%%%%%%%%%%%%%%%%%%%%%%%%%%%%%%%%%%%%%%%%%%%%%%%%%%%%%%%%%%%%%%%%%%%%%%%%
%
\begin{proof}
  First fix \(m\in \mathbb{N}\) and define 
  \begin{align*}                                                          %% EQN
    &\widetilde{F}_m\colon \mathcal{I}_\epsilon\to [0,\infty) \subset
    \mathbb{R}
    \\
    &\widetilde{F}_m(t)=
    \begin{cases}
      \int_{(s\circ u_{k_m})^{-1}(t)} u_{k_m}^*\hat{\lambda} &\text{if }t\in
      \mathcal{I}\setminus \mathcal{Y}_m
      \\
      0 &\text{otherwise}.
      \end{cases}
    \end{align*}
Observe that \(F_m\) and \(\widetilde{F}_m\) agree almost everywhere and
  consequently if  \(\widetilde{F}_m\) is measurable so is $F_m$.
Because \(s\circ u_{k_m}(\mathcal{X}_m)=\mathcal{Y}_m\) which has measure
  zero, it follows that \(\widetilde{F}_m\) is almost everywhere equals the
  function defined in equation (\ref{EQ_auxiliary_F}).
This establishes the second part of the lemma.

To establish the first part of the lemma it is sufficient to show that
  \(\widetilde{F}_m\) is measurable.
That is, it is sufficient to show that for each \(r\in [0, \infty) \), the
  set \(\widetilde{F}_m^{-1}\big([0, r)\big)\) is measurable.
To that end, note that by assumption in Theorem
  \ref{THM_localized_almost_existence} we have \(s \circ u_{k_m}(S_{k_m}) =
  \mathcal{I}_\epsilon\) and consequently for each \(s_0\in
  \mathcal{I}_\epsilon \setminus \mathcal{Y}_m\) we have
  \begin{align*}                                                          %% EQN
    \widetilde{F}_m(s_0) = \int_{(s\circ u_{k_m})^{-1}(s_0)}
    u_{k_m}^*\hat{\lambda}  >0.
    \end{align*}
It also follows that \(\widetilde{F}_m^{-1}(0) = \mathcal{Y}_m\).
Note that \(\mathcal{Y}_m\) is closed in \(\mathcal{I}_\epsilon\).
Also note that \(\widetilde{F}_m\big|_{\mathcal{I}_\epsilon\setminus
  \mathcal{Y}_m}\) is differentiable, and hence continuous, and thus 
  \begin{align*}                                                          %% EQN
    A_r:=\big(\widetilde{F}_m\big|_{\mathcal{I}_\epsilon \setminus
    \mathcal{Y}_m}\big)^{-1} \big([0 , r)\big)\qquad \text{is open in
    }\mathcal{I}_\epsilon\setminus \mathcal{Y}_m.
    \end{align*}
That is, there exists an open set \(O\subset \mathcal{I}_\epsilon\) such that
  \(A_r = O\cap (\mathcal{I}_\epsilon \setminus \mathcal{Y}_m)\).
However \(\mathcal{I}_\epsilon\setminus \mathcal{Y}_m\) is open in
  \(\mathcal{I}_\epsilon\) and hence \(A_r\) is open in
  \(\mathcal{I}_\epsilon\).
Consequently \(A_r\) is measurable.
However, we then have
  \begin{align*}                                                          %% EQN
    \widetilde{F}_m^{-1}\big([0 , r)\big)
    &=
    \big(\widetilde{F}_m\big|_{\mathcal{I}_\epsilon \setminus
    \mathcal{Y}_m}\big)^{-1} \big([0, r )\big)
    \; \; \bigcup  \; \; 
    \big(\widetilde{F}_m\big|_{ \mathcal{Y}_m}\big)^{-1}
    \big([0 , r)\big)
    \\
    &=A_r \; \; \cup  \; \; \mathcal{Y}_m,
    \end{align*}
  with \(A_r\) open (and hence measurable) and \(\mathcal{Y}_m\) having
  measure zero (and hence is measurable).
We conclude that \(\widetilde{F}_m^{-1}\big([0 , r)\big)
  \) is measurable, which completes the proof of Lemma
  \ref{LEM_Fm_are_measurable}.
\end{proof}%%%%%%%%%%%%%%%%%%%%%%%%%%%%%%%%%%%%%%%%%%%%%%%%            END PROOF
For the following discussion we introduce the map \(F\) which is defined
  as follows:
\begin{align}\label{MAP_F}                                               %% EQN
  &F\colon\mathcal{I}_{\varepsilon}\rightarrow {\mathbb
  R}^+\cup\{\infty\}\notag
  \\
  &F(s)=\liminf_{m\rightarrow \infty} F_m(s).
  \end{align}

Then by standard measure theory results \(F\) is an extended measurable
  function, see \cite{Bartle}.
With this definition in place the next guiding observation (made rigorous
  below) is that two important results hold.
The first one is given in the next proposition.

%%%%%%%%%%%%%%%%%%%%%%%%%%%%%%%%%%%%%%%%%%%%%%%%%%%%%%%%%%%%%%%%%%%%%%%%%%%%%%%%
%%%%%%%%%%                         PROPOSITION                         %%%%%%%%%
%%%%%                                                                       %%%%
\begin{proposition}[$F$ is almost everywhere finite]
  \label{PROP_F_is_almost_everywhere_finite}
  \hfill\\
  With $F$ as just defined in (\ref{MAP_F}) it holds that 
 \[
 \textrm{measure}(\{s\in \mathcal{I}_\epsilon : F(s) = \infty\}) =0
 \] 
\end{proposition}
%%%%%                                                                       %%%%
%%%%%%%%%%                                                             %%%%%%%%%
%%%%%%%%%%%%%%%%%%%%%%%%%%%%%%%%%%%%%%%%%%%%%%%%%%%%%%%%%%%%%%%%%%%%%%%%%%%%%%%%
%
\begin{proof}
In view of (\ref{EQ_coarea_lambda}) we have the formula
 \[
 \int_{S_{k_m}} u_{k_m}^*(ds\wedge\hat{\lambda}) =
 \int_{\mathcal{I}_{\varepsilon}}\left(\int_{(s\circ
 u_{k_k})^{-1}(t)\setminus\mathcal{X}} u_{k_m}^*\hat{\lambda}\right)ds.
 \]
 Using (\ref{EQN133}) and (\ref{EQN136}) we infer that 
 \[
  2\cdot C_{\Omega}= \bar{C}_{\Omega} \geq    \int_{S_{k_m}}
  u_{k_m}^*(ds\wedge\hat{\lambda}) = \int_{\mathcal{I}_{\varepsilon}}
  F_m(s)ds.
\]
 In view of Fatou's Lemma, recalling that \(F=\liminf F_m\), we obtain
 \[
\bar{C}_{\Omega}\geq \liminf_{m\rightarrow\infty}
\int_{\mathcal{I}_{\varepsilon}} F_m(s)ds\geq
\int_{\mathfrak{I}_{\varepsilon}} F(s)ds.
 \]
This shows that  \(\{s\in \mathcal{I}_\epsilon : F(s) = \infty\}\) has
 vanishing measure.
\end{proof}

The second result, which is more substantial, studies the points $s$
  satisfying  \(F(s)<\infty\).
We will establish that there is a periodic orbit on \(\{s\}\times M\)
  provided $F(s)<\infty$ and $s\in
  \mathcal{I}_{\varepsilon}\setminus(L\cup\mathcal{Y})$.
The proof of Theorem \ref{THM_localized_almost_existence} then follows
  immediately since we have just established that  \(\{s\in
  \mathcal{I}_\epsilon : F(s) < \infty\}\) has (full measure)
  \(2\varepsilon\) and $L\cup\mathcal{Y}$ has measure zero.

%%%%%%%%%%%%%%%%%%%%%%%%%%%%%%%%%%%%%%%%%%%%%%%%%%%%%%%%%%%%%%%%%%%%%%%%%%%%%%%%
%%%%%%%%%%                         PROPOSITION                         %%%%%%%%%
%%%%%                                                                       %%%%
\begin{proposition}[bounded $F$ yields periodic orbit]
  \label{PROP_bounded_F_yields_periodic_orbit}
  \hfill\\
Let \((\mathcal{I}_\epsilon\times M, \Omega, \hat{\lambda} )\),
  \(\{J_k\}_{k\in \mathbb{N}}\), and \(u_k:(S_k, j_k) \to
  (\mathcal{I}_\epsilon\times M, J_k)\) be as in the hypotheses of Theorem
  \ref{THM_localized_almost_existence}.
Let \(L\) be defined as in equation (\ref{EQ_L_definition}),
  \(\mathcal{Y}\) be as defined in equation (\ref{EQ_Y_definition}), and
  let \(F_m\) be the functions defined in equation
  (\ref{EQ_F_definition}).
Also, let \(s_0 \in \mathcal{I}_\epsilon \setminus (L \cup \mathcal{Y})\),
  and suppose
  \begin{align*}                                                          %% EQN
  F(s_0)=\liminf_{m\rightarrow\infty} F_m(s_0) = \lim_{N\to \infty}
  \inf_{m \geq N} F_m(s_0) < \infty.
    \end{align*}
\emph{Then} there exists a periodic orbit on the energy level
  \(\{s_0\}\times M\).
\end{proposition}
%%%%%                                                                       %%%%
%%%%%%%%%%                                                             %%%%%%%%%
%%%%%%%%%%%%%%%%%%%%%%%%%%%%%%%%%%%%%%%%%%%%%%%%%%%%%%%%%%%%%%%%%%%%%%%%%%%%%%%%
%
\begin{proof}
We begin by passing to a subsequence (still denoted with subscripts \(m\))
  so that
  \begin{align*}                                                          %% EQN
    \lim_{N\to \infty} \inf_{m \geq N} F_m(s_0) = \lim_{m\to \infty}
    F_m(s_0)= \lim_{m\to\infty}\int_{(s\circ u_{k_m})^{-1}(s_0)}
    u_{k_m}^*\hat{\lambda} =: C_{\hat{\lambda}} < \infty.
    \end{align*}
Let \(\Psi_{k_m}^{s_0} \colon \mathcal{I}_\epsilon\times M \to
  \mathbb{R}\times M\) be the embedding provided in equation
  (\ref{EQ_Psi_k_t}), whose restriction  \(\widetilde{\Psi}_m\) to
  \(\mathcal{I}(s_0,k_m)\times M\) given by
    \begin{align*}                                                          %% EQN
    &\widetilde{\Psi}_{m}\colon \mathcal{I}(s_0, k_m)\times M\to
    (-1,1)\times M
    \\
    &\widetilde{\Psi}_m = \Psi_{k_m}^{s_0}\big|_{\mathcal{I}(s_0,
    k_m)\times M}
    \end{align*}
  is a diffeomorphism.  
The key feature is that $\widetilde{\Psi}_m$ maps $s_0$ to $0$ 
  and stretches out the shrinking intervals to  length two.
Define 
  \begin{align*}                                                          %% EQN
    \widetilde{S}_{m}:= (s\circ u_{k_m})^{-1}\big(\mathcal{I}(s_0,
    k_m)\big)\qquad\text{and}\qquad\tilde{j}_m:=
    j_{k_m}\big|_{\widetilde{S}_m},
    \end{align*}
  and
  \begin{align*}                                                          %% EQN
    &\tilde{u}_m\colon \widetilde{S}_m\to (-1,1)\times M
    \\
    &\tilde{u}_m= \widetilde{\Psi}_m\circ u_{k_m}.
    \end{align*}
Also, with \(a\) the coordinate on \((-1,1)\), define the following
  structures on \((-1,1)\times M\)
  \begin{align*}                                                          %% EQN
    \widetilde{J}_m:=& \;  (\widetilde{\Psi}_m)_*J_{k_m},
    \\
    \tilde{\lambda}\, :=&\; (\widetilde{\Psi}_m)_*\hat{\lambda} \ \
    \textrm{(Note that there is no $m$-dependence on the left-hand side!)}
    \\
    \tilde{\omega}_m := & \;(\widetilde{\Psi}_m)_*\hat{\omega}
    \\
    \tilde{g}_m := & \; (da \wedge \tilde{\lambda} +
    \tilde{\omega}_m)\circ ({\rm Id}\times \widetilde{J}_m)
    =da^2 + \tilde{\lambda}^2 + \tilde{\omega}_m\circ ({\rm Id}\times
    \widetilde{J}_m)
    \end{align*}
We now make several observations which follow immediately from our
  construction.
\begin{enumerate}                                                         %% NUM
  \item 
  The maps \(\tilde{u}_m\colon (\widetilde{S}_m, \tilde{j}_m) \to
  \big((-1,1)\times M, \widetilde{J}_m\big)\) are proper pseudoholomorphic
  maps without boundary and which lack constant components.
  \item 
  \({\rm Genus}(\widetilde{S}_m) \leq C_g\)
  \item \((a\circ \tilde{u}_m)^{-1}(0)\neq \emptyset\)
  \item 
  \(\int_{(a\circ \tilde{u}_m)^{-1}(0)} \tilde{u}_m^* \tilde{\lambda}
  \to C_{\hat{\lambda}} <\infty\)
  \item 
  \((-1,1)\times M\), equipped with any of the \((\tilde{\lambda},
  \tilde{\omega}_m)\), is a realized Hamiltonian homotopy in the sense of
  Definition 2.9 of \cite{FH2}, with adapted almost
  Hermitian structures \((\widetilde{J}_m, \tilde{g}_m)\), recalled in
  Definition \ref{DEF_hamiltonian_homotopy} in the Appendix.
  \item 
  \(\tilde{\omega}_m\to \tilde{\omega}^{s_0}\) in
  \(C^\infty\)  as \(m\to \infty\), where
\begin{align*}                                                            %% EQN
  \tilde{\omega}^{s_0} = {\rm pr}_{s_0}^* \hat{\omega}
  \qquad\text{and}\qquad{\rm pr}_{s_0} (s, p) = (s_0, p).
  \end{align*}
  \item 
  \(\sup_{m\in \mathbb{N}} \|  \widetilde{J}_m\|_{C^k} <
  \infty\) and   \(\sup_{m\in \mathbb{N}} \|
  \tilde{g}_m\|_{C^k} < \infty\) for each \(k\in \mathbb{N}\).
  \end{enumerate}
Note that as a consequence of the above, together with Theorem 8 in
  \cite{FH2} (area bounds in a realized Hamiltonian
  homotopy), it follows that there exists a number \(C_A>0\) such that for
  all sufficiently large \(m\in \mathbb{N}\) we have
  \begin{align*}                                                          %% EQN
    {\rm Area}_{\tilde{u}_m^*\tilde{g}_m}(\widetilde{S}_m) =
    \int_{\widetilde{S}_m} \tilde{u}_m^*(da\wedge \tilde{\lambda} +
    \tilde{\omega}_m) \leq C_A.
    \end{align*}
We then pass to a subsequence (still denoted with subscripts \(m\)) so
  that \((\widetilde{J}_m, \tilde{g}_m)\) converges in
  \(C^\infty\).
We then apply the main result from \cite{Fish2} (namely Target-local
  Gromov compactness) to pass to a further subsequence (still denoted with
  subscripts \(m\)) and find compact Riemann surfaces
  \((\widetilde{\Sigma}_m, \tilde{j}_m)\subset (\widetilde{S}_m,
  \tilde{j}_m)\) with smooth boundary for which
  \begin{align}\label{EQ_trimming}                                        %% EQN
    \tilde{u}_m(\widetilde{S}_m\setminus \widetilde{\Sigma}_m) \subset
    \big((-1, {\textstyle -\frac{1}{2}}) \cup ({\textstyle \frac{1}{2}},
    1)\big) \times M \qquad\text{for all }m\in \mathbb{N},
    \end{align}
  and for which the sequence \(\tilde{u}_m\colon \widetilde{\Sigma}_m \to
  (-1, 1)\times M\) converges in a Gromov sense to a (nodal) limit
  pseudoholomorphic map \(\tilde{u}_\infty\colon
  (\widetilde{\Sigma}_\infty, \tilde{j}_\infty )\to \big((-1, 1)\times M,
  \widetilde{J}_\infty\big)\); here \((\widetilde{\Sigma}_\infty,
  \tilde{j}_\infty)\) is compact and may have smooth boundary.
Note that because \((a\circ \tilde{u}_m)^{-1}(0)\neq \emptyset\) for all
  \(m\in \mathbb{N}\), it follows from equation (\ref{EQ_trimming}) that 
  \(\widetilde{\Sigma}_\infty\neq \emptyset\), and if \(\partial
  \widetilde{\Sigma}_\infty\neq \emptyset\), then
  \(\tilde{u}_\infty(\partial \widetilde{\Sigma}_\infty) \subset ((-1,
  -\frac{1}{4})\cup (\frac{1}{4}, 1))\times M\).
Because \(s_0\in \mathcal{I}_\epsilon\setminus (L\cup \mathcal{X})\)
  (specifically because \(s_0\notin L\)), we must have
  \begin{align*}                                                          %% EQN
    \int_{\widetilde{\Sigma}_m}\tilde{u}_m^* \tilde{\omega}_m \to 0
    \qquad\text{and}\qquad
    \int_{\widetilde{\Sigma}_m}\tilde{u}_m^* \tilde{\omega}_m
    \to \int_{\widetilde{\Sigma}_\infty}
    \tilde{u}_\infty^*\tilde{\omega}^{s_0}
    \end{align*}
  so 
  \begin{align*}                                                          %% EQN
    \int_{\widetilde{\Sigma}_\infty}
    \tilde{u}_\infty^*\tilde{\omega}^{s_0} =0 .
    \end{align*}
We now make the following claim.

%%%%%%%%%%%%%%%%%%%%%%%%%%%%%%%%%%%%%%%%%%%%%%%%%%%%%%%%%%%%%%%%%%%%%%%%%%%%%%%%
%%%%%%%%%%                            LEMMA                            %%%%%%%%%
%%%%%                                                                       %%%%
\begin{lemma}[non-trivial limit component with boundary]
  \label{LEM_non_trivial_limit_component_with_boundary}
  \hfill\\
There exists a connected component \(\widetilde{\Sigma}'\) of
  \(\widetilde{\Sigma}_\infty\) for which \(\widetilde{\Sigma}' \cap
  (a\circ \tilde{u}_\infty)^{-1}(0)\neq \emptyset\) and \(\partial
  \widetilde{\Sigma}'\neq \emptyset\).
\end{lemma}
%%%%%                                                                       %%%%
%%%%%%%%%%                                                             %%%%%%%%%
%%%%%%%%%%%%%%%%%%%%%%%%%%%%%%%%%%%%%%%%%%%%%%%%%%%%%%%%%%%%%%%%%%%%%%%%%%%%%%%%
%
\begin{proof}
Suppose not.
For notational clarity, we let \(\mathcal{S}_0\) denote the set of
  connected components of \(\widetilde{\Sigma}_\infty\) which have non-empty
  intersection with \((a\circ \tilde{u}_\infty)^{-1}(0)\).
Then there are three possibilities.\\

\noindent \underline{\emph{Case I:}} \(\mathcal{S}_0=\emptyset\).\\
In this case, \(\widetilde{\Sigma}_\infty\) can be written as the disjoint
  union \(\widetilde{\Sigma}_\infty = \widetilde{\Sigma}_\infty^+ \cup
  \widetilde{\Sigma}_\infty^-\) where
  \begin{align*}                                                          %% EQN
    a\circ \tilde{u}_\infty( \widetilde{\Sigma}_\infty^+) \subset (0, 1)
    \qquad\text{and}\qquad
    a\circ \tilde{u}_\infty( \widetilde{\Sigma}_\infty^-) \subset (-1, 0).
    \end{align*}
If either of \(\widetilde{\Sigma}_\infty^\pm\) are empty, then by Gromov
  convergence there must be some large \(m\in \mathbb{N}\)
  and some number \(a_0\in (-\frac{1}{2},\frac{1}{2})\) for which
  \((a\circ \tilde{u}_m)^{-1}(a_0)=\emptyset\), and hence there exists a
  \(s_0\in \mathcal{I}_\epsilon\) (specifically \(s_0 =
  \psi_{k_m}^{-1}(a_0)\)) for which \((s\circ u_{k_m})^{-1}(s_0)=\emptyset\).
However, this violates the assumption in Theorem
  \ref{THM_localized_almost_existence} which states that \(s\circ
  u_k(S_k)=\mathcal{I}_\epsilon\) for all \(k\).
Consequently \(\widetilde{\Sigma}_\infty^+\neq \emptyset\) and
  \(\widetilde{\Sigma}_\infty^-\neq \emptyset\).

Next note that because \(\widetilde{\Sigma}_\infty\) is compact, it
  follows that each of \(\widetilde{\Sigma}_\infty^\pm\) are compact and
  non-empty.
However, we then have
  \begin{align*}                                                          %% EQN
    \sup_{z\in \widetilde{\Sigma}_\infty^-} a\circ \tilde{u}_\infty(z)
    = a_- < 0 < a_+ =
    \inf_{z\in \widetilde{\Sigma}_\infty^+} a\circ \tilde{u}_\infty(z).
    \end{align*}
But then again by Gromov convergence, this will violate the assumption
  that \(s\circ u_k(S_k)=\mathcal{I}_\epsilon\) for all \(k\).
Thus Case I is impossible.\\

\noindent \underline{\emph{Case II:}}  For each \(\widetilde{\Sigma}' \in
  \mathcal{S}_0\), the restriction
  \(\tilde{u}_\infty\big|_{\widetilde{\Sigma}'}\) is the constant map.\\
In this case, we can write \(\widetilde{\Sigma}_\infty\) as the disjoint
  union \(\widetilde{\Sigma}_\infty=\widetilde{\Sigma}_\infty^+\cup
  \widetilde{\Sigma}_\infty^0 \cup \widetilde{\Sigma}_\infty^-\) where
  \begin{align*}                                                          %% EQN
    a\circ \tilde{u}_\infty( \widetilde{\Sigma}_\infty^+) \subset (0, 1),
    \quad
    a\circ \tilde{u}_\infty( \widetilde{\Sigma}_\infty^0) = \{0\},
    \quad\text{and}\quad
    a\circ \tilde{u}_\infty( \widetilde{\Sigma}_\infty^-) \subset (-1, 0).
    \end{align*}
The argument then proceeds as in Case I, which shows that Case II is also
  impossible.\\

\noindent \underline{\emph{Case III:}}  There exists \(\widetilde{\Sigma}'
  \in \mathcal{S}_0\), such that the restriction
  \(\tilde{u}_\infty\big|_{\widetilde{\Sigma}'}\) is not constant.\\
Note that by the contradiction hypothesis, we must have \(\partial
  \widetilde{\Sigma}' = \emptyset\), and by Gromov convergence
  \(\widetilde{\Sigma}'\) is compact.
That is, \((\widetilde{\Sigma}', \tilde{j}_\infty)\) is a closed Riemann
  surface.
Next, we make use of the fact that
  \begin{align*}                                                          %% EQN
    \int_{\widetilde{\Sigma}_\infty} \tilde{u}_\infty^*
    \tilde{\omega}^{s_0}  = 0
    \end{align*}
  together with the fact that \(\tilde{\omega}^{s_0}\) evaluates
  non-negatively on \(\widetilde{J}_\infty\)-complex lines to conclude
  that for each \(z\in \widetilde{\Sigma}'\) we must have
  \begin{align*}                                                          %% EQN
    {\rm Image}(T_z \tilde{u}_\infty) \subset \ker
    (\tilde{\omega}^{s_0} )_{\tilde{u}_\infty(z)} = {\rm
    Span}\big(\partial_a , X(\tilde{u}_\infty(z))\big),
    \end{align*}
  where \(X\) is the Hamiltonian vector field \(X(a, p) = \widehat{X}(s_0,
  p)\).
Consequently, there exists a Hamiltonian trajectory \(\gamma\colon
  \mathbb{R}\to \{0\}\times M\), solving \(\gamma'(t) = X(\gamma(t))\) for
  all \(t\) for which
  \begin{align*}                                                          %% EQN
    \tilde{u}_\infty(\widetilde{\Sigma}') \subset (-1,1)\times
    \gamma(\mathbb{R}).
    \end{align*}
If \(\gamma\) is not periodic, then the map 
 \begin{align*}                                                           %% EQN
    &\Phi\colon(-1,1)\times \mathbb{R}\rightarrow {\mathbb R}\times M
    \\ 
    &\Phi(s,t) = \big(s, \gamma(t)\big)
    \end{align*}
  is an injective pseudoholomorphic immersion, and hence the map 
  \begin{align*}                                                          %% EQN
    \Phi^{-1}\circ \tilde{u}_\infty \colon \widetilde{\Sigma}' \to
    (-1,1)\times \mathbb{R} \subset \mathbb{C}
    \end{align*}
  is a holomorphic map from a closed Riemann surface into \(\mathbb{C}\).
By the maximum principle, and the fact that \(\Phi\) is an immersion, it
  follows that \(\tilde{u}_\infty \colon \widetilde{\Sigma}'\to
  (-1,1)\times M\) is a constant map, but this contradicts the assumption
  defining Case III.

The case in which \(\gamma\) is a periodic orbit is treated similarly by
  holomorphically parameterizing \((-1,1)\times \gamma(S^1)\) by an
  annulus in \(\mathbb{C}\).
Again the maximum principle applies and we conclude that
  \(\tilde{u}_\infty \colon \widetilde{\Sigma}'\to (-1,1)\times M\) is a
  constant map which is impossible.
We conclude that Case III is impossible.\\

All cases are impossible, and hence this completes the proof by
  contradiction of Lemma
  \ref{LEM_non_trivial_limit_component_with_boundary}.
\end{proof}%%%%%%%%%%%%%%%%%%%%%%%%%%%%%%%%%%%%%%%%%%%%%%%%            END PROOF

With Lemma \ref{LEM_non_trivial_limit_component_with_boundary}
  established, we now observe that there exists a connected component
  \(\widetilde{\Sigma}'\subset \widetilde{\Sigma}_\infty\) for which
  \(\widetilde{\Sigma}' \cap (a\circ \tilde{u}_\infty)^{-1}(0)\neq
  \emptyset\) and \(\partial \widetilde{\Sigma}'\neq \emptyset\).
We make use of the fact that \(\tilde{\omega}^{s_0}\) evaluates
  non-negatively on \(\widetilde{J}_\infty\)-complex lines, together with
  the fact that
  \begin{align*}                                                          %% EQN
    \int_{\widetilde{\Sigma}_\infty}\tilde{u}_\infty^*\tilde{\omega}^{s_0}
    =0,
    \end{align*} 
    to conclude that \(\tilde{u}_\infty(\widetilde{\Sigma}')
    \subset (-1,1)\times \gamma(\mathbb{R})\) for some Hamiltonian
    trajectory \(\gamma\colon\mathbb{R} \to M\).
That is, \(\dot{\gamma}= X(\gamma)\) where \(\hat{\lambda}(X)= 1\) and
  \(i_X\tilde{\omega}^{s_0} = 0\).
If \(\gamma\) is not periodic, then the map 
 \begin{align*}                                                           %% EQN
    &\Phi\colon(-1,1)\times \mathbb{R}\to (-1, 1)\times M
    \\ 
    &\Phi(s,t) = \big(s, \gamma(t)\big)
    \end{align*}
  is an injective pseudoholomorphic immersion, and hence the map 
  \begin{align*}                                                          %% EQN
    &v\colon\widetilde{\Sigma}' \to (-1,1)\times \mathbb{R} \subset
    \mathbb{C}
    \\
    &v=\Phi^{-1}\circ \tilde{u}_\infty 
    \end{align*}
  is a non-constant holomorphic map from a connected compact Riemann
  surface \(\widetilde{\Sigma}'\) with non-empty boundary into
  \(\mathbb{C}\).
Moreover, this holomorphic map satisfies the following two conditions:
  \begin{enumerate}                                                       %% NUM
    \item \(v(\partial \widetilde{\Sigma}') \subset \big((-1,
      -\frac{1}{4})\cup (\frac{1}{4}, 1)\big) \times \mathbb{R}\).            
    \item \(v^{-1}(\{0\}\times \mathbb{R})\neq \emptyset\).
    \end{enumerate}
However, by the maximum principle, this is impossible.
We conclude that \(\gamma\) must be a periodic trajectory of the
  Hamiltonian vector field \(X\) which satisfies
  \begin{align*}                                                          %% EQN
    \hat{\lambda}(X)=1,  \qquad\text{and}\qquad i_X\tilde{\omega}^{s_0}
    =0.
    \end{align*}
Or in other words, for the symplectic manifold
  \((\mathcal{I}_\epsilon\times M, \Omega)\), and Hamiltonian function
  \(H(s,p) = s\), there exists a periodic Hamiltonian orbit on energy level
  \(\{s_0\}\times M\).
This completes the proof of Proposition
  \ref{PROP_bounded_F_yields_periodic_orbit}.
\end{proof}%%%%%%%%%%%%%%%%%%%%%%%%%%%%%%%%%%%%%%%%%%%%%%%%            END PROOF

Let us summarize the already established facts involving the map \(F\colon
  \mathcal{I}_{\varepsilon}\rightarrow [0,+\infty]\). 
Recall that \(F\) has been given as  \(F(s):=\liminf_{m\rightarrow\infty}
  F_m\), where the \(F_m\) have been previously defined in
  (\ref{EQ_auxiliary_F}) by
\begin{align}\label{EQ_FFF_definition}                                    %% EQN
    &F_m\colon \mathcal{I}_\epsilon \to [0, \infty) \subset \mathbb{R}
    \\
    &F_m(t)=
    \begin{cases}
    \int_{(s\circ u_{k_m})^{-1}(t)} u_{k_m}^*\hat{\lambda} &\text{if }t\in
    \mathcal{I}_\epsilon\setminus \mathcal{Y}\notag
    \\
    0 &\text{otherwise}.
    \end{cases}
    \end{align}
The \(\mathcal{Y}\) has been defined in (\ref{EQ_Y_definition}) and  we
  have shown the following.
\begin{itemize}
\item The set \(\{s\in \mathcal{I}_{\varepsilon}\ |\ F(s)=\infty\}\) has
measure zero.
\item The set \(L\) is countable and the set \(\mathcal{Y}\) has measure
zero.
\item If \(s\in \{\mathcal{I}_{\varepsilon}\ |\ F(s)<\infty\}\setminus
(L\cup\mathcal{Y}) \) then \( \{s\}\times M\) contains a periodic orbit.
\end{itemize} 
We are now prepared to prove the following.

\setcounter{CurrentSection}{\value{section}}
\setcounter{CurrentTheorem}{\value{theorem}}
\setcounter{section}{\value{CounterSectionMain}}
\setcounter{theorem}{\value{CounterTheoremMain}}
%%%%%%%%%%%%%%%%%%%%%%%%%%%%%%%%%%%%%%%%%%%%%%%%%%%%%%%%%%%%%%%%%%%%%%%%%%%%%%%%
%%%%%%%%%%                           THEOREM                           %%%%%%%%%
%%%%%                                                                       %%%%
\begin{theorem}[Main Result]
  \hfill\\
Let \((W, \Omega)\) be a symplectic manifold without boundary, and let
  \(H\colon W\to \mathbb{R}\) be a smooth proper\footnote{By proper here, we
  mean that for each compact set \(\mathcal{K}\subset \mathbb{R}\), the set
  \(H^{-1}(\mathcal{K})\) is compact.} Hamiltonian.
Fix \(E_-, E_+ \in H(W) \subset \mathbb{R}\) with \(E_- < E_+\), as well
  as positive constants, \(C_g>0\), and \(C_\Omega>0\).
Suppose that for each \(\Omega\)-tame almost complex structure \(J\) on
  \(W\) there exists a proper pseudoholomorphic map 
  \begin{align*}                                                          %% EQN
      u:(S,j)\to \{p\in W : E_- < H(p) < E_+ \} 
      \end{align*}
  without boundary, which also satisfies the following conditions:
  \vspace{5pt}
\begin{enumerate}                                                         %% NUM
  \item \emph{({genus and area bounds})}
  The following inequalities hold:
  \begin{align*}                                                          %% EQN
    {\rm Genus}(S) \leq C_g\qquad\text{and}\qquad \int_{S} u^*\Omega
    \leq C_\Omega.
    \end{align*}
  \item \emph{({energy surjectivity})}
  The map \(H\circ u : S\to (E_- , E_+)\) is
  surjective.\vspace{5pt}
  \end{enumerate}
Then there is a periodic Hamiltonian orbit on almost every energy level in
  range \((E_-, E_+)\).
That is, if we let \(\mathcal{I}\subset (E_-, E_+)\) denote the energy
  levels of \(H\) which contain a Hamiltonian periodic orbit, then
  \(\mathcal{I}\) has full measure:
  \begin{align*}                                                          %% EQN
    \mu(\mathcal{I}) = \mu\big((E_-, E_+)\big) = E_+ - E_-.
    \end{align*}
\end{theorem}

%%%%%                                                                       %%%%
%%%%%%%%%%                                                             %%%%%%%%%
%%%%%%%%%%%%%%%%%%%%%%%%%%%%%%%%%%%%%%%%%%%%%%%%%%%%%%%%%%%%%%%%%%%%%%%%%%%%%%%%
%
\begin{proof}
\setcounter{section}{\value{CurrentSection}}
\setcounter{theorem}{\value{CurrentTheorem}}
In order to prove this result, we will make use of our localization
  results, but first we need to properly reframe the problem.
For notational convenience we begin by defining:
  \begin{align*}                                                          %% EQN
    \widetilde{W} = \{p\in W : E_- < H(p) < E_+\}.
    \end{align*}
Next, for each \(c\in \mathbb{R}\), we define the function
  \begin{align*}                                                          %% EQN
    &H^c\colon \widetilde{W}\to \mathbb{R}\\
    &H^c(q) = H(q)-c
    \end{align*}
  and observe that \(H\) and \(H^c\) generate identical Hamiltonian vector
  fields on \(\widetilde{W}\).
Consequently, \(\gamma\colon \mathbb{R}\to \widetilde{W}\) is a
  Hamiltonian periodic obit of \(H\) if and only if it is a Hamiltonian
  periodic orbit of \(H^c\).
Next we make the following claim.\\

\noindent{\bf Claim:} \emph{To prove Theorem \ref{THM_main_maybe}, it is
  sufficient to show that for each regular value \(c\in (E_-, E_+)\) of
  \(H\), there exists a \(\delta=\delta(c)>0\) such that the set of energy
  levels \(\{|H^c| < \delta\}\) containing a Hamiltonian periodic orbit has
  measure \(2\delta\).}\\

To see that this claim is true, we first consider the case that \(H\) has
  no critical points in \(\widetilde{W}\).
In this case, every \(c\in (E_-, E_+)\) is a regular energy value, and
  thus for each such \(c\) we define \(\delta_c=\delta(c)\) so that the
  set of energy levels \(\{|H^c| < \delta_c\}\) containing a Hamiltonian
  periodic orbit has measure \(2\delta_c\).\\
It follows that \(\{(c-\delta_c, c+\delta_c)\}_{E_- < c < E_+}\) is an
  open cover of \((E_-, E_+)\).
Using the fact that this is an open cover, together with the fact that the
  open interval \((E_-, E_+)\) can be written as the
  countable union of compact intervals, for example 
  \begin{align*}                                                          %% EQN
    (E_-, E_+) = \bigcup_{\mathcal{I}\in \mathcal{E}}\mathcal{I} \quad
    \text{where} \quad \mathcal{E} = \Big\{[E_-+{\textstyle\frac{1}{n}}L,
    E_+- {\textstyle \frac{1}{n}}L]\Big\}_{n\in \mathbb{N}}\quad\text{and}
    \textstyle \quad L = {\textstyle\frac{E_+-E_-}{4}}, 
    \end{align*}
  it follows that there exists a countable set \(\{c_i\}_{i\in
  \mathbb{N}}\) such that
  \begin{align*}                                                          %% EQN
    (E_-, E_+) = \bigcup_{i\in \mathbb{N}} (c_i - \delta_{c_i} , c_i
    +\delta_{c_i}).
    \end{align*}
That the set of energy levels in \((E_-, E_+)\) has measure \(E_+-E_-\)
  then follows essentially from countable additivity.
More specifically, if \(\Xi\) denotes those energy levels in \((E_-,
  E_+)\) which have a periodic orbit, and \(A_i = (c_i-\delta_{c_i},
  c_i+\delta_{c_i})\), and \(B_n = A_n\setminus
  \cup_{i=1}^{n-1} A_i\), then
  \begin{align*}                                                          %% EQN
    \mu(\Xi) 
    = \mu\Big(\Xi \cap \bigcup_{n=1}^\infty B_n\Big)
    = \sum_{n=1}^\infty\mu\Big(\Xi \cap B_n\Big)
    = \sum_{n=1}^\infty\mu\Big(B_n\Big)
    = \mu\Big(\bigcup_{n=1}^\infty B_n\Big)
    = E_+-E_-.
    \end{align*}
The case that \(H\colon \widetilde{W}\to \mathbb{R}\) has critical points
  is treated similarly, by first observing that the set of critical values
  of \(H\) is closed and has measure zero.
Thus the set of regular values of \(H\) has full measure and can be
  written as the disjoint union of open intervals.
The claim is then established by another application of countable
  additivity.

With the claim established, we can now apply Lemma
  \ref{LEM_localization_to_framed_Hamiltonian_energy_pile} to the
  Hamiltonian \(H^c \colon \widetilde{W}\to \mathbb{R}\) for each regular
  value \(c\) of \(H\).
For each such \(c\), this establishes a framed Hamiltonian energy pile
  \((\mathcal{I}_{\epsilon^c}\times M, \Omega^c, \hat{\lambda}^c)\), and a
  diffeomorphism
  \begin{align*}                                                          %% EQN
    \Phi^c\colon \mathcal{I}_{\epsilon^c}\times M\to \{|H^c| <
    \epsilon^c\}
    \end{align*}
  for which \(H^c\circ \Phi^c(s,p) = s\) and \((\Phi^c)^*\Omega =
  \Omega^c\).
Moreover, we obtain associated structures \(\hat{\rho}^c\), \(\xi^c\) and
  \(\hat{\omega}^c\) on \(\mathcal{I}_{\epsilon^c}\times M\), and Lemma
  \ref{LEM_localization_to_framed_Hamiltonian_energy_pile} also guarantees
  that along \(\{0\}\times M\) we have \(\hat{\lambda}^c(X_{H^c})=1\) and
  also along \(\{0\}\times M\) the sub-bundles \(\hat{\rho}^c\) and
  \(\xi^c\) are symplectic complements.
We then observe that in light of the above claim we have just established,
  it follows that in order to complete the proof of Theorem
  \ref{THM_main_maybe} it is sufficient to show that for each such \(c\),
  there exists a \(\delta^c >0 \) such almost every energy level of the
  framed Hamiltonian energy pile \((\mathcal{I}_{\delta^c}\times M,
  \Omega^c, \hat{\lambda}^c)\) has a Hamiltonian periodic orbit.

To find such a \(\delta^c>0\), the aim will be to apply Proposition
  \ref{PROP_adiabatically_degenerating_constructions}. However to do that we
  must first have at our disposal a sequence of suitable almost complex
  structures.
To construct these, we start by defining an almost complex \(J^c\)
  on \(\mathcal{I}_{\epsilon^c}\times M\) by requiring that \(J^c
  \partial_s =\widehat{X}^c \) and that \(J^c\colon \xi^c\to \xi^c\) have
  the property that \(J_{\xi^c}:= J^c\big|_{\xi^c}\) be translation
  invariant, and that \(\hat{\omega}^c\circ ({\rm Id}\times J_{\xi^c})\) is
  symmetric and positive definite.
We then treat \(J^c\) as a constant sequence and apply Proposition
  \ref{PROP_adiabatically_degenerating_constructions} which guarantees the
  existence of an \(\ell^c\) (stated in the proposition as \(\epsilon\))
  with the following significance.
Let \(\phi_k^c\colon \mathcal{I}_{\ell^c} \to (0, 1]\) be a sequence of
  functions which converge in \(C^\infty\) to a limit function
  \(\phi_\infty^c\) which satisfies
\begin{align*}                                                            %% EQN
  \phi_\infty^c(s) =
  \begin{cases}
  1 & \text{if } |s|\geq {\textstyle \frac{1}{2}}\ell^c\\ 
  0 & \text{if } |s|\leq {\textstyle \frac{1}{4}}\ell^c.
  \end{cases}
  \end{align*}
Then a consequence of Proposition
  \ref{PROP_adiabatically_degenerating_constructions} is that the almost
  complex structures defined by
  \begin{align*}                                                          %% EQN
    \phi_k^c \cdot J_k \partial_s = \widehat{X}^c\qquad\text{and}\qquad
    J_k\big|_{\xi^c} = J_{\xi^c}
    \end{align*}
  are each \(\Omega^c\)-tame on \(\mathcal{I}_{\ell^c}\times M\), and they
  are adiabatically degenerating on \(\mathcal{I}_{\ell^c/4}\times M\).
We define \(\delta^c:= \frac{1}{4}\ell^c\).
We then observe that because these almost complex structures are all tame,
  on \(\mathcal{I}_{\ell^c}\times M\) they can be seen as arising as \(J_k
  = (\Phi^c)^*\widetilde{J}_k\) for some \(\Omega\)-tame almost complex
  structures \(\widetilde{J}_k\) on \(\widetilde{W}\).
The hypotheses of Theorem \ref{THM_main_maybe} then guarantee the
  existence of a sequence of pseudoholomorphic curves with bounded symplectic
  area and genus, and which span the energy levels
  in \(\mathcal{I}_{\delta^c}\times M\), while the almost complex structures
  are adiabatically degenerating on \(\mathcal{I}_{\delta^c}\times M\).
Recall that second countability of the domains of these pseudoholomorphic
  curves guarantees that the set of connected components on which each is a
  constant map is countable, and hence for any such curve the set of energy
  levels containing a constant component is countable.
Making use of the fact that our pseudoholomorphic maps are continuous and
  proper, it follows that one may remove the constant connected components
  while still guaranteeing ``energy surjectivity.''
Thus after removing constant components, we may apply Theorem
  \ref{THM_localized_almost_existence}, which guarantees that the set of
  energy levels with periodic orbits has full measure in \((-\delta^c,
  \delta^c)\).
Because we have reduced the proof of Theorem \ref{THM_main_maybe} to
  establishing just this result, we see that we have completed the proof of
  Theorem \ref{THM_main_maybe}.
\end{proof}%%%%%%%%%%%%%%%%%%%%%%%%%%%%%%%%%%%%%%%%%%%%%%%%            END PROOF

With Theorem \ref{THM_main_maybe} established, we now prove Theorem 
  \ref{THM_existence_and_almost_existence}.

\setcounter{CurrentSection}{\value{section}}
\setcounter{CurrentTheorem}{\value{theorem}}
\setcounter{section}{\value{CounterSectionIntertwine}}
\setcounter{theorem}{\value{CounterTheoremIntertwine}}
%%%%%%%%%%%%%%%%%%%%%%%%%%%%%%%%%%%%%%%%%%%%%%%%%%%%%%%%%%%%%%%%%%%%%%%%%%%%%%%%
%%%%%%%%%%                           THEOREM                           %%%%%%%%%
%%%%%                                                                       %%%%
\begin{theorem}[intertwining existence and almost existence]
  \hfill\\
Let \((W, \Omega)\) be a four-dimensional compact connected exact symplectic
  manifold with boundary \(\partial W = M^+ \cup M^-\).
Suppose \(M^+\) is positive contact type in the sense of Definition
  \ref{DEF_positive_contact_type_boundary}, and suppose that one of the
  following three conditions holds:
\begin{enumerate}                                                         %% NUM
  \item 
  \(M^+\) has a connected component diffeomorphic to \(S^3\),
  \item 
  there exists an embedded \(S^2\subset M^+\) which is homotopically
  nontrivial in \(W\),
  \item 
  \((M^+, \lambda)\) has a connected component which is overtwisted.
  \end{enumerate}
Then for each Hamiltonian \(H\in C^\infty(W)\) for which
  \(H^{-1}(\pm 1) = M^\pm\), the following is true.
For each \(s\in [-1,1]\) the energy level \(H^{-1}(s)\) contains a closed
  non-empty set other than the energy level \(H^{-1}(s)\) itself which is
  invariant under the Hamiltonian flow of \(X_H\); moreover for almost every
  \(s\in [-1, 1]\) this closed invariant subset is a periodic orbit.
\end{theorem}

%%%%%                                                                       %%%%
%%%%%%%%%%                                                             %%%%%%%%%
%%%%%%%%%%%%%%%%%%%%%%%%%%%%%%%%%%%%%%%%%%%%%%%%%%%%%%%%%%%%%%%%%%%%%%%%%%%%%%%%
%
\begin{proof}
\setcounter{section}{\value{CurrentSection}}
\setcounter{theorem}{\value{CurrentTheorem}}
First observe that if \(s_0\in [-1, 1]\) is a critical value of \(H\),
  then there exists \(p \in H^{-1}(s_0)\) such that \(d H(p) = 0\), and
  hence the constant trajectory
  \begin{align*}                                                          %% EQN
    &\gamma\colon \mathbb{R}\to H^{-1}(s_0)
    \\
    &\gamma(t) = p
    \end{align*}
  is periodic orbit.
Also note that because \(M^+ = H^{-1}(1)\) is contact type, and is either
  \(S^3\), overtwisted, or contains a homotopically non-trivial \(S^2\), it
  follows from \cite{H93} that \(M^+\) has a periodic orbit.
Then for any regular value \(s_0\in [-1, 1)\), it follows from \cite{FH2}
  (specifically Theorem 2) that the flow of \(X_H\) on \(H^{-1}(s_0)\) is
  not minimal.
This establishes that each energy level \(H^{-1}(s_0)\) with \(s\in [-1,
  1]\) is not minimal.
To complete the proof of Theorem
  \ref{THM_existence_and_almost_existence}, it remains to show that almost
  every energy level \(H^{-1}(s)\) for \(s\in (-1, 1)\) has a
  periodic orbit.
This follows from Thereom \ref{THM_main_maybe} above, provided we can
  guarantee the existence of the required pseudoholomorphic curves for
  arbitrary tame almost complex structure.
However, this is fairly standard and the relevant details are provided in
  \cite{FH2} (specifically the proof of Theorem 2), however we recall the
  key points here.

Choose an \(\Omega\)-tame almost complex structure \(J\) on \(W\), and fix
  regular values \(E_-\) and \(E_+\) of \(H\) so that
  \begin{align*}                                                          %% EQN
    -1 < E_- < E_+ < 1.
    \end{align*}
Consider the case that tight \(S^3\) is a connected component of \(M^+\).
Then along this spherical component, \(W\) can be symplectically capped
  off by \(\mathbb{C}P^2\setminus \mathcal{O}\) where \(\mathcal{O}\) is
  diffeomorphic to the four-ball, and the resulting symplectic manifold we
  denote by \((\widetilde{W}, \widetilde{\Omega})\).
Note that in order to guatantee that \(\widetilde{\Omega}\) is indeed
  symplectic, one may need to require that its restriction to
  \(\mathbb{C}P^2\setminus \mathcal{O} \subset \widetilde{W}\) be a large
  constant multiple of the Fubini-Study metric.
By adjusting \(J\) in a neighborhood of \(S^3 \subset M^+\), the almost
  complex structure \(J\) can be extended to an \(\widetilde{\Omega}\)-tame
  almost complex structure \(\widetilde{J}\) on \(\widetilde{W}\), which
  is also standard in a neighborhood of \(\mathbb{C}P^1 \subset
  \mathbb{C}P^2\setminus \mathcal{O}\).
Define the constant \(C_{\widetilde{\Omega}}:= \int_{\mathbb{C}P^1}
  \widetilde{\Omega}\) to be the \(\widetilde{\Omega}\)-area of this sphere
  at infinity.
One then considers the connected component of the moduli space of
  \(\widetilde{J}\)-pseudoholomorphic curves which contain this
  \(\mathbb{C}P^1 \subset \mathbb{C}P^2\setminus \mathcal{O} \subset
  \widetilde{W}\).
As detailed in \cite{FH2}, automatic transversality guarantees that this
  four real dimensional moduli space (of unparameterized curves) is cut
  out transversely, each distinct pair of curves intersects at exactly
  one point, and by requiring \(\Omega\) to be exact there cannot be any
  ``bubbles'' that arise in the compactification.
Homotopy invariance of intersection numbers guarantees that each curve in
  this moduli space intersects \(\mathbb{C}P^1 \subset
  \mathbb{C}P^2\setminus \mathcal{O} \subset \widetilde{W}\).
It remains to show that this family of curves satisfies the energy
  surjectivity condition; that is, that this family of curves extends from
  \(M^+ = H^{-1}(1)\) to every energy level \(H^{-1}(s)\) for \(s\in
  (-1,1)\).
This too is detailed in \cite{FH2}, and follows from a mixture of
  automatic transversality and compactness of the family of curves; the
  latter is guaranteed by the fact that \(\widetilde{J}\) is
  \(\widetilde{\Omega}\)-tame.
This then guarantees energy surjectivity.
That is, for each \(\Omega\)-tame almost complex structure \(J\), there
  exists a pseudoholomorphic map \(\tilde{u}:(\widetilde{S},j)\to
  (\widetilde{W},\widetilde{J})\) with the property that the restricted map
  \begin{align*}                                                          %% EQN
    u:=\tilde{u}\big|_{S}\qquad \text{where} \qquad S:= \big\{z\in
    \widetilde{S} : H\circ \tilde{u}(z) \in (E_-, E_+)\big\}
    \end{align*}
  satisfies  \(H\circ u(S) = (E_-, E_+)\).
It is also easily checked that \({\rm genus}(S) = 0\) and
  \(\int_{S}u^*\widetilde{\Omega} \leq C_{\widetilde{\Omega}}\).
Consequently, Theorem \ref{THM_main_maybe} applies, and hence almost every
  energy level in \((E_-, E_+)\) contains a periodic orbit.
By letting \(E_-\to -1\) and \(E_+\to 1\), the desired result is immediate.

This covers the tight \(S^3\) case; the overtwisted case and the
  homotopically non-trivial \(S^2\subset M^+\) case are very similar,
  although the mechanism to generate the curves is different.
The reader is directed to \cite{FH2} for the details.

\end{proof}%%%%%%%%%%%%%%%%%%%%%%%%%%%%%%%%%%%%%%%%%%%%%%%%            END PROOF

\appendix

\section{Miscellaneous Support}

This section is mostly devoted to providing a few support definitions and
  results which are important but are otherwise a bit of a distraction from
  more important arguments.

\subsection{Supporting Proofs}\label{SEC_supporting_proofs}

Our primary goal of this section is to prove Proposition
  \ref{PROP_adiabatically_degenerating_constructions}, however to do so we
  must first establish a few important supporting lemmata.
The first of these is the following.

%%%%%%%%%%%%%%%%%%%%%%%%%%%%%%%%%%%%%%%%%%%%%%%%%%%%%%%%%%%%%%%%%%%%%%%%%%%%%%%%
%%%%%%%%%%                            LEMMA                            %%%%%%%%%
%%%%%                                                                       %%%%
\begin{lemma}[cross term control]
  \label{LEM_cross_term_control}
  \hfill\\
Let \(\big(\mathcal{I}_{\ell}\times M, \Omega, \hat{\lambda}\big)\)
  be a framed Hamiltonian energy pile, and let \(\hat{\rho}\)  and \(\xi\)
  be the associated structures defined in Section \ref{SEC_background},
  and assume as in the conclusions of Lemma
  \ref{LEM_localization_to_framed_Hamiltonian_energy_pile} that along
  \(\{0\}\times M\), the sub-bundles \(\hat{\rho}\) and \(\xi\) are
  \(\Omega\)-symplectic complements.
Let \(J_0\) be an \(\Omega\)-compatible almost complex structure on
  \(\mathcal{I}_{\ell}\times M\), with associated Riemannian metric
  \(g_0= \Omega\circ ({\rm Id}\times J_0) \).
Then for each \(\delta>0\), there exists an \(\epsilon=\epsilon(\delta,
  \Omega, \hat{\lambda}, \hat{\omega}, J_0) \in (0,\ell)\), with the
  property that for each \(v = v^{\hat{\rho}}  + v^\xi \in
  T(\mathcal{I}_{\epsilon} \times M)\) with \(v^{\hat{\rho}}\in
  \hat{\rho}\) and \(v^\xi\in \xi\) we have
  \begin{align*}                                                          %% EQN
    |\Omega(v^{\hat{\rho}}, v^\xi)| \leq \delta \|v^{\hat{\rho}}\|_{g_0}
    \|v^\xi\|_{g_0}.
    \end{align*}
\end{lemma}
%%%%%                                                                       %%%%
%%%%%%%%%%                                                             %%%%%%%%%
%%%%%%%%%%%%%%%%%%%%%%%%%%%%%%%%%%%%%%%%%%%%%%%%%%%%%%%%%%%%%%%%%%%%%%%%%%%%%%%%
%
\begin{proof}
Given $s\in \mathcal{I}_{\ell}$ we consider the embedding
\[
M\rightarrow \mathcal{I}_{\ell}\times M\colon m\rightarrow (s,m)
\]
and the pull-backs of $\hat{\rho}\rightarrow \mathcal{I}_{\ell}\times M$ and
  \(\xi\to \mathcal{I}_{\ell}\times M\), respectively. We denote them by
  \(\hat{\rho}_s\) and \(\xi_s\). 
Then the pull-back of the tangent space \( T(\mathcal{I}_{\ell}\times M)\)
  to \(M\) by the same map has the direct sum decomposition
  $\hat{\rho}_s\oplus\xi_s$.  In the case of $s=0$ this decomposition is
  $\Omega$-orthogonal.
We define the function
\begin{align*}                                                            %% EQN
  &\Theta\colon (-\ell,\ell) \to [0, \infty)
  \\ 
  &\Theta(s):=\sup_{\substack{ v \in \hat{\rho}_s \setminus \{0\}\\ w \in
  \xi_s \setminus \{0\}}} \frac{|\Omega(v,w)|}{\|v\|_{g_0} \|w\|_{g_0}}.
  \end{align*}
It is straightforward to show that \(\Theta\) is continuous and  
  and that \(\Theta(0)=0\). 
By compactness 
\begin{align*}                                                            %% EQN
  &\tilde{\Theta}\colon [0,\ell)\rightarrow {\mathbb R}\\
  &\varepsilon\rightarrow \max_{s\in [-\varepsilon,\varepsilon]} \Theta(s)
  \end{align*}

%\[
%\tilde{\Theta}\colon [0,\ell)\rightarrow {\mathbb R}\colon
%\varepsilon\rightarrow \max_{s\in [-\varepsilon,\varepsilon]} \Theta(s)\]
  is also continuous and \(\tilde{\Theta}(0)=0\).
The desired result is then immediate.
\end{proof}%%%%%%%%%%%%%%%%%%%%%%%%%%%%%%%%%%%%%%%%%%%%%%%%            END PROOF

In order to proceed further, we will need to define a certain metric bound
  on the geometry of a weakly adapted almost complex structure.
We denote this quantity \(|\rangle J\langle |\), and establish it as
  follows.
Given a framed Hamiltonian energy pile
  \[
  (\mathcal{I}_{\ell}\times M, \Omega, \hat{\lambda}),
  \]
  we first fix a background metric \(g_0\) on \(\mathcal{I}_{\ell}\times M\)
  associated to an auxiliary \(\Omega\)-compatible almost complex
  structure \(J_0\) by the usual formula \(g_0:=\Omega\circ ({\rm Id}
  \times J_0)\).  
Of course, near the ends of \(\mathcal{I}_{\ell}\times M\)  the metric
  $g_0$ might not behave well.
Our Hamiltonian energy pile comes with the structures \(\hat{\omega}\) and
  \(\xi\).
Given any other weakly adapted almost complex structure \(J\) we define
  the following quantity, where \(\ell'\in (0,\ell)\):
  \begin{align}\label{EQ_contrived}                                       %% EQN
    &|\rangle J\langle |_{(\mathcal{I}_{\ell'} \times M, \Omega, 
    \hat{\lambda}, \hat{\omega}, J_0)} := 
    \sup_{\substack{q\in \mathcal{I}_{\ell'}\times M\\ v_q^\xi\in
    \xi_q\setminus \{0\} }}
    {\rm max}\Bigg( 
    \frac{\|v_q^\xi\|_{g_J}}{\|v_q^\xi\|_{g_0}},
    \frac{\|v_q^\xi\|_{g_0}}{\|v_q^\xi\|_{g_J}},
    \frac{\|Jv_q^\xi\|_{g_0}}{\|v_q^\xi\|_{g_0}},
    \frac{\|v_q^\xi\|_{g_0}}{\|Jv_q^\xi\|_{g_0}}
    \Bigg);
    \end{align}
  here \(g_J\) is the metric as defined in equation
  (\ref{EQ_Riemannian_metric}) in Definition \ref{DEF_weakly_adapted_J}.

An immediate benefit of such a definition is that whenever it is the case
  that \(|\rangle J\langle | \leq C_1\), it is also the case that the
  following hold for each \(v^\xi \in \xi\) lying above a point in
  \(\mathcal{I}_{\ell'}\times M\)
  \begin{align}                                                           %% EQN
    \|v^\xi\|_{g_0} &\leq C_1 \|v^\xi\|_{g_J}\label{EQ_g0_gJ}\\
    \|v^\xi\|_{g_J} &\leq C_1 \|v^\xi\|_{g_0}\\
    \|Jv_q^\xi\|_{g_0} &\leq C_1 \|v_q^\xi\|_{g_0} \label{EQ_Jg0_g0}\\
    \|v_q^\xi\|_{g_0}  &\leq C_1 \|Jv_q^\xi\|_{g_0}.
    \end{align}
It is also worth noting that if \(J_1\) and \(J_2\) are weakly adapted
  almost complex structures which agree on \(\xi\), then \( |\rangle
  J_1\langle | =|\rangle J_2\langle | \). 
Indeed, the quantity \(|\rangle J\langle |\) depends only on
  \(J\big|_\xi\).
With this definition established, we can now proceed with an important
  application.

%%%%%%%%%%%%%%%%%%%%%%%%%%%%%%%%%%%%%%%%%%%%%%%%%%%%%%%%%%%%%%%%%%%%%%%%%%%%%%%%
%%%%%%%%%%                         PROPOSITION                         %%%%%%%%%
%%%%%                                                                       %%%%
\begin{proposition}[metric area controlled by symplectic area]
  \label{PROP_metric_symplectic}
  \hfill\\
Let \(\big(\mathcal{I}_{\ell}\times M, \Omega, \hat{\lambda}\big)\) be a
  framed Hamiltonian energy pile, and let \(\hat{\rho}\), \(\xi\), and
  \(\hat{\omega}\) be the associated structures defined in equations
  (\ref{Hope}), (\ref{EQQQ_xi}), (\ref{EQ_xi}), and assume,
  as in the conclusions of Lemma
  \ref{LEM_localization_to_framed_Hamiltonian_energy_pile}, that along
  \(\{0\}\times M\) the sub-bundles \(\hat{\rho}\) and \(\xi\) are
  symplectic complements and \(\hat{\lambda}(X_H)=1\).
Let \(J_0\) be an auxiliary \(\Omega\)-compatible almost complex structure
  on \(\mathcal{I}_{\ell}\times M\), with associated Riemannian metric
  \(g_0=\Omega\circ({\rm Id} \times J_0)\).
Fix a large positive constant \(C_1>1\). 
Then there exists a number \(\epsilon=\epsilon(\Omega, \hat{\lambda},
  \hat{\omega}, J_0) \in (0, \ell)\) with the property that after trimming
  \(\mathcal{I}_{\ell}\times M\) to \(\mathcal{I}_\epsilon\times M\), the
  following holds.
If \(J\) is a weakly adapted almost complex structure for which 
  \begin{align*}                                                          %% EQN
    |\rangle J\langle |_{(\mathcal{I}_\epsilon \times M, \Omega,
    \hat{\lambda}, \hat{\omega}, J_0)}  \leq C_1
    \end{align*}
  in the sense of equation (\ref{EQ_contrived}), then 
  \begin{align*}                                                          %% EQN
    (ds \wedge \hat{\lambda} + \hat{\omega})(v, Jv )\leq 2 \Omega(v , J
    v)
    \end{align*}
  for each \(v\in T(\mathcal{I}_\epsilon \times M)\).  
Here $J$ only needs to be defined over \(\mathcal{J}_{\varepsilon}\times M
  \).
\end{proposition}
%%%%%                                                                       %%%%
%%%%%%%%%%                                                             %%%%%%%%%
%%%%%%%%%%%%%%%%%%%%%%%%%%%%%%%%%%%%%%%%%%%%%%%%%%%%%%%%%%%%%%%%%%%%%%%%%%%%%%%%
%
The important fact is that \(\varepsilon\) does not depend on \(J\) and
  only on the numerical bound!
\begin{proof}
We first fix $\ell'\in (0,\ell)$. 
This is done since the background metric $g_0$ might not be well-behaved
  near the ends of \(\mathcal{I}_{\ell}\times M\).
We begin by defining the constants \(C_2\) and \(\delta\) by
  \begin{align}\label{EQ_partial_s_norm_0}                                %% EQN
    C_2:=\sup_{q\in \mathcal{I}_{\ell'}\times M} \|\partial_s\|_{g_0}
    \qquad\text{and}\qquad \delta :=\frac{1}{8 C_2 C_1^2}.
    \end{align}
Let \(\epsilon=\epsilon(\delta, \Omega, \hat{\lambda}, \hat{\omega},
  J_0)\) be the constant guaranteed by Lemma \ref{LEM_cross_term_control}
  and we may assume without loss of generality that
  \(0<\varepsilon\leq \ell'\).
Recall that a consequence of Lemma \ref{LEM_cross_term_control} is that
  for each \(v = v^{\hat{\rho}}  + v^\xi \in T(\mathcal{I}_{\epsilon}
  \times M)\) with \(v^{\hat{\rho}}\in \hat{\rho}\) and \(v^\xi\in \xi\)
  we have
  \begin{align}\label{EQ_Omega_delta}                                     %% EQN
    |\Omega(v^{\hat{\rho}}, v^\xi)| \leq \delta \|v^{\hat{\rho}}\|_{g_0}
    \|v^\xi\|_{g_0}.
    \end{align}
Also recall that along \(\{0\}\times M\) we have \(\hat{\lambda}(X_H)=1\),
  so that by shrinking \(\epsilon>0\), we can further guarantee that 
  \begin{align}\label{EQ_lambda_near_one}                                 %% EQN
    \sup_{q\in \mathcal{I}_\epsilon\times M}
    \big|\hat{\lambda}(X_H(q))-1\big| \leq \frac{1}{2^{10}}.
    \end{align}
Next recall that
  \begin{align*}                                                            %% EQN
    \widehat{X}=\frac{X_H}{\hat{\lambda}(X_H)},
    \qquad
    J \partial_s = {\textstyle \frac{1}{\phi}}
    \widehat{X},
    \qquad\text{and}\qquad
    J\widehat{X} = -\phi
    \partial_s.
    \end{align*}
Additionally, recall that we have the splitting \(T(\mathcal{I}_\epsilon
  \times M) = \hat{\rho}\oplus \xi\), and the associated projections
  \(\pi^{\hat{\rho}}\colon \hat{\rho}\oplus \xi \to \hat{\rho}\) and
  \(\pi^\xi\colon \hat{\rho}\oplus \xi \to \xi\).
In general we will use the abbreviated notation \(v^{\hat{\rho}} =
  \pi^{\hat{\rho}}(v)\) and \(v^\xi = \pi^\xi(v)\).  In the following 
  we work with $\hat{\rho}\oplus\xi\rightarrow
  \mathcal{I}_{\varepsilon}\times M$, where $\varepsilon>0$ is the
  previously chosen number.
We begin with the following claim.
\begin{align}\label{EQ_dslambda_v_Omega}                                  %% EQN
  (ds\wedge\hat{\lambda})(v^{\hat{\rho}},Jv^{\hat{\rho}}) =
  \hat{\lambda}(X_H) \cdot \Omega(v^{\hat{\rho}}, Jv^{\hat{\rho}})
  \quad\text{ for each }v^{\hat{\rho}}\in\hat{\rho}.
  \end{align}
To prove this, we first write \(v^{\hat{\rho}}= a \partial_s + b
  \widehat{X}\) for some \(a,b\in \mathbb{R}\), and then compute as
  follows.
\begin{align*}                                                            %% EQN
  \Omega(v^{\hat{\rho}}, Jv^{\hat{\rho}}) 
  &=
  \Omega\big(a\partial_s +b \widehat{X}, J(a \partial_s + b
  \widehat{X})\big)
  \\
  &=
  \Omega\big(a\partial_s +b \widehat{X}, {\textstyle \frac{1}{\phi}} a
  \widehat{X}  - \phi b\partial_s \big)
  \\
  &=
  \Big(\frac{a^2}{\phi} + \phi b^2\Big) \; \Omega(\partial_s ,\widehat{X})
  \\
  &=
  \Big(\frac{a^2}{\phi} + \phi b^2\Big) \; \frac{\Omega(\partial_s
  ,X_H)}{\hat{\lambda}(X_H)}
  \\
  &=
  \Big(\frac{a^2}{\phi} + \phi b^2\Big) \; \frac{d
  H(\partial_s)}{\hat{\lambda}(X_H)}
  \\
  &=
  \Big(\frac{a^2}{\phi} + \phi b^2\Big) \; \frac{d
  s(\partial_s)}{\hat{\lambda}(X_H)}
  \\
  &=
  \Big(\frac{a^2}{\phi} + \phi b^2\Big)\frac{1}{\hat{\lambda}(X_H)} .
  \end{align*}
Similarly, we compute the following (with the same \(v^{\hat{\rho}}\in
  \hat{\rho}\) as above).
  \begin{align*}                                                          %% EQN
    (ds \wedge \hat{\lambda}) (v^{\hat{\rho}}, Jv^{\hat{\rho}}) 
    &=
    (ds \wedge \hat{\lambda})\big(a\partial_s +b \widehat{X}, J(a \partial_s + b
    \widehat{X})\big)
    \\
    &=
    (ds \wedge \hat{\lambda})\big(a\partial_s +b \widehat{X}, {\textstyle
    \frac{1}{\phi}} a \widehat{X} - \phi b\partial_s \big)
    \\
    &=
    \Big(\frac{a^2}{\phi} + \phi b^2\Big) 
    \\
    &=
    \hat{\lambda}(X_H) \; \Omega(v^{\hat{\rho}}, Jv^{\hat{\rho}}).
    \end{align*}
This establishes equation (\ref{EQ_dslambda_v_Omega}).
The case with \(\hat{\omega}\) is much easier.
Indeed, recall that by definition we have \(\hat{\omega} = \Omega \circ
  (\pi^\xi\times \pi^\xi)\).
It immediately then follows that
  \begin{align}\label{EQ_omega_vs_Omega}                                  %% EQN
    \hat{\omega}(v^\xi,Jv^\xi) = \Omega(v^\xi,Jv^\xi)\qquad \text{for each
    }v^\xi\in \xi.
    \end{align}
In just a moment we will be concerned with estimating cross terms, however
  first we will need an elementary estimate.  Starting with
  \begin{align*}                                                          %% EQN
    \|\partial_s\|_{g_J}^2 
    \; \; =\; \; 
    (ds \wedge \hat{\lambda})(\partial_s, J\partial_s)
    \; \; = \; \; 
   ( ds \wedge \hat{\lambda})(\partial_s, {\textstyle
    \frac{1}{\phi}}\widehat{X})
    \; \; = \; \; 
    {\textstyle \frac{1}{\phi}}.
    \end{align*}
    and combining the above  with the definition of \(C_2\) in equation
  (\ref{EQ_partial_s_norm_0})  yields
  \begin{align}\label{EQ_partial_s_estimate}                              %% EQN
    \|\partial_s\|_{g_0}^2 \leq C_2^2 = \phi C_2^2 \|\partial_s\|_{g_J}^2
    \leq C_2^2 \|\partial_s\|_{g_J}^2.
    \end{align}
We are now prepared to estimate cross terms.
To that end, we let \(v = v^{\hat{\rho}}+ v^\xi \in T(\mathcal{I}_\epsilon
  \times M)\) with \(v^{\hat{\rho}} = a \partial_s +b \widehat{X} \in
  \hat{\rho}\) and \(v^\xi\in \xi\).
Then
\begin{align*}                                                            %% EQN
  |\Omega(v^{\hat{\rho}}, J v^\xi)| 
  &= 
  |\Omega(a \partial_s +b \widehat{X}, J v^\xi)| 
  \\ 
  &= 
  \Big|\Omega(a \partial_s, J v^\xi)+b\frac{\Omega(X_H, J
  v^\xi)}{\hat{\lambda}(X_H)}\Big|
  \\ 
  &= 
  \Big|\Omega(a \partial_s, J v^\xi)-b\, \frac{d H (J
  v^\xi)}{\hat{\lambda}(X_H)}\Big|
  \\ 
  &= 
  \Big|\Omega(a \partial_s, J v^\xi)-b\, \frac{d s (J
  v^\xi)}{\hat{\lambda}(X_H)}\Big|
  \\ 
  &= 
 \Big |\Omega(a\partial_s, J v^\xi)\Big|  & & \text{by equation (\ref{EQ_xi}),
  and } J:\xi\to \xi
  \\ 
  &\leq 
  \delta \|a\partial_s\|_{g_0} \| J v^\xi \|_{g_0} & & \text{by equation 
  (\ref{EQ_Omega_delta})} %{LEM_cross_term_control}
  \\ 
  &\leq 
  \delta C_1\|a\partial_s\|_{g_0} \|v^\xi\|_{g_0} & & \text{by equation
  (\ref{EQ_Jg0_g0})}  %{EQ_J_in_g0}
  \\ 
  &\leq 
  \delta C_1^2\|a\partial_s\|_{g_0} \|v^\xi\|_{g_J} & & \text{by equation
  (\ref{EQ_g0_gJ})} %{EQ_gJ_vs_g0}
  \\
  &\leq 
  \delta C_1^2 C_2 \|a \partial_s\|_{g_J} \|v^\xi\|_{g_J} &&
  \text{by equation (\ref{EQ_partial_s_estimate})}  %
  \\
  &=
  {\textstyle \frac{1}{8}} \|a \partial_s\|_{g_J} \|v^\xi\|_{g_J} &&
  \text{by equation (\ref{EQ_partial_s_norm_0})}
  \\
  &\leq 
  {\textstyle \frac{1}{8}} \big(\|a \partial_s\|_{g_J}^2 +
  \|v^\xi\|_{g_J}^2\big)
  \\
  &\leq 
  {\textstyle \frac{1}{8}} \big(\|a \partial_s\|_{g_J}^2 + \|b
  \widehat{X}\|_{g_J}^2
  + \|v^\xi\|_{g_J}^2\big)
  \\
  &\leq 
  {\textstyle \frac{1}{8}} \big(\|v^{\hat{\rho}}\|_{g_J}^2 +
  \|v^\xi\|_{g_J}^2\big)
  \\
  &=
  {\textstyle \frac{1}{8}}\|v\|_{g_J}^2
  \\
  &=
  {\textstyle \frac{1}{8}}(ds\wedge \hat{\lambda} +\hat{\omega})(v, Jv).
  \end{align*}
The other cross term is estimated rather similarly.  
Again let \(v = v^{\hat{\rho}}+ v^\xi \in T(\mathcal{I}_\epsilon \times
  M)\) with \(v^{\hat{\rho}} = a \partial_s +b \widehat{X} \in
  \hat{\rho}\) and let \(v^\xi\in \xi\).
We estimate:
\begin{align*}                                                            %% EQN
  |\Omega(v^\xi, Jv^{\hat{\rho}})|
  &=
  \big|\Omega\big(v^\xi, J(a\partial_s +b \widehat{X})\big)\big|
  \\
  &=
  \big|\Omega(v^\xi, {\textstyle \frac{a}{\phi}}\widehat{X} - b\phi
  \partial_s)\big|
  \\ 
  &=
  \big|\Omega(v^\xi, b\phi \partial_s)\big|
  \\ 
  &\leq 
  \delta \|v^\xi\|_{g_0} \|b\phi \partial_s\|_{g_0}
  \\ 
  &\leq 
  \delta C_1 C_2 \phi \|v^\xi\|_{g_J} \|b \partial_s\|_{g_J}
  \\ 
  &\leq 
  {\textstyle \frac{1}{8}}  \|v^\xi\|_{g_J} \|b \partial_s\|_{g_J}
  \\ 
  &\leq 
  {\textstyle \frac{1}{8}}  \big( \|v^\xi\|_{g_J}^2+  \|b
  \partial_s\|_{g_J}^2)
  \\
  &\leq 
  {\textstyle \frac{1}{8}}  \big( \|v^\xi\|_{g_J}^2+
  \|v^{\hat{\rho}}\|_{g_J}^2)
   \\ 
  &\leq 
  {\textstyle \frac{1}{8}} \|v\|_{g_J}^2  
   \\ 
  &=
  {\textstyle \frac{1}{8}} (ds\wedge \hat{\lambda} +\hat{\omega})(v, Jv) 
 \end{align*}
With these two estimates established, we can now use them to establish the
  following.

\begin{align*}                                                            %% EQN
  \big|\Omega(v, Jv) &- (ds \wedge \hat{\lambda}
  +\hat{\omega})(v,Jv)\big|\\
  &=
  \big|\Omega\big(v^{\hat{\rho}} + v^\xi, J(v^{\hat{\rho}} + v^\xi)\big) -
  (ds \wedge \hat{\lambda} +\hat{\omega})\big(v^{\hat{\rho}} + v^\xi,
  J(v^{\hat{\rho}} + v^\xi)\big)\big|
  \\
  &=
  \big|\Omega\big(v^{\hat{\rho}} + v^\xi, J(v^{\hat{\rho}} + v^\xi)\big) -
  (ds \wedge \hat{\lambda}) (v^{\hat{\rho}}, Jv^{\hat{\rho}})
  -\hat{\omega} (v^\xi, J v^\xi)\big|
  \\
  &\leq
  |\Omega(v^{\hat{\rho}}, Jv^{\hat{\rho}}) - ds\wedge
  \hat{\lambda}(v^{\hat{\rho}}, J v^{\hat{\rho}})|
  +
  |\Omega(v^{\xi}, Jv^{\xi}) - \hat{\omega}(v^{\xi}, J v^{\xi})| 
  \\
  &\qquad
  +\big|\Omega(v^\xi, Jv^{\hat{\rho}}) \big|
  +
  \big|\Omega(v^{\hat{\rho}}, J v^\xi) \big|
  \\
  &\leq
  |\Omega(v^{\hat{\rho}}, Jv^{\hat{\rho}}) - (ds\wedge
  \hat{\lambda})(v^{\hat{\rho}}, J v^{\hat{\rho}})|
  +\big|\Omega(v^\xi, Jv^{\hat{\rho}}) \big|
  +
  \big|\Omega(v^{\hat{\rho}}, J v^\xi) \big|
  \\
  &\leq
  |\Omega(v^{\hat{\rho}}, Jv^{\hat{\rho}}) - (ds\wedge
  \hat{\lambda})(v^{\hat{\rho}}, J v^{\hat{\rho}})|
  + {\textstyle \frac{1}{4}} (ds\wedge \hat{\lambda} +\hat{\omega})(v, Jv) 
  \\
  &=
  \Big|\frac{1}{\hat{\lambda}(X_H)}(ds\wedge \hat{\lambda})(v^{\hat{\rho}},
  J v^{\hat{\rho}}) - (ds\wedge \hat{\lambda})(v^{\hat{\rho}}, J
  v^{\hat{\rho}})\Big|
  + {\textstyle \frac{1}{4}}  (ds\wedge \hat{\lambda} +\hat{\omega})(v,
  Jv)
  \\
  &\leq
  \Big|\hat{\lambda}(X_H)^{-1}-1 \Big|  (ds\wedge
  \hat{\lambda}+\hat{\omega})(v, J v)
  + {\textstyle \frac{1}{4}}  (ds\wedge \hat{\lambda} +\hat{\omega})(v, Jv) 
  \\
  &\leq
  {\textstyle \frac{1}{2}} (ds\wedge \hat{\lambda} +\hat{\omega})(v, Jv) 
  \end{align*}
In other words, we have shown that 
  \begin{align*}                                                          %% EQN
    \big|\Omega(v, Jv) - (ds \wedge \hat{\lambda} +\hat{\omega})(v,
    Jv)\big| &\leq {\textstyle \frac{1}{2}} (ds\wedge \hat{\lambda}
    +\hat{\omega})(v, Jv)
    \end{align*}
  and thus
  \begin{align*}                                                          %% EQN
    (ds\wedge \hat{\lambda} +\hat{\omega})(v, Jv) \leq 2 \Omega(v, Jv)
    \end{align*}
  for all \(v\in T(\mathcal{I}_\epsilon\times M)\).
\end{proof}%%%%%%%%%%%%%%%%%%%%%%%%%%%%%%%%%%%%%%%%%%%%%%%%            END PROOF

With the above estimates established, we can now turn our attention to the
  main result of this section. 
Note that the metric \(\tilde{g}\) occurring below is a translation
  invariant metric on ${\mathbb R}\times M$ associated to a metric on \(M\).  
The norms $\|{\cdot}\|_{\mathcal{C}^n}$ use this metric.

\renewcommand\thesection{\arabic{section}}
\setcounter{CurrentSection}{\value{section}}
\setcounter{CurrentProposition}{\value{lemma}}
\setcounter{lemma}{\value{CounterPropositionAdiabatic}}
\setcounter{section}{\value{CounterSectionAdiabatic}}
\begin{proposition}[adiabatically degenerating constructions]
%  \label{PROP_adiabatically_degenerating_constructions}
  \hfill\\
Let \(\big(\mathcal{I}_{\ell}\times M, \Omega, \hat{\lambda}\big)\) be a
  framed Hamiltonian energy pile, and let \(\hat{\rho}\), \(\xi\), and
  \(\hat{\omega}\) be the associated structures defined in Section
  \ref{SEC_background}, and assume, as in the conclusions of Lemma
  \ref{LEM_localization_to_framed_Hamiltonian_energy_pile}, that along
  \(\{0\}\times M\) the sub-bundles \(\hat{\rho}\) and \(\xi\) are
  symplectic complements and \(\hat{\lambda}(X_H)=1\). 
Let \(\widehat{X}=X_H/\hat{\lambda}(X_H)\) as above and denote by
\(\|\cdot \|_{\mathcal{C}^n}\) the
  \(\mathcal{C}^n\)-norm on \(\mathbb{R}\times M\) with respect to the
  auxiliary translation invariant metric \(\tilde{g}\), and the \(\Psi_k\)
  are the embeddings as in equation (\ref{EQ_Psi_new}).
Let \(\{\check{J}_k\}_{k\in \mathbb{N}}\) be a sequence of almost
  complex structures on \(\mathcal{I}_\ell \times M\) which satisfy the
  following conditions.
\begin{enumerate}[(D1)]                                                   %% NUM
  \item 
  \(\check{J}_k\colon \hat{\rho} \to \hat{\rho}\; \) and  \(\;
  \check{J}_k\colon \xi \to \xi\) for each \(k\in \mathbb{N}\),
  \item 
  \((ds\wedge \hat{\lambda} +\hat{\omega})\circ ({\rm Id}\times
  \check{J}_k)\) is a Riemannian metric for each \(k\in \mathbb{N}\),
  \item 
  there exist constants  \(\{C_n'\}_{n\in \mathbb{N}}\) such that
  \begin{align*}                                                          %% EQN
    \sup_{k\in \mathbb{N}} \|\check{J}_k\|_{\mathcal{C}^n} \leq C_n'
    \end{align*}
  \end{enumerate}
Then there exists an \(\epsilon>0\) with the following significance.
For any sequence of functions \(\phi_k : (-\epsilon, \epsilon)\to (0, 1]\)
  which converge in \(\mathcal{C}^\infty\), the weakly adapted almost
  complex structures defined by
  \begin{align}\label{EQ_def_Jk}                                          %% EQN
    \phi_k \cdot J_k \partial_s = \widehat{X}\qquad\text{and}\qquad
    J_k\big|_\xi := \check{J}_k\big|_\xi
    \end{align}
  satisfy the following properties
\begin{enumerate}[(E1)]                                                   %% NUM
  \item For each \(v\in T(\mathcal{I}_{\epsilon}\times M)\) we have
  \begin{align*}                                                          %% EQN
    (ds \wedge \hat{\lambda} + \hat{\omega})(v, J_kv )\leq 2 \Omega(v ,
    J_k v).
    \end{align*}
  \item 
  There exists a sequence \(\{K_n\}_{n=0}^\infty\) of positive
  constants so that with respect to the auxiliary translation invariant
  metric \(\tilde{g}\) on
  \(\mathbb{R}\times M\) it holds that
  \begin{align*}                                                          %% EQN
    \sup_{k\in \mathbb{N}} \| (\Psi_k)_* J_k \|_{\mathcal{C}^n} \leq K_n
    \end{align*}
  for each \(n\in \mathbb{N}\). 
  \end{enumerate}
In particular, if \(\epsilon' \in (0, \epsilon)\) and
  \(\phi_k\big|_{(-\epsilon', \epsilon')}\to 0\), then the almost complex
  structures \(J_k\) are adiabatically degenerating on
  \(\mathcal{I}_{\epsilon'}\times M\) in the sense of Definition
  \ref{DEF_adiabatically_degenerating_almost_complex_structures}, and
  on \(\mathcal{I}_{\epsilon}\times M\) these almost complex structures are
  all tame.
\end{proposition}
%%%%%                                                                       %%%%
%%%%%%%%%%                                                             %%%%%%%%%
%%%%%%%%%%%%%%%%%%%%%%%%%%%%%%%%%%%%%%%%%%%%%%%%%%%%%%%%%%%%%%%%%%%%%%%%%%%%%%%%
%
\begin{proof}
\setcounter{section}{\value{CurrentSection}}
\setcounter{lemma}{\value{CurrentProposition}}
\renewcommand\thesection{\Alph{section}}

Fix an auxiliary \(\Omega\)-compatible almost complex structure \(J_0\) on
  \(\mathcal{I}_\ell\times M\), and let \(g_0\) be the associated
  background Riemannian
  metric \(g_0=\Omega\circ ({\rm Id}\times J_0)\).
Pick an \(\ell'\in (0,\ell) \) and consider the sequence \((\check{J}_k)\).  
We define 
  \begin{align*}                                                          %% EQN
    C_1 := \sup_{k\in \mathbb{N}} |\rangle \check{J}_k\langle
    |_{(\mathcal{I}_{\ell'} \times M, \Omega, \hat{\lambda}, \hat{\omega},   
    J_0)}
    \end{align*}
  where \(|\rangle J \langle|\) is defined as in equation
  (\ref{EQ_contrived}).
We note that \(C_1 \) is finite because of hypothesis (D3).  
We then apply Proposition \ref{PROP_metric_symplectic} which guarantees
  the existence of an \(\epsilon>0\) for which (E1) holds for any weakly
  adapted almost complex structures \(J_k\) for which \(J_k\big|_\xi =
  \check{J}_k\big|_\xi\).

To establish (E2), assume that the \(\phi_k\) have been fixed so they
  converge in \(\mathcal{C}^\infty\).
Define the \(J_k\) as in equation (\ref{EQ_def_Jk}).
To estimate the norms of the \((\Psi_k)_* J_k\) we first note that it is
  sufficient to establish \(\mathcal{C}^n\) bounds for
  \begin{align*}                                                          %% EQN
    ((\Psi_k)_* J_k)\partial_{\check{s}} \qquad\text{and}\qquad
    ((\Psi_k)_* J_k)\big|_\xi,
    \end{align*}
  however each of these bounds are immediately obtained as a consequence
  of the fact that the \(\phi_k\) are converging in \(\mathcal{C}^\infty\)
  together with hypothesis (D3).
The remaining desired conclusions are immediate.
This completes the proof of Proposition
  \ref{PROP_adiabatically_degenerating_constructions}. 

\end{proof}%%%%%%%%%%%%%%%%%%%%%%%%%%%%%%%%%%%%%%%%%%%%%%%%            END PROOF

\renewcommand{\thesection}{\Alph{section}}

\subsection{Extra Definitions}\label{SEC_extra_definitions}

For the convenience of the reader, we provide a few definitions here which
  are used indirectly. The reader should consult \cite{FH2}, which contains further discussions of this concept.
  
\begin{definition}[realized Hamiltonian homotopy]
  \label{DEF_hamiltonian_homotopy}
  \hfill\\
Let \(M\) be a smooth (odd-dimensional) closed manifold, let
  \(\mathcal{I}\subset \mathbb{R}\) be an interval equipped with the
  coordinate \(t\), and let \(\hat{\lambda}\) and \(\hat{\omega}\)
  respectively be a one-form and two-form on \(\mathcal{I}\times M\).
We say \((\mathcal{I}\times M, (\hat{\lambda}, \hat{\omega}))\) is a
  realized Hamiltonian homotopy provided the following hold.
\begin{enumerate}                                                        
  \item \(\hat{\lambda}(\partial_t)=0 \).
  \item \(i_{\partial_t} \hat{\omega}=0\).
  \item \(d \hat{\omega}\big|_{\{t={\rm const}\}} = 0\)
  \item \(dt\wedge \hat{\lambda}\wedge \hat{\omega}\wedge \cdots \wedge
    \hat{\omega} >0\).
  \item \(\hat{\lambda}\) is invariant under the flow of \(\partial_t\)
  \item if \(\mathcal{I}\) is unbounded, then there exists a neighborhood
    of \(\{\pm\infty\}\times M\) on  which
    \(\hat{\omega}\) is invariant under the flow of \(\partial_t\).
  \end{enumerate}
\end{definition}
%%%%%                                                                       %%%%
%%%%%%%%%%                                                             %%%%%%%%%
%%%%%%%%%%%%%%%%%%%%%%%%%%%%%%%%%%%%%%%%%%%%%%%%%%%%%%%%%%%%%%%%%%%%%%%%%%%%%%%%
%

%%%%%%%%%%%%%%%%%%%%%%%%%%%%%%%%%%%%%%%%%%%%%%%%%%%%%%%%%%%%%%%%%%%%%%%%%%%%%%%%
%%%%%%%%%%                          DEFINITION                         %%%%%%%%%
%%%%%                                                                       %%%%
\begin{definition}[adapted structures for a realized Hamiltonian
  homotopy]
  \label{DEF_adapted_structures_Ham_homotopy}
  \hfill\\
Let \((\mathcal{I}\times M , (\hat{\lambda}, \hat{\omega}))\) be a
  realized Hamiltonian homotopy.
We say an almost Hermitian structure \((\widehat{J}, \hat{g})\) on
  \(\mathcal{I}\times M\) is adapted to this realized Hamiltonian homotopy
  provided the following hold.
\begin{enumerate}                                                         %% NUM
  \item \(\widehat{J}\partial_t = \widehat{X}\).
  \item \(\widehat{J}\colon \hat{\xi}\to \hat{\xi}\).
  \item \(\hat{g} = (dt \wedge \hat{\lambda} + \hat{\omega})(\cdot,
    \widehat{J} \cdot) \).
  \item if \(\mathcal{I}\) is unbounded, then there exists a neighborhood
    of \(\{\pm\infty\}\times M\) on which the restriction
    \(\widehat{J}\big|_{\hat{\xi}}\) is invariant under
    the flow of \(\partial_t\).
  \end{enumerate}
\end{definition}
%%%%%                                                                       %%%%
%%%%%%%%%%                                                             %%%%%%%%%
%%%%%%%%%%%%%%%%%%%%%%%%%%%%%%%%%%%%%%%%%%%%%%%%%%%%%%%%%%%%%%%%%%%%%%%%%%%%%%%%
%

\bibliographystyle{plain}
\bibliography{bibliography}{}

\end{document}